\documentclass[reqno]{amsart}
\usepackage{amsmath, amssymb, amsthm, epsfig}
\usepackage{hyperref, latexsym}
\usepackage{url}
\usepackage[mathscr]{euscript}
\usepackage{mathabx}

\usepackage{color}
\usepackage{fullpage} 
\usepackage{setspace}

\def\today{\ifcase\month\or
  January\or February\or March\or April\or May\or June\or
  July\or August\or September\or October\or November\or December\fi
  \space\number\day, \number\year}
\DeclareMathOperator{\sgn}{\mathrm{sgn}}

 \newtheorem{theorem}{Theorem}
  
 \newtheorem{lemma}[theorem]{Lemma}
 
 \newtheorem{corollary}[theorem]{Corollary}
 \theoremstyle{definition}

 \theoremstyle{remark}

 \newcommand{\C}{\mathbb{C}}
 \newcommand{\R}{\mathbb{R}}

 \newcommand{\Z}{\mathbb{Z}}

 \newcommand{\hh}{\tfrac12}

 \newcommand{\dt}{\text{\rm d}t}
  \renewcommand{\d}{\text{\rm d}}
 \newcommand{\du}{\text{\rm d}u}

 \newcommand{\dw}{\text{\rm d}w}
 \newcommand{\dx}{\text{\rm d}x}
 
 \newcommand{\dy}{\text{\rm d}y}

\newcommand{\im}{{\rm Im}\,}
\newcommand{\re}{{\rm Re}\,}

\begin{document}
\title[Estimates for the Riemann zeta-function]{Bandlimited approximations and estimates \\ for the Riemann zeta-function}
\author[Carneiro, Chirre and Milinovich]{Emanuel Carneiro, Andr\'{e}s Chirre and Micah B. Milinovich}
\address{IMPA - Instituto Nacional de Matem\'{a}tica Pura e Aplicada - Estrada Dona Castorina, 110, Rio de Janeiro, RJ, Brazil 22460-320}
\email{carneiro@impa.br}
\email{achirre@impa.br}
\address{Department of Mathematics, University of Mississippi, University, MS 38677 USA}
\email{mbmilino@olemiss.edu}
\allowdisplaybreaks
\numberwithin{equation}{section}

\maketitle  

\begin{abstract}
In this paper, we provide explicit upper and lower bounds for the argument of the Riemann zeta-function and its antiderivatives in the critical strip under the assumption of the Riemann hypothesis. This extends the previously known bounds for these quantities on the critical line (and sharpens the error terms in such estimates). Our tools come not only from number theory, but also from Fourier analysis and approximation theory. An important element in our strategy is the ability to solve a Fourier optimization problem with constraints, namely, the problem of majorizing certain real-valued even functions by bandlimited functions, optimizing the $L^1(\R)-$error. Deriving explicit formulae for the Fourier transforms of such optimal approximations plays a crucial role in our approach.
\end{abstract}


\section{Introduction}

In this paper, we make use of fine tools from harmonic analysis and approximation theory to prove a number of new estimates in the theory of the Riemann zeta-function.

\subsection{Definitions} Let $\zeta(s)$ denote the Riemann zeta-function. For $\hh\leq\alpha \leq1$ we define
$$S_{0,\alpha}(t) = \tfrac{1}{\pi} \arg \zeta \big(\alpha + it \big),$$
where the argument is obtained by a continuous variation along straight line segments joining the points $2$, $2+i t$ and $\alpha + it$, assuming that this path has no zeros of $\zeta(s)$, with the convention that $\arg \zeta(2) = 0$. If this path has zeros of $\zeta(s)$ (including the endpoint $\alpha + it$) we set 
$$S_{0,\alpha}(t)  = \tfrac{1}{2}\, \lim_{\varepsilon \to 0^+} \left\{ S_{0,\alpha}(t + \varepsilon) + S_{0,\alpha}(t - \varepsilon)\right\}. $$
We now define a sequence of antiderivatives $S_{n,\alpha}(t)$ that encode useful information on the oscillatory character of $S_{0,\alpha}(t)$. For $n\geq 1$ and $t >0$ we define, inductively, the functions
\begin{equation}\label{Intro_eq1_int_S_n}
S_{n,\alpha}(t) = \int_0^t S_{n-1,\alpha}(\tau) \,\d\tau\, + \delta_{n,\alpha\,},
\end{equation}
where $\delta_{n,\alpha}$ is a specific constant depending on $\alpha$ and $n$. For $k\in \mathbb{N}$, these constants are given by 
$$\delta_{2k-1,\alpha} =\frac{ (-1)^{k-1}}{\pi} \int_{\alpha}^{\infty} \int_{\sigma_{2k-2}}^{\infty} \ldots \int_{\sigma_{2}}^{\infty} \int_{\sigma_{1}}^{\infty} \log |\zeta(\sigma_0)|\, \d\sigma_0\,\d\sigma_1\,\ldots \,\d \sigma_{2k-2} $$
for $n = 2k-1$ and by
$$\delta_{2k,\alpha} = (-1)^{k-1} \int_{\alpha}^{1} \int_{\sigma_{2k-1}}^{1} \ldots \int_{\sigma_{2}}^{1} \int_{\sigma_{1}}^{1} \d\sigma_0\,\d\sigma_1\,\ldots \,\d \sigma_{2k-1} = \frac{(-1)^{k-1}(1-\alpha)^{2k}}{(2k)!}  $$ 
for $n = 2k$. In Theorem \ref{Thm2}, we give precise upper and lower bounds for all the iterates $S_{n,\alpha}(t)$ for $\alpha$ in the critical strip.

\subsection{Behavior on the critical line} In the case $\alpha = \hh$, let us simply write $S_n(t) = S_{n,\frac12}(t)$ to return to classical notation (e.g. Littlewood \cite{L} and Selberg \cite{S}). The argument function $S(t) = S_0(t)$ is intrinsically connected to the distribution of the non-trivial zeros of $\zeta(s)$ via the relation 
$$N(t) = \frac{t}{2\pi} \log \frac{t}{2\pi} - \frac{t}{2\pi} + \frac{7}{8}  + S(t) + O\left( \frac{1}{t}\right),$$
where $N(t)$ counts (with multiplicity) the number of zeros $\rho = \beta + i \gamma$ of $\zeta(s)$ such that $0 < \gamma \leq t$, where zeros with ordinate $\gamma = t$ are counted with weight $\hh$. 

\smallskip

The Riemann hypothesis (RH) states that the non-trivial zeros of $\zeta(s)$ can be written as $\rho = \hh +i\gamma$ with $\gamma\in \mathbb{R}$. In his classical work \cite[Theorem 11]{L} of 1924, J. E. Littlewood established, assuming RH, the bound\footnote{Throughout the paper we use Vinogradov's notation $f = O(g)$ (or $f \ll g$) to mean that $|f(t)| \leq C \,|g(t)|$ for a certain constant $C>0$. In the subscript we indicate the parameters in which such constant $C$ may depend on.}
\begin{equation}\label{Littlewood_bound}
S_n(t) = O_n \left( \frac{\log t}{(\log \log t)^{n+1}}\right)
\end{equation}
for $n \geq 0$. 
The order of magnitude of \eqref{Littlewood_bound} has never been improved, and the efforts have been concentrated in optimizing the value of the implicit constant. The state of the art in this problem is the following result of \cite{CCM, CChi}.

\begin{theorem}[cf. \cite{CCM, CChi}] \label{Thm1}Assume the Riemann hypothesis. For $n \geq 0$ and $t$ sufficiently large we have
\begin{equation}\label{Intro_Thm_1_eq}
-\left( C^-_n + o(1)\right) \frac{\log t}{(\log \log t)^{n+1}} \ \leq \ S_n(t) \ \leq \ \left( C^+_n + o(1)\right) \frac{\log t}{(\log \log t)^{n+1}}\,,
\end{equation}
where $C_n^{\pm}$ are positive constants given by
\begin{itemize}
\item {\it For $n=0$},
$$C_0^{\pm} = \frac{1}{4}.$$

\item For $n = 4k +1$, with $k \in \Z_{\ge0}$,
$$C_{n}^- = \frac{\zeta(n+1)}{\pi \cdot 2^{n+1}} \ \ \ \ {\rm and} \ \ \ \ C_{n}^+ = \frac{\left( 1 - 2^{-n}\right)\zeta(n+1)}{\pi \cdot 2^{n+1}}.$$

\item For $n = 4k +3$, with $k \in \Z_{\ge0}$,
$$C_{n}^- = \frac{\left( 1 - 2^{-n}\right)\zeta(n+1)}{\pi \cdot 2^{n+1}} \ \ \ \ {\rm and} \ \ \ \ C_{n}^+ = \frac{\zeta(n+1)}{\pi \cdot 2^{n+1}}.$$

\item For $n \geq 2$ even,
\begin{align*}
\begin{split}
C_n^+ = C_n^- & =  \left(\frac{2 \big(C_{n+1}^{+} + C_{n+1}^{-}\big) \ C_{n-1}^{+}\ C_{n-1}^{-}}{C_{n-1}^{+} + C_{n-1}^{-}}\right)^{1/2} \\
& = \frac{\sqrt{2}}{\pi \cdot 2^{n+1}}\left(\frac{\left(1 - 2^{-n-2}\right)\,\left( 1 - 2^{-n+1}\right)\,\zeta(n) \ \zeta(n+2)}{\left(1 - 2^{-n}\right)}\right)^{1/2}.
\end{split}
\end{align*}
\end{itemize}
The terms $o(1)$ in \eqref{Intro_Thm_1_eq} are $O_n(\log \log \log t / \log \log t)$.
\end{theorem}

One consequence of Theorem \ref{Thm2}, below, is a slight sharpening of this result. The cases $n=0$ and $n=1$ in Theorem \ref{Thm1} were proved by Carneiro, Chandee and Milinovich \cite{CCM} (see \cite{CCM2} for an alternative proof for $n=0$), while the cases $n\geq 2$ were proved by Carneiro and Chirre \cite{CChi}. In the case $n=0$, this improved upon earlier works of Goldston and Gonek \cite{GG}, Fujii \cite{F2} and Ramachandra and Sankaranarayanan \cite{RS}, who had obtained \eqref{Intro_Thm_1_eq} with constants $C^{\pm}=1/2$, $C^{\pm}=0.67$ and $C^{\pm}=1.12$, respectively, replacing the constant $C_0^{\pm} = 1/4$. In the case $n=1$, this improved upon earlier works of Fujii \cite{F3}, and Karatsuba and Korol\"{e}v \cite{KK}, who had obtained \eqref{Intro_Thm_1_eq} with the pairs of constants $(C^+, C^-) = (0.32,0.51)$ and $(C^+, C^-) = (40,40)$, respectively, replacing the pair $(C_1^+, C_1^-) = (\pi/48, \pi/24)$. In the cases $n\geq 2$, this significantly improved upon the work of Wakasa \cite{W}, who had established \eqref{Intro_Thm_1_eq} with pairs of constants that depended on $n$ but tended to the stationary value of $0.3203696...$ as $n \to \infty$, whereas the constants $C_n^{\pm}$ above go to zero exponentially fast.

\smallskip

Unconditionally, it is known that $S(t) = O(\log t)$, $S_1(t) = O(\log t)$ and $S_n(t) = O_n\big(t^{n-1}/\log t\big)$ for $n\geq 2$ (see, for instance, \cite{F1} for the latter). In fact, the Riemann hypothesis is equivalent to the statement that $S_n(t) = o(t^{n-2})$ as $t \to \infty$, for any $n \geq 3$ (see \cite[Theorem 4]{F1}).

\subsection{Behavior in the critical strip} The main purpose of this paper is to extend the bounds of Theorem \ref{Thm1} to the critical strip in an explicit way. Assuming RH, for $\hh < \alpha <1$, another function that will play an important role in our study is the derivative
$$S_{-1,\alpha}(t):= S_{0,\alpha}'(t) = \frac{1}{\pi}\, \re \frac{\zeta'}{\zeta}(\alpha + it).$$
For an integer $n \geq 0$ we introduce the function 
$$H_n(x) :=  \sum_{k=0}^{\infty} \frac{x^k}{(k+1)^n}.$$
The function $xH_n(x) = {\rm Li}_n(x)$ is known as {\it polylogarithm of order $n$} in the classical terminology of special functions. Note that $H_0(x) = 1/(1-x)$ for $|x| <1$. Our main result is stated below, in which we regard $\alpha$ and $t$ as free parameters.

\begin{theorem}\label{Thm2}
Assume the Riemann hypothesis and let $n \geq -1$. Let $\hh < \alpha < 1$ and $c>0$ be a given real number. Let $t > 0$ be such that $\log \log t \geq 4$. In the range
\begin{equation}\label{Range_Ess}
(1-\alpha)^2\log \log t \geq c
\end{equation}
we have the uniform bounds:
\smallskip

\noindent{\rm (i)} For $n=-1$,
\begin{align}\label{Fev02_3:09pm}
\begin{split}
 -C^-_{-1, \alpha}(t) \, (\log t)^{2-2\alpha} & + O_c\left(\dfrac{(\alpha - \frac12)(\log t)^{2 - 2 \alpha}}{(1 - \alpha)^2 \log \log t}\right) \leq   S_{-1,\alpha}(t)  = \frac{1}{\pi}\, \re \frac{\zeta'}{\zeta}(\alpha + it)  \\
&  \ \ \   \leq 
C^+_{-1,\alpha}(t) \, (\log t)^{2-2\alpha} + O_c\left(\dfrac{(\log t)^{2 - 2 \alpha}}{(\alpha - \frac12)\,(1 - \alpha)^2\, (\log \log t)}\right).
\end{split}
\end{align}
\noindent{\rm (ii)} For $n\geq 0$,
\begin{align}\label{Feb23_4:32pm}
\begin{split}
 -C^-_{n,\alpha}(t)  \frac{(\log t)^{2-2\alpha}}{(\log \log t)^{n+1}} & + O_{n,c}\left(\frac{ (\log t)^{2 - 2\alpha}}{(1-\alpha)^2\,(\log \log t)^{n+2}}\right)  \leq   S_{n,\alpha}(t)  \\
&  \ \ \ \ \   \leq 
C^+_{n,\alpha}(t)  \frac{(\log t)^{2-2\alpha}}{(\log \log t)^{n+1}} + O_{n,c}\left(\frac{ (\log t)^{2 - 2\alpha}}{(1-\alpha)^2\,(\log \log t)^{n+2}}\right).
\end{split}
\end{align}
Above, $C_n^{\pm}(\alpha,t)$ are positive functions given by:
\begin{itemize}
\item For $n \geq -1$ odd, 
\begin{equation}\label{Jan31_5:17pm}
C_{n,\alpha}^{\pm}(t) = \frac{1}{2^{n+1}\pi}\left( H_{n+1}\Big(\pm (-1)^{(n+1)/2}\,(\log t)^{1 - 2\alpha}\Big) + \frac{2\alpha -1}{\alpha (1- \alpha)}\right).
\end{equation}
\item For $n =0$, 
\begin{equation}\label{Jan31_5:15pm}
C_{0,\alpha}^{\pm}(t) =  \Big(2 \big(C_{1,\alpha}^{+}(t) + C_{1,\alpha}^{-}(t)\big) \, C_{-1,\alpha}^{-}(t)\Big)^{1/2}.
\end{equation}
\item For $n \geq 2$ even,
\begin{equation}\label{Jan31_5:16pm}
C_{n,\alpha}^{\pm}(t) = \left(\frac{2 \big(C_{n+1,\alpha}^{+}(t) + C_{n+1,\alpha}^{-}(t)\big) \, C_{n-1,\alpha}^{+}(t)\, C_{n-1,\alpha}^{-}(t)}{C_{n-1,\alpha}^{+}(t) + C_{n-1,\alpha}^{-}(t)}\right)^{1/2}.
\end{equation}
\end{itemize}
\end{theorem}

\noindent \textsc{Remark:} In the course of the proof of Theorem \ref{Thm2} we obtain slightly stronger bounds than the ones presented in \eqref{Fev02_3:09pm} (see inequalities \eqref{Feb02_2:30pm} and \eqref{New_1017} below). In the statement of Theorem \ref{Thm2} we presented the error terms in \eqref{Fev02_3:09pm} and \eqref{Feb23_4:32pm} in a convenient way for our interpolation argument in Section \ref{Sec_Int}.

\smallskip

Observe that letting $\alpha \to \frac12^+$ in our Theorem \ref{Thm2} (for $n \geq 0$), we obtain a sharpened version of Theorem \ref{Thm1} with improved error terms (a factor $\log \log \log t$ has been removed). In particular, we record here the following consequence, a new proof of the best known bound for $S(t)$ under RH (in fact, with a sharpened error term when compared to \cite{CCM} and \cite{CCM2})\footnote{For an explanation of why all these methods lead to the same constant $1/4$ in the bound for $S(t)$, see \cite[Section 3]{CCM2}.}.

\begin{corollary}
Assume the Riemann hypothesis. For $t>0$ sufficiently large we have
\begin{equation*}
|S(t)| \leq \frac{1}{4} \, \frac{\log t }{\log \log t} + O\left(\frac{\log t }{(\log \log t)^2}\right).
\end{equation*}
\end{corollary}

Using the lower bound for the function $S_{-1,\alpha}(t)$ in Theorem \ref{Thm2}, we also deduce a new proof of the best known bound for $\log |\zeta(\tfrac{1}{2}\!+\!it)|$ under RH (see \cite{CS} and \cite{CC}).
\begin{corollary} \label{log zeta}
Assume the Riemann hypothesis. For $t>0$ sufficiently large we have
\begin{equation*}
\log \big|\zeta(\tfrac{1}{2}\!+\!it)\big| \leq \frac{\log 2}{2} \frac{\log t }{\log \log t} + O\left(\frac{\log t }{(\log \log t)^2}\right).
\end{equation*}
\end{corollary}

We deduce Corollary \ref{log zeta} from Theorem \ref{Thm2} at the end of this subsection. Observe that the lower bound for $S_{-1,\alpha}(t)$ in Theorem \ref{Thm2} is stable under the limit $\alpha \to \frac12^+$, whereas the upper bound is not. This is somewhat expected since $S(t)$ has jump discontinuities at the ordinates of the non-trivial zeros of $\zeta(s)$. In our case such blow up comes from the fact that we use a bandlimited majorant for the Poisson kernel and, as $\alpha \to \frac12^+$, this Poisson kernel converges to a delta function. This lack of stability may be related to the existence of small gaps between ordinates of zeros of $\zeta(s)$. Something similar can be seen in the
work of Ki \cite{Ki} on the distribution of the zeros of $\zeta'(s)$.

\smallskip

In order to find bounds for $S_{0,\alpha}(t)$ which are stable under the limit $\alpha \to \frac12^+$ (and hence extend Theorem \ref{Thm1}), we modified a bit our interpolation method in \S\ref{Subsec_Poisson_Feb01_11:24am} to use both bounds for $S_{1,\alpha}(t)$ and {\it only the lower bound for $S_{-1,\alpha}(t)$}. If one is interested in bounds as $t \to \infty$ for a {\it fixed $\alpha$ with $\frac12 < \alpha < 1$}, our Theorem \ref{Thm2} yields the following corollary (the bounds below can be made uniform in $\delta>0$ if we consider $\frac12 + \delta \leq \alpha \leq 1 - \delta$.)
\begin{corollary}\label{Cor5}
Assume the Riemann hypothesis and let $n \geq -1$. Let $\hh < \alpha < 1$ be a fixed number. Then
$$|S_{n,\alpha}(t)| \leq \frac{\omega_n}{2^{n+1}\pi}\left(1 + \frac{2\alpha -1}{\alpha(1-\alpha)} + o(1) \right)  \frac{(\log t)^{2-2\alpha}}{(\log \log t)^{n+1}}\,,$$
as $t \to \infty$, where $\omega_n = 1$ if $n$ is odd and $\omega_n = \sqrt{2}$ if $n$ is even. 
\end{corollary}
\noindent This plainly follows from \eqref{Jan31_5:17pm} and \eqref{Jan31_5:16pm} for $n\neq 0$. For the case $n=0$ one would simply perform the full interpolation method as described in \S\ref{Subsec_not_Poisson_Feb01_11:31am} (using the upper and lower bounds for both $S_{1,\alpha}(t)$ and $S_{-1,\alpha}(t)$) to obtain the optimized constant as in \eqref{Jan31_5:16pm}.

\smallskip

\noindent{\sc Remark:} The extra factor $\sqrt{2}$ in Corollary \ref{Cor5} when $n$ is even comes from \eqref{Jan31_5:16pm} and it is due to our indirect interpolation argument. In principle, if one could directly solve the associated extremal Fourier analysis problem in the case of $n$ even, this could lead to a better bound than \eqref{Jan31_5:16pm}. We note, however, that this is a highly nontrivial problem in approximation theory. See the discussion in \S \ref{Gen_Str} below.

\smallskip

Finally, notice that we have purposely restricted our range to be strictly inside the critical strip, away from the line $\alpha =1$. With our methods it is also possible (by means of some additional technical work) to consider the case when the parameter $\alpha$ is close to $1$, obtaining bounds of the sort $S_{n,\alpha}(t) = O_n(1)$, for $n \geq 1$ (with explicit constants). We do not pursue such matters here, feeling that classical methods in the literature are more suitable to treat this range. In fact, bounds for $S_{n,1}(t)$, for $n \geq 1$, are easily obtainable directly from \eqref{Lem1_eq_1} and the use of Fubini's theorem with the series representation in the region $\{z\in \C; \,\re z >1\}$. These bounds would be equal to our bounds in the cases of $n$ odd, and better in the case of $n$ even, since we use an indirect approach, via interpolation, for these cases. In the particular case of $n=0$, the known bound $|S_{0,1}(t)| \leq \frac{1}{\pi} \log \log \log t + O(1)$ (see \cite[Corollary 13.16]{MV}) is not easily obtainable by our particular interpolation argument.

\begin{proof}[Proof of Corollary \ref{log zeta}] Assuming RH, it follows from \cite[Corollary 13.16]{MV} that
\[
 \log|\zeta(\sigma+it)| \le \log \frac{1}{1-\sigma} + O\!\left( \frac{(\log t)^{2-2\sigma}}{(1-\sigma)\log\log t} \right)
\]
uniformly for $1/2 + 1/\log\log t \le \sigma \le 1 - 1/\log\log t$ and $t\ge 3$. Therefore, letting $\delta = \delta(t) = \frac{1}{2}+\frac{\log\log\log t}{\log\log t}$, we have
\begin{equation*}
\begin{split}
\log|\zeta(\tfrac{1}{2}+it)| &= -\int_{1/2}^\delta \re \frac{\zeta'}{\zeta}(\sigma + it) \, \mathrm{d}\sigma + \log|\zeta(\delta+it)| 
\\
&= -\int_{1/2}^\delta \re \frac{\zeta'}{\zeta}(\sigma + it) \, \mathrm{d}\sigma + O\left(\frac{\log t }{(\log \log t)^2}\right).
\end{split}
\end{equation*}
Since the lower bound in \eqref{Fev02_3:09pm} implies that 
\[
-\re \frac{\zeta'}{\zeta}(\sigma + it) \le \frac{(\log t)^{2-2\sigma}}{1+(\log t)^{1-2\sigma}} + O\!\left( (\sigma\!-\!\tfrac{1}{2})(\log t)^{2-2\sigma} \right)
\]
uniformly for $1/2 < \sigma \le \delta$, we see that
\[
\log|\zeta(\tfrac{1}{2}+it)| \le \int_{1/2}^\delta \left\{ \frac{(\log t)^{2-2\sigma}}{1+(\log t)^{1-2\sigma}} + O\!\left( (\sigma\!-\!\tfrac{1}{2})(\log t)^{2-2\sigma} \right) \right\} \mathrm{d}\sigma  + O\left(\frac{\log t }{(\log \log t)^2}\right).
\]
The corollary now follows from the estimates
\[
\int_{1/2}^\delta \frac{(\log t)^{2-2\sigma}}{1+(\log t)^{1-2\sigma}} \, \mathrm{d}\sigma \ \le \ \int_{1/2}^1 \frac{(\log t)^{2-2\sigma}}{1+(\log t)^{1-2\sigma}} \, \mathrm{d}\sigma \ = \ \frac{\log 2}{2} \frac{\log t}{\log\log t} - \frac{\log t \log(1+1/\log t)}{2\log\log t}
\]
and
\[
\int_{1/2}^\delta (\sigma\!-\!\tfrac{1}{2})(\log t)^{2-2\sigma} \, \mathrm{d}\sigma \ll \frac{\log t }{(\log \log t)^2}.
\]
\end{proof}

\subsection{General Strategy} \label{Gen_Str} Our approach is partly motivated by the ideas of Goldston and Gonek \cite{GG} on the use of the Guinand-Weil explicit formula applied to special functions with compactly supported Fourier transforms. These ideas have later been used by Chandee and Soundararajan \cite{CS} and other authors \cite{CC, CCM, CCM2, CChi,CF,M} to bound several objects related to the Riemann zeta-function and other $L$-functions. It is worth mentioning that here we face severe additional technical challenges in order to fully develop this circle of ideas to reach our desired conclusion.

\smallskip

The strategy can be broadly divided into the following four main steps:

\subsubsection{Step 1: Representation lemma.} The first step is to identify certain functions of a real variable that are naturally connected with the objects to be bounded, in our case the functions $S_{n,\alpha}(t)$. For each $n\geq -1$ and $\hh \leq \alpha \leq 1$ we define the function $f_{n,\alpha}:\R \to \R$ in the following manner.

\smallskip

\begin{itemize}
\item If $n = 2m$, for $m \in \Z_{\ge0}$, we define
\begin{equation}\label{Def_f_2m}
f_{2m,\alpha}(x)=\int_{\alpha}^{3/2}(\sigma-\alpha)^{2m}{\Bigg(\dfrac{x}{(\sigma-\frac{1}{2})^2+x^2}-\dfrac{x}{1+x^2}\Bigg)}\,\d\sigma.
\end{equation}  
\item If $n = 2m+1$, for $m \in \Z_{\ge0}$, we define
\begin{equation}\label{Def_f_2m+1}
f_{2m+1,\alpha}(x)=\frac{1}{2}\int_{\alpha}^{3/2}{\left(\sigma-\alpha\right)^{2m}\log\left(\dfrac{1+x^2}{(\sigma-\hh)^2+x^2}\right)} \,\d\sigma.
\end{equation}
\item If $n = -1$, we define
\begin{equation}\label{Def_f_-1}
f_{-1,\alpha}(x)= \frac{(\alpha - \hh)}{(\alpha - \hh)^2 + x^2}.
\end{equation}
\end{itemize}
We prove a representation lemma (Lemma \ref{Rep_lem}) where we write $S_{n,\alpha}(t)$, for each $n \geq -1$, as a sum of a translate of the function $f_{n,\alpha}$ over the non-trivial zeros of $\zeta(s)$ plus some known terms and a small error.

\subsubsection{Step 2: Extremal functions} Our tool to evaluate sums over the non-trivial zeros of $\zeta(s)$ is the Guinand-Weil explicit formula. However, the functions $f_{n,\alpha}$ defined above do not possess the required smoothness to allow a direct evaluation. In fact, for $\alpha = \tfrac12$ and $n \geq 1$, we have that $f_{n,\frac12}$ is of class $C^{n-1}(\R)$ but not higher (the $n$-th derivative is discontinuous at the origin). Note also that $f_{0,\frac12}$ is discontinuous at the origin and $f_{-1,\frac12}$ is identically zero. For $\hh < \alpha$, the functions $f_{n,\alpha}$ are of class $C^{\infty}(\R)$ but do not have an analytic extension to the strip $\{z \in \C; \ -\hh - \varepsilon < \im z < \hh + \varepsilon\}$. In fact, the functions $f_{n,\alpha}$ are analytic in the strip  $\{z \in \C; \ -(\alpha - \hh) < \im z < (\alpha - \hh)\}$ but the $n$-th derivative of $f_{n,\alpha}$ cannot be extended continuously to the points $\pm(\alpha - \hh)i$, for $n \geq 0$ (for $n = -1$ the function $f_{-1,\alpha}$ has a pole at $\pm(\alpha - \hh)i$).

\smallskip

The idea is then to replace the functions $f_{n,\alpha}$ by suitable bandlimited approximations (real-valued majorants and minorants with compactly supported Fourier transforms) chosen in such a way to minimize the $L^1(\R)-$distance. This is known as the Beurling-Selberg extremal problem in harmonic analysis and approximation theory (the excellent survey of J. D.  Vaaler \cite{V} provides a nice introduction to the subject). In our case, the situation is markedly different depending upon whether $n$ is even or odd. When $n \geq -1$ is odd, the function $f_{n,\alpha}$ is {\it even}, and the robust Gaussian subordination framework of Carneiro, Littmann and Vaaler \cite{CLV} provides the required extremal functions. When $n$ is even, the function $f_{n,\alpha}$ is {\it odd and continuous} (except in the case $n=0$ and $\alpha = \hh$, which was considered in \cite{CCM}). In this general situation, the solution of the Beurling-Selberg extremal problem is unknown. Therefore, we adopt a  different approach based on an interpolation argument.

\subsubsection{Step 3: Guinand-Weil explicit formula and asymptotic analysis.} In the case of $n$ odd, $n\geq -1$, we bound $S_{n,\alpha}(t)$ by applying the Guinand-Weil explicit formula to the Beurling-Selberg majorants and optimizing the size of the support of the Fourier transform. This is possible via a careful asymptotic analysis of all the terms that appear in the explicit formula. In particular, we highlight that one of the main technical difficulties of this work, when compared to \cite{CCM,CChi}, is in the analysis of the sum over primes powers. This term is easily handled in the works \cite{CCM,CChi} when $\alpha = \hh$ but, in the case $\alpha > \hh$ that we treat here, we must perform a much deeper analysis, using the explicit knowledge of the Fourier transform of the majorant function. We collect in two appendices at the end of the paper some of the calculus facts and some of the number theory facts that are needed for this analysis. 

\subsubsection{Step 4: Interpolation tools.} Having obtained the desired bounds for all odd $n$'s, with $n \geq -1$, we proceed with an interpolation argument to obtain the estimate for the even $n$'s in between, exploring the smoothness of $S_{n,\alpha}(t)$ via the mean value theorem. An optimal choice of the parameters involved in the interpolation argument yields the desired bounds for the even $n$'s.

\section{The representation lemma}

In this section we collect some useful auxiliary results. Lemmas \ref{Lem2} and \ref{Rep_lem} below have appeared in \cite[Lemmas 2 and 3]{CChi} in the case $\alpha = \tfrac12$. The proofs for general $\tfrac12\leq\alpha \leq 1$ are essentially analogous. We include here brief versions of these proofs, both for completeness and for the convenience of the reader. Our starting point is the following formula motivated by the work of Selberg \cite[Section 2]{S}. 

\begin{lemma}\label{Lem2}
Assume the Riemann hypothesis. 
\begin{itemize}
\item[(i)] For $n \geq 0$, $\hh\leq\alpha\leq1$ and $t >0$ $($and $t$ not coinciding with the ordinate of a zero of $\zeta(s)$ when $n=0$ and $\alpha = \hh$$)$, we have
\begin{equation}\label{Lem1_eq_1}
S_{n,\alpha}(t) = -\frac{1}{\pi} \,\,\im{\left\{\dfrac{i^{n}}{n!}\int_{\alpha}^{\infty}{\left(\sigma-\alpha\right)^{n}\,\frac{\zeta'}{\zeta}(\sigma+it)}\,\d\sigma\right\}}.
\end{equation}
\item[(ii)] For $n =-1$, $\hh < \alpha\leq1$ and $t >0$, we have
$$S_{-1,\alpha}(t):= S_{0,\alpha}'(t) = \frac{1}{\pi}\, \re \frac{\zeta'}{\zeta}(\alpha + it).$$
\end{itemize}
\end{lemma}
\begin{proof} For the case $\alpha = \hh$ this is stated without proof in Selberg \cite[Section 2]{S} and it is proved by Fujii in \cite[Lemmas 1 and 2]{F1}. We provide here a proof for general $\hh\leq\alpha\leq1$ by induction for $n\ge0$. The validity of the formula for $n=0$ is clear. Let $R_{n,\alpha}(t)$ be the expression on the right-hand side of \eqref{Lem1_eq_1}. Let us assume that the result holds for $n = 0,1,2, \ldots, m-1$. Differentiating under the integral sign and using integration by parts one can check that $R_{m,\alpha}'(t) = R_{m-1,\alpha}(t) = S_{m-1,\alpha}(t)$ (in case $\alpha = \hh$ and $m=1$ we may assume here that $t$ is not the ordinate of a zero). From \eqref{Intro_eq1_int_S_n} it remains to show that $\lim_{t \to 0^+}R_{m,\alpha}(t) = \delta_{m,\alpha}$ for $m \geq 1$. This follows by integrating by parts $m$ times and then taking the limit as $t \to 0^+$. Part (ii) just follows from the definition of $S_{-1,\alpha}(t)$.
\end{proof}

We are now in position to state the main result of this section, an expression that connects $S_{n,\alpha}(t)$ with the functions $f_{n,\alpha}$ defined in \eqref{Def_f_2m}, \eqref{Def_f_2m+1} and \eqref{Def_f_-1}. In the proof of Theorem \ref{Thm2} we shall only use the case of $n$ odd, but we state here the representation for $n$ even as well, as a result of independent interest.

\begin{lemma}[Representation lemma]\label{Rep_lem}
Assume the Riemann hypothesis. For each $n\geq -1$ and $\hh \leq \alpha \leq 1$ $($except $n = -1$ and $\alpha = \hh$$)$, let $f_{n,\alpha}:\R \to \R$ be defined as in \eqref{Def_f_2m}, \eqref{Def_f_2m+1} and \eqref{Def_f_-1}. For $t\geq2$ $($and $t$ not coinciding with an ordinate of a zero of $\zeta(s)$ in the case $n=0$ and $\alpha = \hh$$)$ the following formulae hold.
\begin{itemize}
\item[(i)] If $n = 2m$, for $m \in \Z_{\ge0}$, then
\begin{equation} \label{Lem2_eq_representation_even}
S_{2m,\alpha}(t)= \frac{(-1)^{m}}{\pi(2m)!} \,\sum_{\gamma}f_{2m,\alpha}(t-\gamma)\,+\,O_m(1).  
\end{equation}
\item[(ii)] If $n = 2m+1$, for $m \in \Z_{\ge0}$, then
\begin{align} \label{Lem2_eq_representation_odd}
S_{2m+1,\alpha}(t)& =\dfrac{(-1)^{m}}{2\pi(2m+2)!}\left(\tfrac{3}{2}-\alpha\right)^{2m+2}\log t-
\dfrac{(-1)^{m}}{\pi(2m)!} \sum_{\gamma} f_{2m+1,\alpha}(t-\gamma) + O_m(1).
\end{align}

\smallskip

\item[(iii)] If $n = -1$, then
\begin{equation} \label{Lem2_eq_representation_-1}
S_{-1,\alpha}(t)= -\frac{1}{2\pi}\log\frac{t}{2\pi} + \frac{1}{\pi} \sum_{\gamma}f_{-1,\alpha}(t-\gamma)+O\left(\frac{1}{t}\right).
\end{equation}
\end{itemize}
The sums in \eqref{Lem2_eq_representation_even}, \eqref{Lem2_eq_representation_odd} and \eqref{Lem2_eq_representation_-1} run over the ordinates of the non-trivial zeros $\rho = \tfrac12 + i \gamma$ of $\zeta(s)$. 
\end{lemma}

\begin{proof}
We first treat (ii). It follows from Lemma \ref{Lem2} and integration by parts that 
\begin{align}\label{expansion_S_m_odd}
\begin{split}
S_{2m+1,\alpha}(t) &= -\frac{1}{\pi} \,\,\im {\left\{\frac{i^{2m+1}}{(2m+1)!}\int_{\alpha}^{\infty}{\left(\sigma-\alpha\right)^{2m+1}\,\frac{\zeta'}{\zeta}(\sigma+it)}\,\d\sigma\right\}}\\
& = \frac{(-1)^{m+1}}{\pi (2m+1)!}\, \re {\left\{\int_{\alpha}^{\infty}{\left(\sigma-\alpha\right)^{2m+1}\,\frac{\zeta'}{\zeta}(\sigma+it)}\,\d\sigma\right\}}\\
& = \frac{(-1)^{m}}{\pi (2m)!}\, \re {\left\{\int_{\alpha}^{\infty}{\left(\sigma-\alpha\right)^{2m}\,\log \zeta(\sigma+it)}\,\d\sigma\right\}}\\
& = \frac{(-1)^{m}}{\pi (2m)!}\, \int_{\alpha}^{3/2}{\left(\sigma-\alpha\right)^{2m}\,\log |\zeta(\sigma+it)}|\,\d\sigma + O_m(1).
\end{split}
\end{align}
The idea is to replace the integrand by an absolutely convergent sum over the zeros of $\zeta(s)$ and then integrate term-by-term. In order to do so, we start with Riemann's $\xi-$function defined by 
$$\xi(s)=\tfrac{1}{2}\,s\,(s-1)\,\pi^{-s/2}\,\Gamma\left(\tfrac{s}{2}\right)\,\zeta(s).$$ 
The function $\xi(s)$ is entire of order $1$ and the zeros of $\xi(s)$ correspond to the non-trivial zeros of $\zeta(s)$. By Hadamard's factorization formula (cf. \cite[Chapter 12]{Dav}), we have
\begin{equation*}
\xi(s)=e^{A+Bs}\displaystyle\prod_{\rho}\bigg(1-\frac{s}{\rho}\bigg)e^{s/\rho}\,,
\end{equation*}
where $\rho = \beta + i \gamma$ runs over the non-trivial zeros of $\zeta(s)$, A is a constant and $B=-\sum_{\rho}\re(1/\rho)$. Note that $\re(1/\rho)$ is positive and that $\sum_{\rho}\re(1/\rho)$ converges absolutely. Assuming the Riemann hypothesis, it follows that
\begin{equation*}
\left|\dfrac{\xi(\sigma+it)}{\xi(\tfrac32+it)}\right|=\displaystyle\prod_{\gamma}\left(\dfrac{\big(\sigma-\tfrac12)^2+(t-\gamma)^2}{1+(t-\gamma)^2}\right)^{\frac12}.
\end{equation*}
Hence
$$
\log|\xi(\sigma+it)|-\log\left|\xi\left(\tfrac{3}{2}+it\right)\right|= \dfrac{1}{2}\displaystyle\sum_{\gamma}\log\left(\dfrac{\big(\sigma-\tfrac12)^2+(t-\gamma)^2}{1+(t-\gamma)^2}\right).
$$
From Stirling's formula for $\Gamma(s)$ (cf. \cite[Chapter 10]{Dav}) we obtain
\begin{align} \label{Stirling_1}
\begin{split}
\log|\zeta(\sigma+it)|&=\left(\tfrac{3}{4}-\tfrac{\sigma}{2}\right)\log t - \frac{1}{2}\displaystyle\sum_{\gamma}\log\left(\frac{1+(t-\gamma)^2}{\big(\sigma-\frac12)^2+(t-\gamma)^2}\right) + O(1), 
\end{split}
\end{align}
uniformly for $\tfrac{1}{2}\leq \sigma \leq \tfrac{3}{2}$ and $t\geq 2$. Inserting \eqref{Stirling_1} into \eqref{expansion_S_m_odd} yields
\begin{align*}
S_{2m+1,\alpha}(t)  & = \dfrac{(-1)^{m}}{\pi(2m)!}\left(\int_{\alpha}^{3/2}\left(\sigma-\alpha\right)^{2m}\left(\tfrac{3}{4}-\tfrac{\sigma}{2}\right) \,\d \sigma \right)\log t\\
& \ \ \ \ \ \ \ \ \ \ \ -\frac{(-1)^{m}}{2\pi(2m)!} \int_{\alpha}^{3/2}{\sum_{\gamma}\left(\sigma-\alpha\right)^{2m}\log\left(\dfrac{1+(t-\gamma)^2}{(\sigma-\tfrac12)^2+(t-\gamma)^2}\right)} \,\d\sigma + O_m(1) \\
& = \frac{(-1)^{m}}{2\pi(2m+2)!}\left(\tfrac{3}{2}-\alpha\right)^{2m+2}\log t \\
&  \ \ \ \ \ \ \ \ \ \ \ -\frac{(-1)^{m}}{2\pi(2m)!} \sum_{\gamma}\int_{\alpha}^{3/2}\left(\sigma-\alpha\right)^{2m}\log\left(\dfrac{1+(t-\gamma)^2}{(\sigma-\tfrac12)^2+(t-\gamma)^2}\right) \,\d\sigma + O_m(1) \\
              & = \dfrac{(-1)^{m}}{2\pi(2m+2)!}\left(\tfrac{3}{2}-\alpha\right)^{2m+2}\log t-\dfrac{(-1)^{m}}{\pi(2m)!} \sum_{\gamma} f_{2m+1,\alpha}(t-\gamma) + O_m(1),         
\end{align*}
where the interchange between summation and integration can be justified, for instance, by the monotone convergence theorem, since all the terms involved are nonnegative. This concludes the proof of (ii).

\smallskip

We now move to the proof of (iii). Let $s=\alpha+it$ and recall that we are assuming $t \geq 2$. From the partial fraction decomposition for $\zeta'(s)/\zeta(s)$ (cf. \cite[Chapter 12]{Dav}), we have
\begin{align}\label{partial_fraction_dec}
	\dfrac{\zeta'}{\zeta}(s) & = \displaystyle\sum_{\rho}\bigg(\dfrac{1}{s-\rho}+\frac{1}{\rho}\bigg)-\dfrac{1}{2}\dfrac{\Gamma'}{\Gamma}\bigg(\dfrac{s}{2}+1\bigg)+B+\frac{1}{2}\log\pi - \frac{1}{s-1},
\end{align}
with $B=-\sum_{\rho}\re(1/\rho)$.
Again using Stirling's formula 
we obtain
\begin{align*} 
	S_{-1,\alpha}(t)= \frac{1}{\pi}\, \re \frac{\zeta'}{\zeta}(\alpha + it)& = - \dfrac{1}{2\pi}\log\frac{t}{2\pi} + \dfrac{1}{\pi}\sum_{\gamma}f_{-1,\alpha}(t-\gamma)+O\bigg(\frac{1}{t}\bigg).
\end{align*}
This proves (iii).

\smallskip

Finally, the proof of (i) follows along the same lines, starting with \eqref{Lem1_eq_1}, restricting the range of integration to the interval $(\alpha, \tfrac32)$, and using the partial fraction decomposition \eqref{partial_fraction_dec} after adding and subtracting a term $\frac{\zeta'}{\zeta}(\tfrac32 +it)$ to balance the equation. The details of the proof are left to the interested reader.
\end{proof}

In the proof of Theorem \ref{Thm2} we will be using the following version of the Guinand-Weil explicit formula which connects the zeros of the zeta function and the prime powers. 

\begin{lemma}[Guinand-Weil explicit formula] \label{GW}
Let $h(s)$ be analytic in the strip $|\im{s}|\leq \tfrac12+\varepsilon$ for some $\varepsilon>0$, and assume that $|h(s)|\ll(1+|s|)^{-(1+\delta)}$ for some $\delta>0$ when $|\re{s}|\to\infty$. Let $h(w)$ be real-valued for real $w$, and let $\widehat{h}(x)=\int_{-\infty}^{\infty}h(w)e^{-2\pi ixw}\,\dw$. Then
	\begin{align*}
	\displaystyle\sum_{\rho}h\left(\frac{\rho-\frac12}{i}\right) & = h\left(\dfrac{1}{2i}\right)+h\left(-\dfrac{1}{2i}\right)-\dfrac{1}{2\pi}\widehat{h}(0)\log\pi+\dfrac{1}{2\pi}\int_{-\infty}^{\infty}h(u)\,\re{\dfrac{\Gamma'}{\Gamma}\left(\dfrac{1}{4}+\dfrac{iu}{2}\right)}\,\du \\
	 &  \ \ \ \ \ \ \ \ \ \ \ \ \ -\dfrac{1}{2\pi}\displaystyle\sum_{n\geq2}\dfrac{\Lambda(n)}{\sqrt{n}}\left(\widehat{h}\left(\dfrac{\log n}{2\pi}\right)+\widehat{h}\left(\dfrac{-\log n}{2\pi}\right)\right)\,, 
	\end{align*}
where $\rho = \beta + i \gamma$ are the non-trivial zeros of $\zeta(s)$, $\Gamma'/\Gamma$ is the logarithmic derivative of the Gamma function, and $\Lambda(n)$ is the Von-Mangoldt function defined to be $\log p$ if $n=p^m$ with $p$ a prime number and $m\geq 1$ an integer, and zero otherwise.
\end{lemma}
\begin{proof}
	This is \cite[Lemma 1]{GG}. The proof follows from \cite[Theorem 5.12]{IK}. 
\end{proof}

As noted in the introduction, the functions $f_{n,\alpha}$ that appear in Lemma \ref{Rep_lem} do not possess the required smoothness properties to allow the application of the Guinand-Weil formula. The key idea to prove Theorem \ref{Thm2}, {\it in the case of $n$ odd}, is to replace the functions $f_{n,\alpha}$ by appropriate extremal majorants and minorants of exponential type (thus with a compactly supported Fourier transform by the Paley-Wiener theorem). These bandlimited approximations are described in the next section.

\section{Extremal bandlimited approximations}\label{Extremal_functions_section}

\subsection{Preliminaries} Recall that an entire function $G:\C \to \C$ is said to have exponential type $\tau$ if 
$$\limsup_{|z| \to \infty} \frac{\log |G(z)|}{|z|} \leq \tau.$$ 
The celebrated Paley-Wiener theorem states that a function $g \in L^2(\R)$ has Fourier transform supported in the interval $[-\Delta, \Delta]$ if and only if it is equal almost everywhere to the restriction to $\R$ of an entire function of exponential type $2\pi \Delta$. The term {\it bandlimited} is commonly used in the applied literature in reference to functions that have compactly supported Fourier transforms.

\smallskip

In this section we consider a particular instance of the so called Beurling-Selberg extremal problem in Fourier analysis and approximation theory. In general terms, this is the problem of finding one-sided approximations of real-valued functions by entire functions of prescribed exponential type, seeking to minimize the $L^1(\R)-$error. This problem has its origins in the work of A. Beurling in the late 1930's, in which he constructed extremal majorants and minorants of exponential type for the signum function. Later, A. Selberg used Beurling's extremal functions to produce majorants and majorants for characteristic functions of intervals, and applied these in connection to large sieve inequalities. The survey \cite{V} by J. D. Vaaler is the classical reference on the subject, describing some of the historical milestones of the problem and presenting several interesting applications of such special functions in analysis and number theory.

\smallskip

In recent years there has been considerable progress both in the constructive aspects and in the range of applications of such extremal bandlimited approximations. For the constructive theory we highlight, for instance, the works \cite{CL, CLV, CV2, GV, Ke, Litt, Litt2, LS} in the one-dimensional theory and the works \cite{CG, CL2, CL3, GKM, HV} in the multi-dimensional and weighted theory. These allowed new applications in the theory of the Riemann zeta-function and general $L$-functions, for instance in \cite{CC, CCLM, CCM, CCM2, CChi, CF, CS, G, GG, M}.

\smallskip

The {\it Gaussian subordination framework} of Carneiro, Littmann and Vaaler \cite{CLV} is a robust method to solve the Beurling-Selberg problem for even functions in dimension one. This is the main tool we shall use in this section. In particular, functions $g:\mathbb{R}\to\mathbb{R}$ of the form
\[
g(x)=\int_{0}^{\infty}e^{-\pi\lambda x^2}\,\d\nu(\lambda),
\]
where $\nu$ is a finite nonnegative Borel measure on $(0,\infty)$, fall under the scope of \cite{CLV}. It turns out that our functions $f_{n,\alpha}$ when $n$ is odd, defined in \eqref{Def_f_2m+1} and \eqref{Def_f_-1}, are included in this class as we shall see from the results below. Moreover, it is also crucial for our purposes to have a detailed description of the Fourier transforms of our majorants and minorants in order to analyze the contribution from the primes and prime powers in the explicit formula.

\subsection{Approximations to the Poisson kernel} We start with the case of the Poisson kernel $f_{-1,\alpha}$. In order to simplify the notation we let $\beta = \alpha - \hh$ and define
\begin{equation}\label{Def_Poisson_kernel_beta}
h_{\beta}(x) := f_{-1,\alpha}(x) = \frac{\beta}{\beta^2 + x^2}.
\end{equation}
The solution of the extremal problem for the Poisson kernel below is of independent interest and may have other applications in analysis and number theory. Recall that a {\it real entire function} is an entire function whose restriction to $\R$ is real-valued.

\begin{lemma}[Extremal functions for the Poisson kernel]\label{lemma_extr_Poisson} Let $\beta >0$ be a real number and let $\Delta >0$ be a real parameter. Let $h_{\beta}:\R \to \R$ be defined as in \eqref{Def_Poisson_kernel_beta}. Then there is a unique pair of real entire functions $m_{\beta,\Delta}^{-}:\mathbb{C}\to\mathbb{C}$ and $m_{\beta,\Delta}^{+}:\mathbb{C}\to\mathbb{C}$ satisfying the following properties:
	\begin{itemize}
	\item[(i)] The real entire functions $m_{\beta,\Delta}^{\pm}$ have exponential type $2\pi\Delta$.
	
	\smallskip
	
	\item[(ii)] The inequality 
	$$m_{\beta,\Delta}^{-}(x) \leq h_{\beta}(x) \leq m_{\beta,\Delta}^{+}(x)$$
holds pointwise for all $x \in \R$.

\smallskip

	\item[(iii)] Subject to conditions {\rm(i)} and {\rm (ii)}, the value of the integral 
$$ \int_{-\infty}^{\infty}\big\{m^{+}_{\beta,\Delta}(x)-m^{-}_{\beta,\Delta}(x)\big\}\,\dx$$
is minimized.
\end{itemize}

\smallskip

\noindent The functions $m_{\beta,\Delta}^{\pm}$ are even and verify the following additional properties:

\smallskip

	\begin{itemize}
	 \item[(iv)]The $L^1-$distances of $m_{\beta,\Delta}^{\pm}$ to $h_{\beta}$ are explicitly given by
	 \begin{equation}\label{Poisson_int1}
		\int_{-\infty}^{\infty}\big\{m^{+}_{\beta,\Delta}(x)-h_{\beta}(x)\big\}\,\dx=\dfrac{2\pi e^{-2\pi\beta\Delta}}{1-e^{-2\pi\beta\Delta}}
		\end{equation}
		and
		\begin{equation}\label{Poisson_int2}
		\int_{-\infty}^{\infty}\big\{h_{\beta}(x)-m^{-}_{\beta,\Delta}(x)\big\}\,\dx=\dfrac{2\pi e^{-2\pi\beta\Delta}}{1+e^{-2\pi\beta\Delta}}.
		\end{equation}

		\item[(v)] The Fourier transforms of $m_{\beta,\Delta}^{\pm}$, namely
		\[
		\widehat{m}_{\beta,\Delta}^{\pm}(\xi)=\int_{-\infty}^{\infty}m_{\beta,\Delta}^{\pm}(x)\,e^{-2\pi ix\xi}\,\dx\,,
		\]
		are even continuous functions supported on the interval $[-\Delta,\Delta]$ given by
		\begin{equation} \label{FT_maj_min_Poisson}
		\widehat{m}_{\beta,\Delta}^{\pm}(\xi) = \pi \left(\dfrac{e^{2\pi\beta(\Delta - |\xi|)}-e^{-2\pi\beta(\Delta-|\xi|)}}{\left(e^{\pi\beta\Delta}\mp e^{-\pi\beta\Delta}\right)^2}\right).
		\end{equation}
		
		\smallskip
		
		\item[(vi)] The functions $m_{\beta,\Delta}^{\pm}$ are explicitly given by
		\begin{equation}\label{Explicit_expression_maj_min_Poisson}
		m_{\beta,\Delta}^{\pm}(z) = \left(\frac{\beta}{\beta^2 + z^2}\right) \left( \frac{e^{2\pi \beta \Delta} +  e^{-2\pi \beta \Delta} - 2\cos(2\pi\Delta z)}{\left(e^{\pi\beta\Delta}\mp e^{-\pi\beta\Delta}\right)^2} \right).
		\end{equation}
In particular, the function $m_{\beta,\Delta}^{-}$ is nonnegative on $\R$.	

\smallskip

	\item[(vii)] Assume that $0 < \beta \leq \frac12$ and $\Delta \geq 1$. For any real number $x$ we have
	\begin{equation}\label{decay_square_maj_Poisson}
	 0 < m_{\beta,\Delta}^{-}(x) \leq h_{\beta}(x) \leq  m_{\beta,\Delta}^{+}(x) \ll \frac{1}{\beta(1 + x^2)},
	 \end{equation}
and, for any complex number $z=x+iy$, we have
		\begin{align}\label{Complex_est_maj}
		\big|m_{\beta,\Delta}^{+}(z)\big|\ll \frac{\Delta^{2}e^{2\pi\Delta|y|}}{\beta(1+\Delta|z|)} 
		\end{align}
		and
		\begin{align}\label{Complex_est_min}
		\big|m_{\beta,\Delta}^{-}(z)\big|\ll \dfrac{\beta \Delta^{2}e^{2\pi\Delta|y|}}{1+\Delta|z|},   
		\end{align}
where the constants implied by the $\ll$ notation are universal.	
\end{itemize}	
\end{lemma}

\begin{proof} We start by observing that 
$$h_{\beta}(x) =  \int_0^{\infty} e^{-\pi \lambda x^2} \d\nu_{\beta}(\lambda)\,,$$
where $\nu_{\beta}$ is the finite nonnegative measure given by $\d\nu_{\beta}(\lambda)= \pi\beta  \, e^{-\pi \lambda \beta^2} \,\d\lambda$. Let us define the auxiliary function
$$H_{\beta, \Delta}(x) = h_{\beta}\left(\frac{x}{\Delta}\right) = \frac{\beta \Delta^2 }{\beta^2\Delta^2 + x^2} = \int_0^{\infty} e^{-\pi \lambda x^2} \d\nu_{\beta,\Delta}(\lambda)\,,$$
where $\nu_{\beta,\Delta}$ is the finite nonnegative measure given by $\d\nu_{\beta,\Delta}(\lambda)= \pi\beta\Delta^2  \, e^{-\pi \lambda\beta^2 \Delta^2} \,\d\lambda.$

\smallskip

From \cite[Section 11]{CLV} we know that there is a unique extremal majorant $M_{\beta, \Delta}^{+}(z)$ of exponential type $2\pi$ and a unique extremal minorant $M_{\beta, \Delta}^{-}(z)$ of exponential type $2\pi$ for the real-valued function $H_{\beta, \Delta}$, and these are given by
\begin{align}\label{Lem6_expl_maj_Poisson}
M^{+}_{\beta,\Delta}(z)=\bigg(\dfrac{\sin\pi z}{\pi}\bigg)^2\left\{\displaystyle\sum_{n=-\infty}^{\infty}\dfrac{H_{\beta,\Delta}(n)}{(z-n)^2}+\displaystyle\sum_{n\neq 0}\dfrac{H^{'}_{\beta,\Delta}(n)}{(z-n)}\right\}
\end{align}
and
\begin{align}\label{Lem6_expl_min_Poisson}
M^{-}_{\beta,\Delta}(z)=\bigg(\dfrac{\cos\pi z}{\pi}\bigg)^2 \left\{\displaystyle\sum_{n=-\infty}^{\infty}\dfrac{H_{\beta,\Delta}\big(n-\frac{1}{2}\big)}{\big(z-n+\frac{1}{2}\big)^2}+\dfrac{H^{'}_{\beta,\Delta}\big(n-\frac{1}{2}\big)}{\big(z-n+\frac{1}{2}\big)}\right\}.
\end{align}
We now set
\begin{equation*}
m_{\beta, \Delta}^+(z) := M_{\beta, \Delta}^+(\Delta z) \,\,\,\,\,\, {\rm and} \,\,\,\,\,\,m_{\beta, \Delta}^-(z) := M_{\beta, \Delta}^-(\Delta z) ,
\end{equation*}
and a simple change of variables shows that these will be the unique extremal functions of exponential type $2\pi \Delta$ for $h_{\beta}$, as described in (i), (ii) and (iii). From \eqref{Lem6_expl_maj_Poisson} and \eqref{Lem6_expl_min_Poisson} it is clear that $M^{\pm}_{\beta,\Delta}$, and hence $m^{\pm}_{\beta,\Delta}$, are even functions.

\smallskip

We now verify properties (iv) - (vii).

\subsubsection*{Property {\rm(iv)}} Since $M_{\beta, \Delta}^{\pm}$ are entire functions of exponential type $2\pi$ whose restrictions to $\R$ belong to $L^1(\R)$, a classical result of Plancherel and P\'{o}lya \cite{PP} (see also \cite[Eq. (3.1) and (3.2)]{V}) guarantees that $M_{\beta, \Delta}^{\pm}$ are bounded on the real line and hence belong to $L^2(\R)$ as well. Moreover, still by \cite{PP}, their derivatives $(M_{\beta, \Delta}^{\pm})'$ are also entire functions of exponential type $2\pi$ whose restrictions to $\R$ belong to $L^1(\R) \cap L^2(\R)$. In particular, $M_{\beta, \Delta}^{\pm}$ are integrable and of bounded variation on $\R$, and thus the Poisson summation formula holds pointwise. This can be used to calculate the values of the integrals of $M_{\beta, \Delta}^{\pm}$. Using the fact that $\widehat{M}_{\beta, \Delta}^{\pm}$ are supported in the interval $[-1,1]$ (which follows from the Paley-Wiener theorem) and the fact that $M_{\beta, \Delta}^{+}$ interpolates the values of $H_{\beta,\Delta}$ at $\Z$ (resp. $M_{\beta, \Delta}^{-}$ interpolates the values of $H_{\beta,\Delta}$ at $\Z+ \hh$) we find
\begin{align*}
\widehat{M}_{\beta, \Delta}^{+}(0) & = \sum_{n = -\infty}^{\infty} M_{\beta, \Delta}^{+}(n) = \sum_{n = -\infty}^{\infty} H_{\beta,\Delta}(n) = \sum_{k = -\infty}^{\infty} \widehat{H}_{\beta,\Delta}(k) = \sum_{k = -\infty}^{\infty} \pi \Delta e^{-2\pi \beta \Delta |k|} = \pi \Delta \left(\frac{ 1 + e^{-2\pi\beta\Delta }}{1 - e^{-2\pi\beta \Delta}}\right)
\end{align*}
and
\begin{align*}
\begin{split}
\widehat{M}_{\beta, \Delta}^{-}(0)  = \sum_{n = -\infty}^{\infty} M_{\beta, \Delta}^{-}(n + \hh) & = \sum_{n = -\infty}^{\infty} H_{\beta,\Delta}(n + \hh) \\
& = \sum_{k = -\infty}^{\infty} (-1)^k \,\widehat{H}_{\beta,\Delta}(k) = \sum_{k = -\infty}^{\infty} (-1)^k\,\pi \Delta e^{-2\pi \beta \Delta |k|} = \pi \Delta \left(\frac{ 1 - e^{-2\pi\beta\Delta }}{1 + e^{-2\pi\beta\Delta}}\right).
\end{split}
\end{align*}
The relation $\widehat{m}_{\beta, \Delta}^{\pm}(0) = \frac{1}{\Delta} \widehat{M}_{\beta, \Delta}^{\pm}(0)$ and the fact that $\widehat{h}_{\beta}(0) = \int_{-\infty}^{\infty}h_{\beta}(x) \,\dx = \pi$ lead us directly to \eqref{Poisson_int1} and \eqref{Poisson_int2}. This establishes property (iv).

\subsubsection*{Property {\rm(v)}} We have already noted that the Fourier transforms $\widehat{M}_{\beta, \Delta}^{\pm}$ are continuous functions (since ${M}_{\beta, \Delta}^{\pm} \in L^1(\R)$) supported in the interval $[-1,1]$. From a classical result of Vaaler \cite[Theorem 9]{V} one has the explicit expression for the Fourier transform of the majorant, in which we use the fact that $M_{\beta, \Delta}^{+}(n) = H_{\beta,\Delta}(n)$ and $(M^{+}_{\beta, \Delta})'(n) = H'_{\beta,\Delta}(n)$ for all $n \in \Z$, 
\begin{align}\label{FT_maj_poisson}
\begin{split}
\widehat{M}_{\beta, \Delta}^{+}(\xi) &= \sum_{n=-\infty}^{\infty} \left( \left(1 - |\xi| \right) M_{\beta, \Delta}^{+}(n) + \frac{1}{2\pi i} \,\sgn(\xi) \,(M^{+}_{\beta, \Delta})'(n)\right) e^{-2\pi i n \xi}
\\
& = \sum_{n=-\infty}^{\infty} \left( \left(1 - |\xi| \right) H_{\beta,\Delta}(n) + \frac{1}{2\pi i} \,\sgn(\xi) \,H'_{\beta,\Delta}(n)\right) e^{-2\pi i n \xi}
\end{split}
\end{align}
for $\xi \in [-1,1]$. Using the Poisson summation formula we have 
\begin{align}\label{PSF_Poisson_eq1}
\begin{split}
\sum_{n=-\infty}^{\infty} H_{\beta,\Delta}(n) \,e^{-2\pi i n \xi}  & = \sum_{k=-\infty}^{\infty} \widehat{H}_{\beta,\Delta}(\xi + k) \\
& = \sum_{k=-\infty}^{\infty} \pi \Delta \,e^{-2 \pi \beta \Delta |\xi + k|}\\
& = \pi \Delta \left( \frac{e^{-2\pi\beta \Delta |\xi|} + e^{-2\pi\beta \Delta (1 - |\xi|)}}{1 - e^{-2\pi\beta\Delta}} \right)
\end{split}
\end{align}
and
\begin{align}\label{PSF_Poisson_eq2}
\begin{split}
\sum_{n=-\infty}^{\infty}  H'_{\beta,\Delta}(n) & \,e^{-2\pi i n \xi}  = \sum_{k=-\infty}^{\infty} \widehat{H'}_{\beta,\Delta}(\xi + k)\\
&  = \sum_{k=-\infty}^{\infty}2\pi i (\xi + k) \widehat{H}_{\beta,\Delta}(\xi + k)\\
&  = \sum_{k=-\infty}^{\infty}2\pi i (\xi + k) \pi \Delta \,e^{-2 \pi \beta \Delta |\xi + k|}\\
& = 2\pi^2 i \Delta \sgn(\xi) \left(  \frac{|\xi| \left(e^{-2 \pi \beta \Delta |\xi|} + e^{-2 \pi \beta \Delta(1- |\xi|)}\right)}{1 - e^{-2 \pi \beta \Delta}} - \frac{e^{-2\pi \beta \Delta} \left(e^{2 \pi \beta \Delta |\xi|} - e^{-2 \pi \beta \Delta |\xi|}\right)}{\left(1 - e^{-2 \pi \beta \Delta}\right)^2}\right).
\end{split}
\end{align}
Plugging \eqref{PSF_Poisson_eq1} and \eqref{PSF_Poisson_eq2} into \eqref{FT_maj_poisson} gives us
\begin{align*}
\widehat{M}_{\beta, \Delta}^{+}(\xi) =  \pi \Delta \left(\dfrac{e^{2\pi\beta\Delta(1 - |\xi|)}-e^{-2\pi\beta\Delta(1-|\xi|)}}{\left(e^{\pi\beta\Delta} - e^{-\pi\beta\Delta}\right)^2}\right),
\end{align*}
and from the fact that 
\begin{equation}\label{change_var_FT}
\widehat{m}_{\beta, \Delta}^{\pm}(\xi) = \frac{1}{\Delta} \widehat{M}_{\beta, \Delta}^{\pm}\Big(\frac{\xi}{\Delta}\Big)
\end{equation}
we arrive at \eqref{FT_maj_min_Poisson} for the majorant.

\smallskip

For the minorant we proceed analogously. From \cite[Theorem 9]{V} one has the representation, in which we use the fact that $M_{\beta, \Delta}^{-}(n+\hh) = H_{\beta,\Delta}(n+\hh)$ and $(M^{-}_{\beta, \Delta})'(n+\hh) = H'_{\beta,\Delta}(n+\hh)$ for all $n \in \Z$,
\begin{align}\label{FT_min_poisson}
\begin{split}
\widehat{M}_{\beta, \Delta}^{-}(\xi) & = \sum_{n=-\infty}^{\infty} \left( \left(1 - |\xi| \right) M_{\beta, \Delta}^{-}(n+\hh) + \frac{1}{2\pi i} \,\sgn(\xi) \,(M^{-}_{\beta, \Delta})'(n+\hh)\right) e^{-2\pi i (n+\frac12) \xi}
\\& = \sum_{n=-\infty}^{\infty} \left( \left(1 - |\xi| \right) H_{\beta,\Delta}(n + \hh) + \frac{1}{2\pi i} \,\sgn(\xi) \,H'_{\beta,\Delta}(n + \hh)\right) e^{-2\pi i (n+\frac12) \xi}
\end{split}
\end{align}
for $\xi \in [-1,1]$. Poisson summation now yields
\begin{align}\label{PSF_Poisson_eq3}
\begin{split}
\sum_{n=-\infty}^{\infty} H_{\beta,\Delta}(n+\hh) \,e^{-2\pi i (n+\frac12) \xi}  & = \sum_{k=-\infty}^{\infty} (-1)^k \,\widehat{H}_{\beta,\Delta}(\xi + k)  = \pi \Delta \left( \frac{e^{-2\pi\beta \Delta |\xi|} - e^{-2\pi\beta \Delta (1 - |\xi|)}}{1 + e^{-2\pi\beta\Delta}} \right)
\end{split}
\end{align}
and
\begin{align}\label{PSF_Poisson_eq4}
\begin{split}
\sum_{n=-\infty}^{\infty} & H'_{\beta,\Delta}(n+\hh) \,e^{-2\pi i (n+\frac12) \xi}   = \sum_{k=-\infty}^{\infty}2\pi i \,(\xi + k)\,(-1)^k\, \widehat{H}_{\beta,\Delta}(\xi + k)  \\
& =  2\pi^2 i \Delta \sgn(\xi) \left(  \frac{|\xi| \left(e^{-2 \pi \beta \Delta |\xi|} - e^{-2 \pi \beta \Delta(1- |\xi|)}\right)}{1 + e^{-2 \pi \beta \Delta}} + \frac{e^{-2\pi \beta \Delta} \left(e^{2 \pi \beta \Delta |\xi|} - e^{-2 \pi \beta \Delta |\xi|}\right)}{\left(1 + e^{-2 \pi \beta \Delta}\right)^2}\right).
\end{split}
\end{align}
Plugging \eqref{PSF_Poisson_eq3} and \eqref{PSF_Poisson_eq4} into \eqref{FT_min_poisson} gives us
\begin{align*}
\widehat{M}_{\beta, \Delta}^{-}(\xi) =  \pi \Delta \left(\dfrac{e^{2\pi\beta\Delta(1 - |\xi|)}-e^{-2\pi\beta\Delta(1-|\xi|)}}{\left(e^{\pi\beta\Delta} + e^{-\pi\beta\Delta}\right)^2}\right),
\end{align*}
and using \eqref{change_var_FT} we arrive at \eqref{FT_maj_min_Poisson} for the minorant. This completes the proof of (v).

\subsubsection*{Property {\rm(vi)}} The proof of (vi) is a direct computation using (v) and Fourier inversion
\begin{align*}
m_{\beta,\Delta}^{\pm}(z)& = \int_{-\Delta}^{\Delta}  \pi \left(\dfrac{e^{2\pi\beta(\Delta - |\xi|)}-e^{-2\pi\beta(\Delta-|\xi|)}}{\left(e^{\pi\beta\Delta}\mp e^{-\pi\beta\Delta}\right)^2}\right) e^{2\pi i \xi z}\,\d\xi.
\end{align*}
We omit the details of this calculation.

\subsubsection*{Property {\rm(vii)}} From \eqref{Explicit_expression_maj_min_Poisson} it follows directly that $0 < m_{\beta,\Delta}^{-}(x)$ for all $x \in \R$. We may also write
\begin{equation}\label{Est_0_maj_real_line}
m_{\beta,\Delta}^{+}(x) = \frac{\beta}{\beta^2 + x^2} \left(1 + \frac{4\sin^2(\pi \Delta x)}{\left(e^{\pi\beta\Delta}- e^{-\pi\beta\Delta}\right)^2}\right). 
\end{equation}
We then note that in the range $0 < \beta \leq \hh$ and $\Delta \geq 1$ the following estimates hold:
\begin{equation}\label{Est_1_maj_real_line}
\frac{\beta}{\beta^2 + x^2} \ll \frac{1}{\beta(1 + x^2)}
\end{equation}
and
\begin{align}\label{Est_2_maj_real_line}
\begin{split}
\left(\frac{\beta}{\beta^2 + x^2}\right)\frac{\sin^2(\pi \Delta x)}{\left(e^{\pi\beta\Delta}- e^{-\pi\beta\Delta}\right)^2} & = \left(\frac{\beta}{\beta^2 + x^2}\right) \left(\frac{\sin(\pi \Delta x)}{\Delta x}\right)^2\left( \frac{\Delta x }{\beta \Delta}\right)^2\left(\frac{\beta\Delta} {e^{\pi\beta\Delta}- e^{-\pi\beta\Delta}}\right)^2 \\
& \ll \left(\frac{\beta}{\beta^2 + x^2}\right)\left(\frac{1}{1 + \Delta^2x^2}\right)\left( \frac{ x }{\beta}\right)^2 \\
& \ll \frac{1}{\beta(1 + x^2)}.
\end{split}
\end{align}
Using \eqref{Est_1_maj_real_line} and \eqref{Est_2_maj_real_line} in \eqref{Est_0_maj_real_line}  yields the estimate
$$m_{\beta,\Delta}^{+}(x) \ll \frac{1}{\beta(1 + x^2)}.$$

The idea to analyze the growth in the complex plane is similar. We start by rewriting \eqref{Explicit_expression_maj_min_Poisson} as 
\begin{equation}\label{rew_complex_Poisson}
m_{\beta,\Delta}^{\pm}(z) = \frac{4}{\beta} \left(\frac{\sin\pi \Delta (z + i \beta)}{\Delta (z + i\beta)} \right)\left(\frac{\sin\pi \Delta (z - i \beta)}{\Delta (z - i\beta)} \right) \left( \frac{\beta\Delta}{e^{\pi\beta\Delta}\mp e^{-\pi\beta\Delta}}\right)^2
\end{equation}
and then apply the following uniform bounds
\begin{align}\label{Unif_bound_1_Poisson}
\left|\frac{\sin w}{ w}\right| \ll \frac{e^{|\im w|}}{1 + |w|}
\end{align}
and
\begin{equation}\label{Unif_bound_2_Poisson}
\frac{1}{(1 + |w + i \gamma|)}\cdot  \frac{1}{(1 + |w - i \gamma|)} \ll  \frac{1}{1 + |w|}
\end{equation}
that are valid for any $w \in \C$ and $\gamma >0$. Using \eqref{Unif_bound_1_Poisson} and \eqref{Unif_bound_2_Poisson} in \eqref{rew_complex_Poisson} we derive that
\begin{align*}
\big|m_{\beta,\Delta}^{\pm}(z)\big| & \ll \frac{1}{\beta} \left(\frac{e^{\pi \Delta (|\im z| + \beta) }}{1 + \Delta|z +i\beta|}\right)\left(\frac{e^{\pi \Delta (|\im z| + \beta) }}{1 + \Delta|z -i\beta|}\right)\left( \frac{\beta\Delta}{e^{\pi\beta\Delta}\mp e^{-\pi\beta\Delta}}\right)^2\\
& \ll\frac{1}{\beta}\left(\frac{e^{2\pi \Delta |\im z| }}{1 + \Delta|z|}\right)\left( \frac{\beta\Delta\,e^{\pi \beta \Delta}}{e^{\pi\beta\Delta}\mp e^{-\pi\beta\Delta}}\right)^2.
\end{align*}
In the majorant case, we have 
$$\left( \frac{\beta\Delta\,e^{\pi \beta \Delta}}{e^{\pi\beta\Delta}- e^{-\pi\beta\Delta}}\right)^2 \ll 1 + (\beta \Delta)^2 \ll \Delta ^2\,,$$
and this leads to \eqref{Complex_est_maj}. In the minorant case we have 
$$\left( \frac{\beta\Delta\,e^{\pi \beta \Delta}}{e^{\pi\beta\Delta}+ e^{-\pi\beta\Delta}}\right)^2 \ll  (\beta \Delta)^2\,, $$
and this leads to \eqref{Complex_est_min}. This concludes the proof of the lemma.
\end{proof}

\subsection{Approximations to the functions $f_{2m+1,\alpha}$} Our next task is to present the analogue of Lemma \ref{lemma_extr_Poisson} (i.e. the solution of the Beurling-Selberg extremal problem) for the family of even functions $f_{2m+1,\alpha}$ defined in \eqref{Def_f_2m+1}. We highlight the explicit description of the Fourier transforms of the extremal bandlimited approximations. This is a slightly technical but extremely important part of this work, since these Fourier transforms will play an important role in the evaluation of the sum over prime powers in the explicit formula.

\begin{lemma}[Extremal functions for $f_{2m+1,\alpha}$]\label{lema_extremal_f_2m+1_alpha} Let $m \geq0$ be an integer and let $\hh \leq \alpha \leq 1$ and $\Delta\geq1$ be real parameters. Let $f_{2m+1,\alpha}$ be the real-valued function defined in \eqref{Def_f_2m+1}, namely
	$$f_{2m+1,\alpha}(x)=\frac{1}{2}\int_{\alpha}^{3/2}{\left(\sigma-\alpha\right)^{2m}\log\left(\dfrac{1+x^2}{(\sigma-\hh)^2+x^2}\right)} \,\d\sigma.$$
	Then there is a unique pair of real entire functions $g_{2m+1,\alpha,\Delta}^{-}:\mathbb{C}\to\mathbb{C}$ and $g_{2m+1,\alpha,\Delta}^{+}:\mathbb{C}\to\mathbb{C}$ satisfying the following properties:
	\begin{itemize}
	\item[(i)] The real entire functions $g_{2m+1,\alpha,\Delta}^{\pm}$ have exponential type $2\pi\Delta$.
	
	\smallskip
	
	\item[(ii)] The inequality 
	$$g_{2m+1,\alpha,\Delta}^{-}(x) \leq f_{2m+1,\alpha}(x) \leq g_{2m+1,\alpha,\Delta}^{+}(x)$$
holds pointwise for all $x \in \R$.

\smallskip

	\item[(iii)] Subject to conditions {\rm(i)} and {\rm (ii)}, the value of the integral 
$$ \int_{-\infty}^{\infty}\big\{g_{2m+1,\alpha,\Delta}^{+}(x)-g_{2m+1,\alpha,\Delta}^{-}(x)\big\}\,\dx$$
is minimized.
\end{itemize}

\smallskip

\noindent The functions $g_{2m+1,\alpha,\Delta}^{\pm}$ are even and verify the following additional properties:

\smallskip

	\begin{itemize}
		\item[(iv)] For any real number $x$ we have
		\begin{align}\label{bound_g_2m+1_real}
		\big|g_{2m+1,\alpha,\Delta}^{\pm}(x)\big|\ll_m \dfrac{1}{1+x^2}\,,
		\end{align}
		and, for any complex number $z=x+iy$, we have
		\begin{align}\label{bound_g_2m+1_complex}
		\big|g_{2m+1,\alpha,\Delta}^{\pm}(z)\big|\ll_m\dfrac{\Delta^{2}e^{2\pi\Delta|y|}}{(1+\Delta|z|)},
		\end{align}
where the constants implied by the $\ll_m$ notation depend only on $m$.

		\smallskip
		
		\item[(v)] The Fourier transforms of $g_{2m+1,\alpha,\Delta}^{\pm}$, namely
		\[
		\widehat{g}_{2m+1,\alpha,\Delta}^{\pm}(\xi)=\int_{-\infty}^{\infty}g_{2m+1,\alpha,\Delta}^{\pm}(x)\,e^{-2\pi ix\xi}\,\dx,
		\]
		are continuous functions supported on the interval $[-\Delta,\Delta]$ and satisfy
		\begin{align}\label{Feb23_3:4pm}
		\big|\widehat{g}^{\pm}_{2m+1,\alpha,\Delta}(\xi)\big|\ll_m 1.
		\end{align}
		
		\smallskip
		
		\item[(vi)] The $L^1-$distances of $g_{2m+1,\alpha,\Delta}^{\pm}$ to $f_{2m+1,\alpha}$ are explicitly given by
		\begin{equation}\label{Lem8_EV_int_maj}
		\int_{-\infty}^{\infty}\big\{g^{+}_{2m+1,\alpha,\Delta}(x)-f_{2m+1,\alpha}(x)\big\}\,\dx=
		-\dfrac{1}{\Delta}\int_{\alpha}^{3/2}\left(\sigma-\alpha\right)^{2m}\,\log\left(\dfrac{1-e^{-2\pi(\sigma-\frac12)\Delta}}{1-e^{-2\pi\Delta}}\right)\d\sigma,
		\end{equation}
		and
		\begin{equation}\label{Lem8_EV_int_min}
		\int_{-\infty}^{\infty}\big\{f_{2m+1,\alpha}(x)-g^{-}_{2m+1,\alpha,\Delta}(x)\big\}\,\dx=\dfrac{1}{\Delta}\int_{\alpha}^{3/2}\left(\sigma-\alpha\right)^{2m}\,\log\left(\dfrac{1+e^{-2\pi(\sigma-\frac12)\Delta}}{1+e^{-2\pi\Delta}}\right)\d\sigma.
		\end{equation}
		
		\smallskip
		
		\item[(vii)] At $\xi=0$ we have
		\begin{align}\label{Lem8_FTg+_zero}
		\widehat{g}^{\pm}_{2m+1,\alpha,\Delta}(0)=\dfrac{\pi \left(\tfrac{3}{2}-\alpha\right)^{2m+2}}{(2m+1)(2m+2)}-\dfrac{1}{\Delta}\int_{\alpha}^{3/2}(\sigma-\alpha)^{2m}\,\log\left(\dfrac{1\mp e^{-2\pi(\sigma-\frac12)\Delta}}{1\mp e^{-2\pi\Delta}}\right)\d\sigma.
		\end{align}
		
		\smallskip
		
\item[(viii)]	The Fourier transforms $\widehat{g}^{\pm}_{2m+1,\alpha,\Delta}$ are even functions and, for $0 < \xi < \Delta$, we have the explicit expressions	
		\begin{align}\label{FT_maj_general_case}
		\begin{split}
		  &\widehat{g}^{\pm}_{2m+1,\alpha,\Delta}(\xi)  =\\
		 &   \ \ \ \dfrac{1}{2}\sum_{k=-\infty}^{\infty}(\pm 1)^k \left[\frac{k+1}{|\xi+k\Delta|}\Bigg(\frac{(2m)!\,e^{-2\pi |\xi+k\Delta|(\alpha-\frac12)}}{(2\pi|\xi+k\Delta|)^{2m+1}}- \!\!\sum_{j=0}^{2m+1}\frac{\gamma_{j}\,e^{-2\pi|\xi+k\Delta|}}{(2\pi |\xi+k\Delta|)^{j}}\left(\tfrac{3}{2}-\alpha\right)^{2m+1-j}\Bigg)\right]\,,
		 \end{split}
		\end{align}
		\smallskip
		where $\gamma_{j}=\frac{(2m)!}{(2m+1-j)!}$, for $0\leq j \leq 2m+1$. 
	\end{itemize}
\end{lemma}

\begin{proof} Fix $m \geq0 $ and $\hh \leq \alpha \leq 1$. For $\Delta\geq 1$ we consider the nonnegative Borel measure $\nu_{\Delta} = \nu_{2m+1, \alpha, \Delta} $ on $(0,\infty)$ given by
$$
\d\nu_{\Delta}(\lambda):=\int_{\alpha}^{3/2}\left(\sigma-\alpha\right)^{2m}\left(\frac{e^{-\pi\lambda(\sigma-\frac12)^{2} \Delta^2}-e^{-\pi\lambda\Delta^2}}{2\lambda}\right)\,\d\sigma\, \d\lambda\,,
$$
and let $F_{\Delta}=F_{2m+1,\alpha,\Delta}$ be the function 
\[
F_\Delta(x):=\int_{0}^{\infty}e^{-\pi\lambda x^2}\,\d\nu_{\Delta}(\lambda).
\]
Recall that 
\[
\dfrac{1}{2}\log\left(\dfrac{x^2+\Delta^2}{x^2+(\sigma-\frac12)^2\Delta^2}\right)=\int_{0}^{\infty}e^{-\pi\lambda x^2}\left(\dfrac{e^{-\pi\lambda(\sigma-\frac12)^{2} \Delta^2}-e^{-\pi\lambda\Delta^2}}{2\lambda}\right)\,\d\lambda.
\]
Multiplying both sides by $(\sigma-\alpha)^{2m}$ and integrating from $\sigma=\alpha$ to $\sigma=\tfrac32$ yields
\begin{align}\label{relation_F_Delta_PSF}
\begin{split}
\dfrac{1}{2} \int_{\alpha}^{3/2} &\left(\sigma-\alpha\right)^{2m}\log\Bigg(\frac{x^2+\Delta^2}{x^2+(\sigma-\frac12)^2\Delta^2}\Bigg)\,\d\sigma \\
& =  \int_{\alpha}^{3/2}\int_{0}^{\infty}\left(\sigma-\alpha\right)^{2m}e^{-\pi\lambda x^2}\Bigg(\dfrac{e^{-\pi\lambda(\sigma-\frac12)^{2} \Delta^2}-e^{-\pi\lambda\Delta^2}}{2\lambda}\Bigg)\,\d\lambda \,\d\sigma \\
& =\int_{0}^{\infty}e^{-\pi\lambda x^2}\int_{\alpha}^{3/2}\left(\sigma-\alpha\right)^{2m}\Bigg(\dfrac{e^{-\pi\lambda(\sigma-\frac12)^{2} \Delta^2}-e^{-\pi\lambda\Delta^2}}{2\lambda}\Bigg)\, \d\sigma \,\d\lambda\\
& =F_{\Delta}(x), 
\end{split}
\end{align}
where the interchange of the integrals is justified since the terms involved are all nonnegative. It follows from \eqref{Def_f_2m+1} that 
\begin{align}\label{relation_f_F}
f_{2m+1,\alpha}(x)=F_{\Delta}(\Delta x).
\end{align}
In particular, this shows that the measure $\nu_{\Delta}$ is finite on $(0,\infty)$ since
\begin{align*}
\int_{0}^{\infty}\d\nu_{\Delta}(\lambda)=F_{\Delta}(0)=f_{2m+1,\alpha}(0). 
\end{align*}

\smallskip

From the general Gaussian subordination framework of \cite[Section 11]{CLV}, there is a unique extremal majorant $G^{+}_{\Delta}(z)=G^{+}_{2m+1,\alpha,\Delta}(z)$ and a unique extremal minorant $G^{-}_{\Delta}(z)=G^{-}_{2m+1,\alpha, \Delta}(z)$ of exponential type $2\pi$ for $F_{\Delta}(x)$, and these functions are given by
\begin{align}\label{Def_G+_general2m+1}
G^{+}_{\Delta}(z)=\bigg(\dfrac{\sin\pi z}{\pi}\bigg)^2\left\{\displaystyle\sum_{n=-\infty}^{\infty}\dfrac{F_{\Delta}(n)}{(z-n)^2}+\displaystyle\sum_{n\neq 0}\dfrac{F^{'}_{\Delta}(n)}{(z-n)}\right\}
\end{align}
and
\begin{align}\label{Def_G-_general2m+1}
G^{-}_{\Delta}(z)=\bigg(\dfrac{\cos\pi z}{\pi}\bigg)^2 \left\{\displaystyle\sum_{n=-\infty}^{\infty}\dfrac{F_{\Delta}\big(n-\frac{1}{2}\big)}{\big(z-n+\frac{1}{2}\big)^2}+\dfrac{F^{'}_{\Delta}\big(n-\frac{1}{2}\big)}{\big(z-n+\frac{1}{2}\big)}\right\}.
\end{align}
Hence, the functions  $g^{+}_{\Delta}(z)=g^{+}_{2m+1,\alpha, \Delta}(z)$ and $g^{-}_{\Delta}(z)=g^{-}_{2m+1,\alpha, \Delta}(z)$ defined by
\begin{align}\label{relation_g_G}
g^{+}_{\Delta}(z):=G^{+}_{\Delta}(\Delta z) \hspace{0.3cm} \mbox{and} \hspace{0.3cm} g^{-}_{\Delta}(z):=G^{-}_{\Delta}(\Delta z) 
\end{align} 
are the unique extremal functions of exponential type $2\pi\Delta$ for $f_{2m+1,\alpha}$, as described in (i), (ii) and (iii). From \eqref{Def_G+_general2m+1} and \eqref{Def_G-_general2m+1} it is clear that $G^{\pm}_{\Delta}$, and hence $g^{\pm}_{\Delta}$, are even functions. 

\smallskip

We now verify the properties (iv) - (viii).

\subsubsection*{Property {\rm(iv)}} For $\alpha = \hh$, the function $f_{2m+1, \frac12}$ was already used in \cite{CChi} in connection to bounds for $S_{2m+1}$ in the critical line and is explicitly given by
\begin{equation*}
f_{2m+1, \frac12}(x)=\dfrac{1}{(2m+1)}\left[(-1)^{m+1}x^{2m+1}\arctan\left(\frac{1}{x}\right) + \sum_{k=0}^{m}\dfrac{(-1)^{m-k}}{2k+1}x^{2m-2k}\right].
\end{equation*}
Note that $f_{2m+1,\frac12}$ and $f'_{2m+1,\frac12}$ are bounded functions with power series representations 
\begin{align*}
f_{2m+1,\frac12}(x)=\frac{1}{2m+1}\sum_{k=1}^{\infty}\dfrac{(-1)^{k-1}}{(2k+2m+1)x^{2k}} 
\end{align*}
and
\begin{align*}
f'_{2m+1,\frac12}(x)=\dfrac{1}{2m+1}\displaystyle\sum_{k=1}^{\infty}\dfrac{(-1)^{k}(2k)}{(2k+2m+1)x^{2k+1}},
\end{align*}
for $|x| >1$. In particular, this implies that 
\begin{align}\label{Est_f_crit_line_12}
\big|f_{2m+1,\frac12}(x)\big| \ll_m \frac{1}{1+x^2}   \ \ \ \ {\rm and} \ \ \ \  \big|f'_{2m+1,\frac12}(x)\big| \ll_m \frac{1}{|x|(1+x^2)}.
\end{align}
On the other hand, directly from the definition \eqref{Def_f_2m+1} we see that 
\begin{equation}\label{comparison_f_12_alpha}
0 \leq f_{2m+1,\alpha}(x) \leq f_{2m+1,\frac12}(x) \ \ \ {\rm and} \ \ \ 0 \leq \big|f'_{2m+1,\alpha}(x)\big| \leq \big|f'_{2m+1,\frac12}(x)\big|
\end{equation}
for all $x \in \R$ and $\hh \leq \alpha \leq 1$. From \eqref{Est_f_crit_line_12} and \eqref{comparison_f_12_alpha} it follows that 
\begin{align*}
\big|f_{2m+1,\alpha}(x)\big| \ll_m \frac{1}{1+x^2}   \ \ \ \ {\rm and} \ \ \ \  \big|f'_{2m+1,\alpha}(x)\big| \ll_m \frac{1}{|x|(1+x^2)}
\end{align*}
(note that the implicit constants do not depend on $\alpha$). It then follows from \eqref{relation_f_F} that (recall the shorthand notation $F_{\Delta}=F_{2m+1,\alpha,\Delta}$)
\begin{align}\label{bounds for big F}
\big|F_{\Delta}(x)\big| \ll_m \frac{\Delta^2}{\Delta^2+x^2}   \ \ \ \ {\rm and} \ \ \ \  \big|F'_{\Delta}(x)\big| \ll_m \frac{\Delta^2}{|x|(\Delta^2+x^2)}.
\end{align}

Expressions \eqref{Def_G+_general2m+1} and \eqref{Def_G-_general2m+1} can be rewritten as
\begin{equation}\label{G+_rewritten}
G^{+}_{\Delta}(z)= \bigg(\dfrac{\sin\pi z }{\pi z }\bigg)^2  F_{\Delta}(0) + \displaystyle\sum_{n\neq 0}\bigg(\dfrac{\sin\pi (z - n)}{\pi (z - n)}\bigg)^2  \left\{F_{\Delta}(n) + (z-n)F^{'}_{\Delta}(n)\right\}
\end{equation}
and
\begin{equation}\label{G-_rewritten}
G^{-}_{\Delta}(z)= \displaystyle\sum_{n=-\infty}^{\infty} \bigg(\dfrac{\sin\pi (z - n + \tfrac12)}{\pi (z - n + \tfrac12)}\bigg)^2  \left\{F_{\Delta}\big(n-\tfrac{1}{2}\big) + (z-n+\tfrac{1}{2}\big)F^{'}_{\Delta}\big(n-\tfrac{1}{2}\big)\right\}.
\end{equation}
We now use \eqref{bounds for big F}, \eqref{G+_rewritten}, \eqref{G-_rewritten} and the uniform bound
\begin{align}\label{unif_bound_sin}
\left|\frac{\sin \pi z}{ \pi z}\right|^2 \ll \frac{e^{2\pi |\im z|}}{1 + |z|^2}
\end{align}
to get
\begin{align*}
\left|G^{\pm}_{\Delta}(z)\right| \ll_m \frac{\Delta^2e^{2 \pi |\im z|}}{1 + |z|}.
\end{align*}
One can break the sums in \eqref{G+_rewritten} and \eqref{G-_rewritten} into the ranges $\{n \leq |z|/2\}$, $\{|z|/2 < n \leq 2|z|\}$ and $\{2|z| < n\}$ to verify this last claim. From \eqref{relation_g_G} we arrive at \eqref{bound_g_2m+1_complex}.

\smallskip

To bound the functions $G^{\pm}_{\Delta}$ on the real line, we explore the fact that $F_{\Delta}$ is an even function (and hence $F_{\Delta}'$ is odd) to group the terms conveniently. For the majorant we group the terms $n$ and $-n$ in \eqref{G+_rewritten} to get
\begin{align}\label{G^+_bound_real}
\begin{split}
G^{+}_{\Delta}(x)  =\bigg(&\dfrac{\sin\pi x}{\pi x}\bigg)^2  F_{\Delta}(0) +  \displaystyle\sum_{n=1}^{\infty} \bigg(\dfrac{\sin^2\pi (x - n)}{\pi^2 (x^2 - n^2)^2}\bigg) \Big\{(2x^2 + 2n^2)F_{\Delta}(n)  + (x^2 - n^2) \,2 n\, F^{'}_{\Delta}(n)\Big\}\,,
\end{split}
\end{align}
and it follows from \eqref{bounds for big F} and \eqref{unif_bound_sin} that 
\begin{align}\label{Pf_Lem5_eq_1_G+}
\big|G^{+}_{\Delta}(x)\big| \ll_m \frac{\Delta^2}{\Delta^2 + x^2}. 
\end{align}
Again, it may be useful to split the sum in \eqref{G^+_bound_real} into the ranges $\{n \leq |x|/2\}$, $\{|x|/2 < n \leq 2|x|\}$ and $\{2|x| < n\}$ to verify this last claim. The bound
\begin{align}\label{Pf_Lem5_eq_1_G-}
\big|G^{-}_{\Delta}(x)\big| \ll_m \frac{\Delta^2}{\Delta^2 + x^2}
\end{align} 
follows in an analogous way, grouping the terms $n$ and $1-n$ (for $n \geq 1$) in \eqref{G-_rewritten}. From \eqref{relation_g_G}, \eqref{Pf_Lem5_eq_1_G+} and \eqref{Pf_Lem5_eq_1_G-} we arrive at \eqref{bound_g_2m+1_real}.

\subsubsection*{Property {\rm(v)}} Since $g_{2m+1,\alpha,\Delta}^{\pm}$ are entire functions of exponential type $2\pi\Delta$ whose restrictions to $\R$ are integrable, it follows from the Paley-Wiener theorem that their Fourier transforms are continuous functions supported on the interval $[-\Delta,\Delta]$. Moreover, from the uniform bounds \eqref{bound_g_2m+1_real} we see that 
\begin{equation*}
		\big|\widehat{g}_{2m+1,\alpha,\Delta}^{\pm}(\xi)\big|\leq \int_{-\infty}^{\infty}\big|g_{2m+1,\alpha, \Delta}^{\pm}(x)\big|\,\dx \ll_m 1.
\end{equation*}

\subsubsection*{Properties {\rm(vi) {\it and} (vii)}} From \eqref{Pf_Lem5_eq_1_G+}, \eqref{Pf_Lem5_eq_1_G-}, and the fact that the Fourier transforms $\widehat{G}^{\pm}_{\Delta}$ are supported on $[-1,1]$, we may apply the Poisson summation formula pointwise to $G^{\pm}_{\Delta}$. Recalling that $G^{+}_{\Delta}$ interpolates the values of $F_{\Delta}$ at $\Z$, we use \eqref{relation_F_Delta_PSF} to derive that
\begin{align}\label{Eval_PSF_maj_G_log}
\begin{split}
\widehat{G}^{+}_{\Delta}(0)  = \sum_{n = -\infty}^{\infty} G^{+}_{\Delta}(n) & = \sum_{n = -\infty}^{\infty} F_{\Delta}(n) \\
& = \frac{1}{2} \int_{\alpha}^{3/2} \left(\sigma-\alpha\right)^{2m} \sum_{n = -\infty}^{\infty} \log\left( \frac{n^2 + \Delta^2}{n^2 + (\sigma - \frac12)^2\Delta^2}\right)\d\sigma\\
& = \frac{1}{2} \int_{\alpha}^{3/2} \left(\sigma-\alpha\right)^{2m} \left(2\pi \Delta(\tfrac32 - \sigma) - 2 \log\left(\frac{1 - e^{-2\pi (\sigma - \frac12)\Delta}}{1 - e^{-2\pi \Delta}}\right)\right)\d\sigma\\
& = \dfrac{\pi\Delta \left(\tfrac{3}{2}-\alpha\right)^{2m+2}}{(2m+1)(2m+2)}- \int_{\alpha}^{3/2}(\sigma-\alpha)^{2m}\,\log\left(\dfrac{1-e^{-2\pi(\sigma-\frac12)\Delta}}{1-e^{-2\pi\Delta}}\right)\d\sigma.
\end{split}
\end{align}
Above we have used the fact that, for $b \geq a >0$ (see, for instance, \cite[\S 4.2.1]{CC})
\begin{equation*}
\sum_{n = -\infty}^{\infty} \log\left( \frac{n^2 + b^2}{n^2 +a^2}\right) = 2\pi(b-a) - 2 \log\left(\dfrac{1-e^{-2\pi a}}{1-e^{-2\pi b}}\right).
\end{equation*}
One can prove this directly regarding both sides as a function of the variable $b$, observing that they agree when $b =a$, and showing that they have the same derivative. 

\smallskip

We proceed analogously for the minorant
\begin{align}\label{Eval_PSF_min_G_log}
\widehat{G}^{-}_{\Delta}(0)  = \sum_{n = -\infty}^{\infty} G^{-}_{\Delta}(n) & = \sum_{n = -\infty}^{\infty} F_{\Delta}(n+\hh) \nonumber \\
& = \frac{1}{2} \int_{\alpha}^{3/2} \left(\sigma-\alpha\right)^{2m} \sum_{n = -\infty}^{\infty} \log\left( \frac{(n+\frac12)^2 + \Delta^2}{(n+\frac12)^2 + (\sigma - \frac12)^2\Delta^2}\right)\d\sigma \nonumber\\
& = \frac{1}{2} \int_{\alpha}^{3/2} \left(\sigma-\alpha\right)^{2m} \left(2\pi \Delta(\tfrac32 - \sigma) - 2 \log\left(\frac{1 + e^{-2\pi (\sigma - \frac12)\Delta}}{1 + e^{-2\pi \Delta}}\right)\right)\d\sigma\\
& = \dfrac{\pi\Delta \left(\tfrac{3}{2}-\alpha\right)^{2m+2}}{(2m+1)(2m+2)}- \int_{\alpha}^{3/2}(\sigma-\alpha)^{2m}\,\log\left(\dfrac{1+ e^{-2\pi(\sigma-\frac12)\Delta}}{1 + e^{-2\pi\Delta}}\right)\d\sigma, \nonumber
\end{align}
now using the fact that, for $b \geq a >0$ (see \cite[\S 4.1.2]{CC})
\begin{equation*}
\sum_{n = -\infty}^{\infty} \log\left( \frac{(n+\frac12)^2 + b^2}{(n+\frac12)^2 +a^2}\right) = 2\pi(b-a) - 2 \log\left(\dfrac{1+e^{-2\pi a}}{1+e^{-2\pi b}}\right).
\end{equation*}
From \eqref{Eval_PSF_maj_G_log}, \eqref{Eval_PSF_min_G_log} and the dilation relation 
\begin{equation}\label{dilation_relation}
\widehat{g}^{\pm}_{\Delta} (\xi) = \frac{1}{\Delta} \widehat{G}^{\pm}_{\Delta} \left(\frac{\xi}{\Delta}\right),
\end{equation}
we arrive at \eqref{Lem8_FTg+_zero}. 

\smallskip

Using the fact that (see, for instance, \cite[\S 2.733 Eq.1] {GR})
\begin{equation*}
\int_{-\infty}^{\infty} f_{2m+1,\alpha}(x)\,\dx = \dfrac{\pi \left(\tfrac{3}{2}-\alpha\right)^{2m+2}}{(2m+1)(2m+2)},
\end{equation*}
we arrive at \eqref{Lem8_EV_int_maj} and \eqref{Lem8_EV_int_min} from \eqref{Lem8_FTg+_zero}.

\subsubsection*{Property {\rm(viii)}} From relation \eqref{dilation_relation} it suffices to find the explicit form of $\widehat{G}^{\pm}_{\Delta}(\xi)$ for $-1 \leq \xi \leq 1$. Since $\widehat{G}^{\pm}_{\Delta}(\xi)$ are even functions, we only need to consider the case $0 < \xi \leq 1$ (recall that the values at $\xi =0$ were computed in the proof of property (vii)).

\smallskip

We consider first the majorant. Recall that $G^{+}_{\Delta}(k) = F_{\Delta}(k)$ for all $k \in \Z$ and $(G^{+}_{\Delta})'(k) = F'_{\Delta}(k)$ for all $k \in \Z\setminus \{0\}$. Note also that $(G^{+}_{\Delta})'(0) = 0$, since $G^{+}_{\Delta}$ is an even function, and that $F'_{\Delta}(0) = 0$ except in the case $\alpha = \hh$ and $m=0$, for which $F_{\Delta}$ is not differentiable at $x=0$. Our starting point is a result of Vaaler \cite[Theorem 9]{V} that gives us
\begin{align}
\widehat{G}^{+}_{\Delta}(\xi) & = \left(1 - |\xi| \right)\sum_{k=-\infty}^{\infty}  G^{+}_{\Delta}(k) \,e^{-2\pi i k \xi} +  \frac{1}{2\pi i} \,\sgn(\xi) \sum_{k=-\infty}^{\infty}(G^{+}_{\Delta})'(k) \,e^{-2\pi i k \xi}\nonumber \\
& = \left(1 - |\xi| \right)\sum_{k=-\infty}^{\infty}  F_{\Delta}(k) \,e^{-2\pi i k \xi} +  \frac{1}{2\pi i} \,\sgn(\xi) \sum_{k\neq0}F'_{\Delta}(k) \,e^{-2\pi i k \xi}. \label{Step0_int_1_FT_maj}
\end{align}
Using \eqref{relation_F_Delta_PSF}, the first sum in \eqref{Step0_int_1_FT_maj} is given by
\begin{align}
\sum_{k=-\infty}^{\infty}  F_{\Delta}(k) \,e^{-2\pi i k \xi}  & = \sum_{k=-\infty}^{\infty} \left(\frac{1}{2} \int_{\alpha}^{3/2} \left(\sigma-\alpha\right)^{2m}\log\Bigg(\frac{k^2+\Delta^2}{k^2+(\sigma-\frac12)^2\Delta^2}\Bigg)\,\d\sigma\right)\,e^{-2\pi i k \xi}\nonumber \\
& = \frac{1}{2} \int_{\alpha}^{3/2} \left(\sigma-\alpha\right)^{2m} \left(\sum_{k=-\infty}^{\infty}\log\Bigg(\frac{k^2+\Delta^2}{k^2+(\sigma-\frac12)^2\Delta^2}\Bigg)e^{-2\pi i k \xi}\right)\d\sigma,\label{int_1_maj_FT_G}
\end{align}
where the use of Fubini's theorem is justified by the absolute convergence of the sum on the left-hand side (which follows by \eqref{bounds for big F}). The inner sum in \eqref{int_1_maj_FT_G} can be evaluated via Poisson summation applied to the Fourier transform pair
\begin{equation}\label{PS_pair_1}
h(x) = \log\Bigg(\frac{x^2+b^2}{x^2+a^2}\Bigg) \ \ \ {\rm and} \ \ \ \widehat{h}(\xi) = \frac{e^{-2\pi |\xi| a } - e^{-2\pi  |\xi| b}}{|\xi|}
\end{equation}
for real numbers $b \geq a >0$ (see \cite[\S 4.1.2]{CC}). We then arrive at
\begin{align}\label{Step2_int_1_FT_maj}
\begin{split}
\sum_{k=-\infty}^{\infty}  F_{\Delta}(k) \,e^{-2\pi i k \xi} & = \frac{1}{2} \int_{\alpha}^{3/2} \left(\sigma-\alpha\right)^{2m} \left( \sum_{k = -\infty}^{\infty}  \frac{e^{-2 \pi |\xi + k| (\sigma-\frac12)\Delta} - e^{-2 \pi |\xi + k| \Delta}}{|\xi + k|}\right)\,\d\sigma\\
& = \frac{1}{2} \sum_{k = -\infty}^{\infty}  \int_{\alpha}^{3/2} \left(\sigma-\alpha\right)^{2m} \left( \frac{e^{-2 \pi |\xi + k| (\sigma-\frac12)\Delta} - e^{-2 \pi |\xi + k| \Delta}}{|\xi + k|}\right)\d\sigma.
\end{split}
\end{align}
We shall use the following indefinite integral \cite[\S 2.321]{GR} in our computations 
\begin{align}\label{Int_GR_exp_poli}
\int x^n \,e^{-ax}\,\dx = -e^{-ax} \left( \sum_{\ell=0}^n \frac{\ell!\, \binom{n}{\ell}}{a^{\ell+1}}\,x^{n-\ell}\right).
\end{align}
Using \eqref{Int_GR_exp_poli} in \eqref{Step2_int_1_FT_maj} we get
\begin{align}
& \sum_{k=-\infty}^{\infty}  F_{\Delta}(k) \,e^{-2\pi i k \xi} \nonumber \\
& = \frac{1}{2} \sum_{k = -\infty}^{\infty} \frac{e^{-2\pi |\xi + k| (\alpha - \frac12)\Delta}}{|\xi + k|}\left( \frac{(2m)!}{(2\pi |\xi + k|\Delta)^{2m+1}} -e^{-2\pi |\xi + k| (\frac32 - \alpha)\Delta}\sum_{\ell=0}^{2m} \frac{\ell! \binom{2m}{\ell}}{(2\pi |\xi + k|\Delta)^{\ell+1}}\left(\tfrac32 - \alpha\right)^{2m-\ell} \right) \nonumber  \\
&  \ \ \ \ \ \ \ \ - \frac{1}{2} \sum_{k = -\infty}^{\infty} \frac{e^{-2\pi |\xi + k| \Delta}}{(2m+1) |\xi + k|} {\left( \tfrac32 - \alpha\right)^{2m+1}} \nonumber  \\
& = \frac{1}{2} \sum_{k = -\infty}^{\infty} \frac{1}{|\xi + k|}\left[ \frac{(2m)!\,\,e^{-2\pi |\xi + k| (\alpha - \frac12)\Delta}}{(2\pi|\xi + k|\Delta)^{2m+1}} -  \sum_{j=0}^{2m+1} \frac{\gamma_j \,e^{-2\pi |\xi + k| \Delta}}{(2\pi |\xi+k|\Delta)^{j}}\left(\tfrac32 - \alpha\right)^{2m+1-j}\right]\,, \label{Final_com_FT_G+_p1}
\end{align}
with $\gamma_{j}=\frac{(2m)!}{(2m+1-j)!}$, for $0\leq j \leq 2m+1$.

\smallskip

We now evaluate the second sum in \eqref{Step0_int_1_FT_maj}. Using \eqref{relation_F_Delta_PSF} we have
\begin{align}
\sum_{k\neq0}  F_{\Delta}'(k) \,e^{-2\pi i k \xi}  & = \sum_{k\neq0} \left(\int_{\alpha}^{3/2} \left(\sigma-\alpha\right)^{2m}\left( \frac{k}{k^2+\Delta^2} - \frac{k}{k^2+(\sigma-\frac12)^2\Delta^2}\right)\,\d\sigma\right)\,e^{-2\pi i k \xi} \nonumber \\
& = \int_{\alpha}^{3/2} \left(\sigma-\alpha\right)^{2m} \left(\sum_{k=-\infty}^{\infty}\left( \frac{k}{k^2+\Delta^2} - \frac{k}{k^2+(\sigma-\frac12)^2\Delta^2}\right)\,e^{-2\pi i k \xi}\right)\d\sigma, \label{int_2_maj_FT_G}
\end{align}
where the use of Fubini's theorem is again justified by the absolute convergence of the sum on the left-hand side, which again follows by \eqref{bounds for big F}. The inner sum in \eqref{int_2_maj_FT_G} can be evaluated via Poisson summation applied to the Fourier transform pair
\begin{equation}\label{PS_pair_2}
h(x) =   \frac{x}{x^2+a^2} - \frac{x}{x^2+b^2} \ \ \ {\rm and} \ \ \ \widehat{h}(\xi) = - \pi i \,\sgn(\xi) \left(e^{-2\pi |\xi|a} - e^{-2\pi |\xi| b}\right)
\end{equation}
for real numbers $b\geq a >0$ (see \cite[\S 4.1.2]{CC}).
We then arrive at the expression
\begin{align*}
\sum_{k\neq0}  F_{\Delta}'(k) \,e^{-2\pi i k \xi}  & = \pi i \int_{\alpha}^{3/2} \left(\sigma-\alpha\right)^{2m} \left( \sum_{k=-\infty}^{\infty} \sgn(\xi + k) \left( e^{-2\pi (\sigma - \frac12)|\xi + k|\Delta} - e^{-2\pi |\xi + k|\Delta}\right)\right)\,\d\sigma\\
& =  \pi i  \sum_{k=-\infty}^{\infty}  \sgn(\xi + k) \int_{\alpha}^{3/2} \left(\sigma-\alpha\right)^{2m} \left( e^{-2\pi (\sigma - \frac12)|\xi + k|\Delta} - e^{-2\pi |\xi + k|\Delta}\right)\d\sigma.
\end{align*}
The latter use of Fubini's theorem can be justified by the absolute convergence of the double integral (one can explicitly sum the exponentials in geometric progressions). In the case $\alpha = \hh$ and $m=0$ one has to be a bit more careful and group the terms $k$ and $-k-1$, for $k \geq 0$, to have convergence. Using \eqref{Int_GR_exp_poli} we get
\begin{align}
&\sum_{k\neq0}  F_{\Delta}'(k) \,e^{-2\pi i k \xi}  \nonumber \\
& = \pi i  \sum_{k=-\infty}^{\infty}  \sgn(\xi + k) \left( \frac{(2m)!\,e^{-2\pi |\xi + k| (\alpha - \frac12)\Delta}}{(2\pi |\xi + k|\Delta)^{2m+1}} -e^{-2\pi |\xi + k| \Delta}\sum_{\ell=0}^{2m} \frac{\ell! \binom{2m}{\ell}}{(2\pi |\xi + k|\Delta)^{\ell+1}}\left(\tfrac32 - \alpha\right)^{2m-\ell} \right)\nonumber \\
&  \ \ \ \ \ \ \ \ - \pi i \sum_{k = -\infty}^{\infty} \sgn(\xi + k)\,\frac{e^{-2\pi |\xi + k| \Delta}}{(2m+1)} {\left( \tfrac32 - \alpha\right)^{2m+1}}\nonumber \\
& = \pi i \sum_{k=-\infty}^{\infty}  \sgn(\xi + k) \left[ \left( \frac{(2m)!\,e^{-2\pi |\xi + k| (\alpha - \frac12)\Delta}}{(2\pi |\xi + k|\Delta)^{2m+1}}\right) -   \sum_{j=0}^{2m+1} \frac{\gamma_j \,e^{-2\pi |\xi + k| \Delta}}{(2\pi |\xi + k|\Delta)^{j}}\left(\tfrac32 - \alpha\right)^{2m+1-j}\right]\,,\label{Final_com_FT_G+_p2}
\end{align}
with $\gamma_{j}=\frac{(2m)!}{(2m+1-j)!}$, for $0\leq j \leq 2m+1$.

\smallskip

From \eqref{Step0_int_1_FT_maj}, \eqref{Final_com_FT_G+_p1} and \eqref{Final_com_FT_G+_p2} we find, for $0 < \xi \leq 1$, that
\begin{align*}
\widehat{G}^{+}_{\Delta}(\xi) =  \frac{1}{2} \sum_{k = -\infty}^{\infty} \frac{k+1}{|\xi + k|}\left[ \frac{(2m)!\,\,e^{-2\pi |\xi + k| (\alpha - \frac12)\Delta}}{(2\pi|\xi + k|\Delta)^{2m+1}} -  \sum_{j=0}^{2m+1} \frac{\gamma_j \,e^{-2\pi |\xi + k| \Delta}}{(2\pi |\xi+k|\Delta)^{j}}\left(\tfrac32 - \alpha\right)^{2m+1-j}\right].
\end{align*}
The change of variables \eqref{dilation_relation} leads us directly to the expression \eqref{FT_maj_general_case} for the majorant.

\smallskip

The proof for the minorant follows along the same lines, starting with Vaaler's relation \cite[Theorem 9]{V} and the fact that $G^{-}_{\Delta}(k+\hh) = F_{\Delta}(k+\hh)$ and $(G^{-}_{\Delta})'(k+\hh) = F'_{\Delta}(k+\hh)$ for all $k \in \Z$, we have
\begin{align*}
\widehat{G}_{\Delta}^{-}(\xi) & = (1 - |\xi|) \sum_{k=-\infty}^{\infty} G_{\Delta}^{-}(k + \hh) \,e^{-2\pi i (k+\frac12) \xi}+ \frac{1}{2\pi i} \,\sgn(\xi) \sum_{k=-\infty}^{\infty} (G_{\Delta}^{-})'(k + \hh)\,e^{-2\pi i (k+\frac12) \xi}\\
& = (1 - |\xi|) \sum_{k=-\infty}^{\infty} F_{\Delta}(k + \hh) \,e^{-2\pi i (k+\frac12) \xi}+ \frac{1}{2\pi i} \,\sgn(\xi) \sum_{k=-\infty}^{\infty} F'_{\Delta}(k + \hh)\,e^{-2\pi i (k+\frac12) \xi}.
\end{align*}
One now uses Poisson summation with the pairs \eqref{PS_pair_1} and \eqref{PS_pair_2} to derive that
\begin{align*}
\sum_{k=-\infty}^{\infty}  F_{\Delta}(k + \hh) \,e^{-2\pi i (k+\frac12) \xi}  & = \frac{1}{2} \int_{\alpha}^{3/2} \left(\sigma-\alpha\right)^{2m} \left(\sum_{k=-\infty}^{\infty}\log\Bigg(\frac{(k+\hh)^2+\Delta^2}{(k+\hh)^2+(\sigma-\frac12)^2\Delta^2}\Bigg)e^{-2\pi i (k+\frac12) \xi}\right)\d\sigma\\
& = \frac{1}{2} \int_{\alpha}^{3/2} \left(\sigma-\alpha\right)^{2m} \left( \sum_{k = -\infty}^{\infty} (-1)^k\, \frac{e^{-2 \pi |\xi + k| (\sigma-\frac12)\Delta} - e^{-2 \pi |\xi + k| \Delta}}{|\xi + k|}\right)\d\sigma
\end{align*}
and 
\begin{align*}
\sum_{k=-\infty}^{\infty} &  F_{\Delta}'(k+\hh) \,e^{-2\pi i (k+\frac12) \xi}  \\
& = \int_{\alpha}^{3/2} \left(\sigma-\alpha\right)^{2m} \left(\sum_{k=-\infty}^{\infty}\left( \frac{(k+\hh)}{(k+\hh)^2+\Delta^2} - \frac{(k+\hh)}{(k+\hh)^2+(\sigma-\frac12)^2\Delta^2}\right)\,e^{-2\pi i (k+\frac12) \xi}\right)\d\sigma\\
& = \pi i \int_{\alpha}^{3/2} \left(\sigma-\alpha\right)^{2m} \left( \sum_{k=-\infty}^{\infty} (-1)^k \sgn(\xi + k) \left( e^{-2\pi (\sigma - \frac12)|\xi + k|\Delta} - e^{-2\pi |\xi + k|\Delta}\right)\right)\,\d\sigma.
\end{align*}
The remaining computations are analogous to the majorant case. This concludes the proof of the lemma.

\end{proof}

\section{The sum over prime powers}

The idea for our proof of Theorem \ref{Thm2}, in the case of odd $n$, is to replace the functions $f_{n,\alpha}$ in our representation lemma (Lemma \ref{Rep_lem}) by appropriate majorants and minorants, apply the Guinand-Weil explicit formula (Lemma \ref{GW}), and then asymptotically evaluate the resulting terms. Our majorants and minorants of exponential type $2\pi \Delta$, denoted here by $m_{\Delta}^{\pm}$, are even functions, and hence the resulting sum over prime powers will appear as
$$\frac{1}{\pi}\sum_{n\geq 2}\dfrac{\Lambda(n)}{\sqrt{n}} \,\widehat{m}^{\pm}_{\Delta}\left(\dfrac{\log n}{2\pi}\right) \cos(t \log n).$$
The purpose of this section is provide a detailed qualitative study of this expression. In order to ease the flow of the proofs below, we collect several auxiliary calculus and number theory facts in two appendices at the end of the paper.

\subsection{The case of the Poisson kernel $f_{-1,\alpha}$} Recall that in Lemma \ref{lemma_extr_Poisson} we denoted the Poisson kernel by $h_{\beta}(x) := f_{-1,\alpha}(x) = \frac{\beta}{\beta^2 + x^2}$, by introducing the parameter $\beta = \alpha - \frac12$.

\begin{lemma} [Sum over prime powers I] \label{Lem_SOPP_I} Assume the Riemann hypothesis.
Let $0 < \beta < \frac12$ and $\Delta \geq 1$, and let $m^{\pm}_{\Delta} = m^{\pm}_{\beta,\Delta}$ be the extremal functions for the Poisson kernel obtained in Lemma \ref{lemma_extr_Poisson}. Then
\begin{align}\label{Lem8_prime_sum_maj_Poisson_NEW}
\begin{split}
\frac{1}{\pi}\sum_{n\geq 2}\dfrac{\Lambda(n)}{\sqrt{n}} \,\widehat{m}^{+}_{\Delta}& \left(\dfrac{\log n}{2\pi}\right) \cos(t \log n)  \\
&  \geq - \dfrac{2\beta\,e^{(1-2\beta)\pi\Delta} - 2^{\frac{1}{2}-\beta} \big(\frac12 + \beta\big)^2 + 2^{\frac{1}{2}+\beta}e^{-4\pi\beta\Delta} \big(\frac12 - \beta\big)^2}{\big(\frac{1}{4}-\beta^2\big)\big(1 - e^{-2\pi\beta\Delta}\big)^2} + O\left(\frac{\Delta^4}{\beta}\right)
\end{split}
\end{align}
and
\begin{align}\label{Lem8_prime_sum_maj_Poisson_NEW_2}
\begin{split}
\frac{1}{\pi}\sum_{n\geq 2}\dfrac{\Lambda(n)}{\sqrt{n}} \,\widehat{m}^{-}_{\Delta}&\left(\dfrac{\log n}{2\pi}\right) \cos(t \log n)  \\
&\leq \dfrac{2\beta\,e^{(1-2\beta)\pi\Delta} - 2^{\frac{1}{2}-\beta} \big(\frac12 + \beta\big)^2 + 2^{\frac{1}{2}+\beta}e^{-4\pi\beta\Delta} \big(\frac12 - \beta\big)^2}{\big(\frac{1}{4}-\beta^2\big)\big(1+ e^{-2\pi\beta\Delta}\big)^2} + O\left(\beta \Delta^4\right).\end{split}
\end{align}

\end{lemma}
\begin{proof}Let $x = e^{2\pi \Delta}$ and note that the sums in \eqref{Lem8_prime_sum_maj_Poisson_NEW} and \eqref{Lem8_prime_sum_maj_Poisson_NEW_2}  only run for $2 \leq n \leq x$. Using the explicit description for the Fourier transforms $\widehat{m}_{\Delta}^{\pm}$ given by \eqref{FT_maj_min_Poisson} we get
\begin{align}\label{Poisson_lem_expres_antes}
\frac{1}{\pi}\sum_{n\geq 2}\dfrac{\Lambda(n)}{\sqrt{n}} \,\widehat{m}^{\pm}_{\Delta}\left(\dfrac{\log n}{2\pi}\right) \cos(t \log n) = \frac{e^{-2\pi \beta \Delta}}{\big(1 \mp e^{-2\pi \beta \Delta}\big)^2}\!\left(\sum_{n\leq e^{2\pi\Delta}}\dfrac{\Lambda(n)}{n^{1/2}}\,\bigg(\dfrac{e^{2\pi\beta\Delta}}{n^\beta}-\dfrac{n^\beta}{e^{2\pi\beta\Delta}}\bigg)\cos(t\log n)\!\right).
\end{align}
In the case of the majorant we use that $\cos(t\log n) \geq -1$ in \eqref{Poisson_lem_expres_antes}, together with Appendix {\bf B.4}, to get  \begin{align*}
\frac{1}{\pi}\sum_{n\geq 2} & \dfrac{\Lambda(n)}{\sqrt{n}} \,\widehat{m}^{+}_{\Delta}\left(\dfrac{\log n}{2\pi}\right) \cos(t \log n)  \geq - \frac{e^{-2\pi \beta \Delta}}{\big(1 -e^{-2\pi \beta \Delta}\big)^2}\sum_{n\leq e^{2\pi\Delta}}\dfrac{\Lambda(n)}{n^{1/2}}\,\bigg(\dfrac{e^{2\pi\beta\Delta}}{n^\beta}-\dfrac{n^\beta}{e^{2\pi\beta\Delta}}\bigg)\\
& = - \frac{e^{-2\pi \beta \Delta}}{\big(1 -e^{-2\pi \beta \Delta}\big)^2} \left( \frac{2 \beta e^{\pi \Delta} - 2^{\frac{1}{2}-\beta}e^{2\pi \beta \Delta} \big(\frac12 + \beta\big)^2 + 2^{\frac{1}{2}+\beta}e^{-2\pi\beta\Delta} \big(\frac12 - \beta\big)^2}{\frac{1}{4}-\beta^2} + O\left(\beta \,e^{2\pi \beta \Delta} \Delta^4 \right)\right)\\
& = - \dfrac{2\beta\,e^{(1-2\beta)\pi\Delta} - 2^{\frac{1}{2}-\beta} \big(\frac12 + \beta\big)^2 + 2^{\frac{1}{2}+\beta}e^{-4\pi\beta\Delta} \big(\frac12 - \beta\big)^2}{\big(\frac{1}{4}-\beta^2\big)\big(1 - e^{-2\pi\beta\Delta}\big)^2} + O\left(\frac{\Delta^4}{\beta}\right),
\end{align*}
where we have used the fact 
\begin{equation*}
\frac{1}{\big(1 -e^{-2\pi \beta \Delta}\big)^2} \leq \frac{1}{\big(1 -e^{-\beta}\big)^2} \ll \frac{1}{\beta^2}.
\end{equation*}
In the case of the minorant we use that $\cos(t\log n) \leq 1$ in \eqref{Poisson_lem_expres_antes}, together with Appendix {\bf B.4}, to get  
\begin{align*}
\frac{1}{\pi}\sum_{n\geq 2} & \dfrac{\Lambda(n)}{\sqrt{n}} \,\widehat{m}^{-}_{\Delta}\left(\dfrac{\log n}{2\pi}\right) \cos(t \log n)  \leq  \frac{e^{-2\pi \beta \Delta}}{\big(1 +e^{-2\pi \beta \Delta}\big)^2}\sum_{n\leq e^{2\pi\Delta}}\dfrac{\Lambda(n)}{n^{1/2}}\,\bigg(\dfrac{e^{2\pi\beta\Delta}}{n^\beta}-\dfrac{n^\beta}{e^{2\pi\beta\Delta}}\bigg)\\
&  =  \frac{e^{-2\pi \beta \Delta}}{\big(1 +e^{-2\pi \beta \Delta}\big)^2} \left( \frac{2 \beta e^{\pi \Delta} - 2^{\frac{1}{2}-\beta}e^{2\pi \beta \Delta} \big(\frac12 + \beta\big)^2 + 2^{\frac{1}{2}+\beta}e^{-2\pi\beta\Delta} \big(\frac12 - \beta\big)^2}{\frac{1}{4}-\beta^2} + O\left(\beta \,e^{2\pi \beta \Delta} \Delta^4 \right)\right)\\
& = \dfrac{2\beta\,e^{(1-2\beta)\pi\Delta} - 2^{\frac{1}{2}-\beta} \big(\frac12 + \beta\big)^2 + 2^{\frac{1}{2}+\beta}e^{-4\pi\beta\Delta} \big(\frac12 - \beta\big)^2}{\big(\frac{1}{4}-\beta^2\big)\big(1+ e^{-2\pi\beta\Delta}\big)^2} + O\left(\beta \Delta^4\right).
\end{align*}
This proves the lemma.
\end{proof}

\subsection{The case of $f_{2m+1,\alpha}$, for $m\geq 0$} We now consider the sum over prime powers applied to the extremal functions of exponential type $2\pi\Delta$ for the even functions $f_{2m+1,\alpha}$ defined in \eqref{Def_f_2m+1}. The next lemma collects the required bounds for our purposes.

\begin{lemma} [Sum over prime powers II] \label{Lem_SOPP_II} Assume the Riemann hypothesis.
Let $m \geq 0$, $\hh \leq \alpha < 1$ and $\Delta \geq 1$. Let $g^{\pm}_{\Delta} = g^{\pm}_{2m+1, \alpha,\Delta}$ be  the extremal functions for $f_{2m+1,\alpha}$ obtained in Lemma \ref{lema_extremal_f_2m+1_alpha}, and let $c>0$ be a given real number. In the region 
$$ \pi \Delta (1 - \alpha)^{2} \geq c$$
we have
\begin{align}\label{Lem8_prime_sum_maj}
\begin{split}
& \mp \frac{1}{\pi}\sum_{n\geq 2}\dfrac{\Lambda(n)}{\sqrt{n}} \,\widehat{g}^{\pm}_{\Delta}\left(\dfrac{\log n}{2\pi}\right) \cos(t \log n)  \leq \dfrac{(2\alpha-1)(2m)!}{\alpha(1-\alpha)}\dfrac{e^{(2-2\alpha)\pi\Delta}}{(2\pi\Delta)^{2m+2}} +  O_{m,c}\left(\dfrac{e^{(2-2\alpha)\pi\Delta}}{(1-\alpha)^{2}\Delta^{2m+3}}\right).
\end{split}
\end{align}

\end{lemma}
\begin{proof} Again we let $x = e^{2\pi \Delta}$ and note that the sum in \eqref{Lem8_prime_sum_maj} only runs for $2 \leq n \leq x$. Our idea is to explore the formula \eqref{FT_maj_general_case}. First observe that, for $0 < \xi < \Delta$, we have
\begin{align}\label{Pf_Lemma8_sum_primes_eq1}
\sum_{k\neq 0} \frac{|k+1|}{|\xi+k\Delta|} \sum_{j=0}^{2m+1}\frac{\gamma_{j}\,e^{-2\pi|\xi+k\Delta|}}{(2\pi |\xi+k\Delta|)^{j}}\left(\tfrac{3}{2}-\alpha\right)^{2m+1-j} \ll_m e^{-2\pi \Delta}.
\end{align}
Using \eqref{FT_maj_general_case}, \eqref{Pf_Lemma8_sum_primes_eq1} and the prime number theorem (it suffices to use the weaker estimate $\sum_{n \leq x} \frac{\Lambda(n)}{\sqrt{n}} \ll x^{1/2}$) we find that
\begin{align*}
 \mp \frac{1}{\pi}\sum_{n\geq 2}\dfrac{\Lambda(n)}{\sqrt{n}}& \,\widehat{g}^{\pm}_{\Delta}\left(\dfrac{\log n}{2\pi}\right) \cos(t \log n)  \\
&    = \mp (2m)!\sum_{n \leq x}\dfrac{\Lambda(n)}{\sqrt{n}} \left(\sum_{k=-\infty}^{\infty} \frac{(\pm 1)^k\,(k+1)\,e^{-|\log nx^k|(\alpha - \frac12)}}{|\log n x^k|^{2m+2}}\right) \cos (t \log n)  \\
&   \ \ \ \ \ \ \ \ \ \ \ \  \pm \sum_{j=0}^{2m+1}\gamma_{j}\left(\tfrac{3}{2}-\alpha\right)^{2m+1-j}\,\re \left(\sum_{n\leq x}\dfrac{\Lambda(n)}{n^{\frac{3}{2}+it}(\log n)^{j+1}}\right) + O_m\big(x^{-1/2}\big)\\
&    = \mp (2m)!\sum_{n \leq x}\dfrac{\Lambda(n)}{\sqrt{n}} \left(\sum_{k=-\infty}^{\infty} \frac{(\pm 1)^k \,(k+1)\,e^{-|\log nx^k|(\alpha - \frac12)}}{|\log n x^k|^{2m+2}}\right) \cos (t \log n)  + O_m(1).
\end{align*}
It is now convenient to split the inner sum in the ranges $k\geq 0$ and $k \leq -2$, and regroup them as
\begin{align*}
 & \mp \frac{1}{\pi}\sum_{n\geq 2}\dfrac{\Lambda(n)}{\sqrt{n}} \,\widehat{g}^{\pm}_{\Delta}\left(\dfrac{\log n}{2\pi}\right) \cos(t \log n)\\
 & = \mp (2m)!\!\!\sum_{n \leq x}\!\dfrac{\Lambda(n)}{\sqrt{n}} \sum_{k=0}^{\infty} (\pm 1)^k\!\left(\!\frac{k+1}{(\log nx^k)^{2m+2} \,(nx^k)^{\alpha - \frac12}} - \frac{k+1}{\big( \log \frac{x^{k+2}}{n}\big)^{2m+2} \big(\frac{x^{k+2}}{n}\big)^{\alpha - \frac12}}\!\right) \cos (t \log n) + O_m(1).
\end{align*} 
Using Appendix \textbf{B.3} and \eqref{Feb02_4:52pm}, we isolate the term $k=0$ and get
\begin{align}\label{Sum_PP_final_inequ}
\begin{split}
  \mp \frac{1}{\pi} & \sum_{n\geq 2}\dfrac{\Lambda(n)}{\sqrt{n}}  \,\widehat{g}^{\pm}_{\Delta}\left(\dfrac{\log n}{2\pi}\right) \cos(t \log n)\\
 & = \mp (2m)! \displaystyle\sum_{n\leq x} \left(\dfrac{\Lambda(n)}{n^\alpha(\log n)^{2m+2}}- \dfrac{\Lambda(n)}{x^{2\alpha-1}\,n^{1-\alpha}(2\log x-\log n)^{2m+2}}\right)\cos (t \log n) \\
& \ \ \ \ \ \ \ \ \ \ \ \ \ \ \ + O_{m,c}\left(\dfrac{x^{1-\alpha}}{(1-\alpha)^{2}(\log x)^{2m+3}}\right).
\end{split}
\end{align}
Observe that the terms
$$\dfrac{\Lambda(n)}{n^\alpha(\log n)^{2m+2}}- \dfrac{\Lambda(n)}{x^{2\alpha-1}\,n^{1-\alpha}(2\log x-\log n)^{2m+2}}$$
are all nonnegative for $n \leq x$, and we can get upper bounds in \eqref{Sum_PP_final_inequ} by just using the trivial inequality
\begin{equation}\label{cos_ineq}
-1 \leq \cos (t \log n) \leq 1.
\end{equation}
Estimate \eqref{Lem8_prime_sum_maj} plainly follows  from \eqref{Sum_PP_final_inequ}, \eqref{cos_ineq} and Appendices \textbf{B.1} and \textbf{B.2}.
\end{proof}

\section{Applying the Guinand-Weil formula}

In this section we prove Theorem \ref{Thm2} in the case of odd $n \geq -1$.

\subsection{The case $n=-1$} Here we keep the notation $\beta = \alpha - \frac12$, with $0 < \beta < \frac12$. To further simplify notation, let $m^{\pm}_{\Delta} = m^{\pm}_{\beta,\Delta}$ be the extremal functions for the Poisson kernel obtained in Lemma \ref{lemma_extr_Poisson}. From Lemma \ref{Rep_lem} and Lemma \ref{lemma_extr_Poisson} we have
\begin{align}\label{Poisson_before_app_GW}
& - \dfrac{1}{2\pi}\log\frac{t}{2\pi} + \dfrac{1}{\pi}\sum_{\gamma}m^{-}_{\Delta}(t-\gamma)+O\bigg(\frac{1}{t}\bigg) \leq S_{-1,\alpha}(t) \leq - \dfrac{1}{2\pi}\log\frac{t}{2\pi} + \dfrac{1}{\pi}\sum_{\gamma}m^{+}_{\Delta}(t-\gamma)+O\bigg(\frac{1}{t}\bigg).
\end{align}
For a fixed $t >0$, we consider the functions $\ell^{\pm}_{\Delta}(z):=m^{\pm}_{\Delta}(t-z)$. Then $\widehat{\ell}^{\pm}_{\Delta}(\xi)=\widehat{m}^{\pm}_{\Delta}(-\xi)e^{-2\pi i\xi t}$ and the condition $|\ell^{\pm}_{\Delta}(s)|\ll(1+|s|)^{-2}$ when $|\re{s}|\to\infty$ in the strip $|\im{s}|\leq 1$ follows from \eqref{decay_square_maj_Poisson}, \eqref{Complex_est_maj}, \eqref{Complex_est_min} and an application of the Phragm\'{e}n-Lindel\"{o}f principle. Recalling that $\widehat{m}^{\pm}_{\Delta}$ are even functions, we apply the Guinand-Weil explicit formula (Lemma \ref{GW}) and find that
\begin{align}\label{GW_applied_to_m}
\begin{split}
\displaystyle\sum_{\gamma}m^{\pm}_{\Delta}(t-\gamma) & = 
\Big\{m^{\pm}_{\Delta}\left(t-\tfrac{1}{2i}\right)+m^{\pm}_{\Delta}\left(t+\tfrac{1}{2i}\right)\Big\}-\dfrac{1}{2\pi}\widehat{m}^{\pm}_{\Delta}(0)\log\pi \\ 
     & \ \ \ \ + \dfrac{1}{2\pi}\int_{-\infty}^{\infty}m^{\pm}_{\Delta}(t-x)\,\re\,\dfrac{\Gamma'}{\Gamma}\bigg(\dfrac{1}{4}+\dfrac{ix}{2}\bigg) \,\dx - \frac{1}{\pi}\sum_{n\geq 2}\dfrac{\Lambda(n)}{\sqrt{n}} \,\widehat{m}^{\pm}_{\Delta}\left(\dfrac{\log n}{2\pi}\right) \cos(t \log n). 
\end{split}
\end{align}
We now proceed with an asymptotic analysis of each term on the right-hand side of \eqref{GW_applied_to_m}.

\subsubsection{First term} From \eqref{Complex_est_maj} and \eqref{Complex_est_min} we see that
\begin{align}\label{AsA_eq1}
\Big|m^{+}_{\Delta}\left(t-\tfrac{1}{2i}\right)+m^{+}_{\Delta}\left(t+\tfrac{1}{2i}\right)\Big| \ll \frac{\Delta^{2}e^{\pi\Delta}}{\beta(1+\Delta t)} 
\end{align}
and
\begin{align}\label{AsA_eq1_min}
\Big|m^{-}_{\Delta}\left(t-\tfrac{1}{2i}\right)+m^{-}_{\Delta}\left(t+\tfrac{1}{2i}\right)\Big| \ll \frac{\beta \Delta^{2}e^{\pi\Delta}}{1+\Delta t} .
\end{align}

\subsubsection{Second term} From \eqref{FT_maj_min_Poisson} it follows that
\begin{equation}\label{conclusion_FT_at_zero_Poisson_maj}
\widehat{m}^{+}_{\Delta}(0) = \pi\left( \frac{e^{\pi \beta \Delta} + e^{-\pi \beta \Delta} }{e^{\pi \beta \Delta} - e^{-\pi \beta \Delta} }\right)\ll \frac{1}{\beta}
\end{equation}
and
\begin{equation}\label{conclusion_FT_at_zero_Poisson_min}
\widehat{m}^{-}_{\Delta}(0) = \pi\left( \frac{e^{\pi \beta \Delta} - e^{-\pi \beta \Delta} }{e^{\pi \beta \Delta} + e^{-\pi \beta \Delta} }\right) \ll \min\{1, \beta \Delta\}.
\end{equation}

\subsubsection{Third term} Recall that the Poisson kernel $h_{\beta}(x) = \frac{\beta}{\beta^2 + x^2}$ defined in \eqref{Def_Poisson_kernel_beta} satisfies $\int_{-\infty}^{\infty} h_{\beta}(x)\,\dx=\pi $. Note also that for $0 < \beta \leq \frac12$ and $|x| \geq 1$ we have
\begin{equation}\label{h beta}
h_{\beta}(x) = \frac{\beta}{\beta^2 + x^2} \leq \frac{1}{1 + x^2}.
\end{equation}
Hence, from \eqref{decay_square_maj_Poisson}, we get
\begin{align}\label{aux_pf_Poisson_eq1}
\begin{split}
0 & \leq \int_{-\infty}^{\infty}m^{-}_{\Delta}(x)\,\log(2+|x|) \,\dx \\
& \leq \int_{-\infty}^{\infty} h_{\beta}(x) \log (2 + |x|)\,\dx  = \int_{-1}^{1}h_{\beta}(x) \log (2 + |x|) \,\dx + \int_{|x| \ge 1}h_{\beta}(x) \log (2 + |x|) \,\dx = O(1).
\end{split} 
\end{align}
From \eqref{conclusion_FT_at_zero_Poisson_min}, \eqref{h beta},  \eqref{aux_pf_Poisson_eq1}, and Stirling's formula it follows that
\begin{align}\label{Conclusion_poisson_log_sum_min}
\begin{split}
\dfrac{1}{2\pi}\int_{-\infty}^{\infty}m^{-}_{\Delta}(t-x)\,\re\,\dfrac{\Gamma'}{\Gamma}\bigg(\dfrac{1}{4}+\dfrac{ix}{2}\bigg) \,\dx  &= \dfrac{1}{2\pi}\int_{-\infty}^{\infty}m^{-}_{\Delta}(x)\big(\log t + O(\log(2+|x|))\big) \,\dx \\
& = \dfrac{\log t}{2} \left( \frac{e^{\pi \beta \Delta} - e^{-\pi \beta \Delta} }{e^{\pi \beta \Delta} + e^{-\pi \beta \Delta} }\right) + O(1).  \\
\end{split}
\end{align}
Similarly, using \eqref{decay_square_maj_Poisson} and \eqref{conclusion_FT_at_zero_Poisson_maj}, we have
\begin{align}\label{Conclusion_poisson_log_sum_maj}
\begin{split}
\dfrac{1}{2\pi}\int_{-\infty}^{\infty}m^{+}_{\Delta}(t-x)\,\re\,\dfrac{\Gamma'}{\Gamma}\bigg(\dfrac{1}{4}+\dfrac{ix}{2}\bigg) \,\dx  &= \dfrac{1}{2\pi}\int_{-\infty}^{\infty}m^{+}_{\Delta}(x)\big(\log t + O(\log(2+|x|))) \,\dx \\
& = \dfrac{\log t}{2} \left( \frac{e^{\pi \beta \Delta} + e^{-\pi \beta \Delta} }{e^{\pi \beta \Delta} - e^{-\pi \beta \Delta} }\right) + O\big(\tfrac{1}{\beta}\big).  \\
\end{split}
\end{align}

\subsubsection{Fourth term} This term was treated in Lemma \ref{Lem_SOPP_I}.

\subsubsection{Conclusion \textup{(}lower bound\textup{)}} Combining the estimates \eqref{Poisson_before_app_GW}, \eqref{GW_applied_to_m}, \eqref{AsA_eq1_min}, \eqref{conclusion_FT_at_zero_Poisson_min}, \eqref{Conclusion_poisson_log_sum_min}, and \eqref{Lem8_prime_sum_maj_Poisson_NEW_2} we derive that
\begin{align}\label{Con_Poisson_min_eq1}
\begin{split}
S_{-1,\alpha}(t) & \geq - \left[ \frac{\log t}{\pi} \left(\frac{e^{-2 \pi \beta \Delta}}{1 + e^{-2\pi \beta \Delta}}\right) + \dfrac{2\beta\,e^{(1-2\beta)\pi\Delta} - 2^{\frac{1}{2}-\beta} \big(\frac12 + \beta\big)^2 + 2^{\frac{1}{2}+\beta}e^{-4\pi\beta\Delta} \big(\frac12 - \beta\big)^2}{\pi\big(\frac{1}{4}-\beta^2\big)\big(1+ e^{-2\pi\beta\Delta}\big)^2}\right] \\
 & \ \ \ \ \ \ \ \ \ \ \ \ \ \ \ \ + O \left(\frac{\beta \Delta^{2}e^{\pi\Delta}}{1+\Delta t}\right) + O\big(\min\{1,\beta\Delta\}\big) + O\big(\beta \Delta^4\big).
 \end{split}
\end{align}
Note that in deducing \eqref{Con_Poisson_min_eq1}, the term $-(1/2\pi)\log t$ in \eqref{Poisson_before_app_GW} cancels with part of the leading term in \eqref{Conclusion_poisson_log_sum_min}. We now choose $\pi \Delta = \log \log t$ in \eqref{Con_Poisson_min_eq1}, which is essentially the optimal choice. Recalling that $\beta = \alpha - \frac12$, this choice yields
\begin{align}\label{Feb02_2:30pm}
\begin{split}
\!\!S_{-1,\alpha}(t) & \geq - \frac{(\log t)^{2 - 2 \alpha}}{\pi} \!\left(\! \frac{1}{\big( 1 \!+\! (\log t)^{1 - 2 \alpha}\big)} \!+\!  \frac{(2\alpha -1)}{\alpha (1- \alpha)\big( 1 \!+\! (\log t)^{1 - 2 \alpha}\big)^2}\!\right) \\
&  \ \ \ \ \ \ \ \  + \frac{2^{1-\alpha}\,\alpha^2 - 2^{\alpha}\,(1-\alpha)^2\,(\log t)^{2 -4\alpha}}{\pi \alpha (1- \alpha)\big( 1 \!+\! (\log t)^{1 - 2 \alpha}\big)^2} + O\big((\alpha - \tfrac12)(\log \log t)^4\big).\\
&  \geq - \frac{(\log t)^{2 - 2 \alpha}}{\pi} \!\left(\! \frac{1}{\big( 1 \!+\! (\log t)^{1 - 2 \alpha}\big)} \!+\!  \frac{(2\alpha -1)}{\alpha (1- \alpha)}\!\right) + O\big((\alpha - \tfrac12)(\log \log t)^4\big).
 \end{split}
\end{align}
In the last inequality we only dismissed nonnegative terms. Note the fact that $2^{1-\alpha}\,\alpha^2 \geq  2^{\alpha}\,(1-\alpha)^2$, for $\tfrac12 \leq \alpha \leq 1$. Finally, notice that in the range \eqref{Range_Ess} we may use \eqref{Feb02_4:52pm} to transform the error term of \eqref{Feb02_2:30pm} into the error term on the left-hand side of \eqref{Fev02_3:09pm}.

\subsubsection{Conclusion \textup{(}upper bound\textup{)}} Combining the estimates \eqref{Poisson_before_app_GW}, \eqref{GW_applied_to_m}, \eqref{AsA_eq1}, \eqref{conclusion_FT_at_zero_Poisson_maj}, \eqref{Conclusion_poisson_log_sum_maj}, and \eqref{Lem8_prime_sum_maj_Poisson_NEW} we derive that

\begin{align}\label{Con_Poisson_maj_eq1}
\begin{split}
S_{-1,\alpha}(t) & \leq \left[ \frac{\log t}{\pi} \left(\frac{e^{-2 \pi \beta \Delta}}{1 - e^{-2\pi \beta \Delta}}\right) + \dfrac{2\beta\,e^{(1-2\beta)\pi\Delta} - 2^{\frac{1}{2}-\beta} \big(\frac12 + \beta\big)^2 + 2^{\frac{1}{2}+\beta}e^{-4\pi\beta\Delta} \big(\frac12 - \beta\big)^2}{\pi\big(\frac{1}{4}-\beta^2\big)\big(1- e^{-2\pi\beta\Delta}\big)^2}\right] \\
 & \ \ \ \ \ \ \ \ \ \ \ \ \ \ \ \ + O \left(\frac{\Delta^{2}e^{\pi\Delta}}{\beta(1+\Delta t)}\right) + O\left(\frac{1}{\beta}\right) + O\left(\frac{\Delta^4}{\beta}\right).
 \end{split}
\end{align}
We now choose $\pi \Delta = \log \log t$ in \eqref{Con_Poisson_maj_eq1}, which again is essentially the optimal choice. Recalling that $\beta = \alpha - \frac12$, this yields
\begin{align}\label{Fev02_4:00pm}
\begin{split}
S_{-1,\alpha}(t) & \leq \frac{(\log t)^{2 - 2 \alpha}}{\pi} \left( \frac{1}{\big( 1 - (\log t)^{1 - 2 \alpha}\big)} +  \frac{(2\alpha -1)}{\alpha (1- \alpha)\big( 1 - (\log t)^{1 - 2 \alpha}\big)^2}\right) \\
&\ \ \ \ \ \ \ \  - \frac{\big(2^{1-\alpha}\,\alpha^2 - 2^{\alpha}\,(1-\alpha)^2\,(\log t)^{2 -4\alpha}\big)}{\pi \alpha (1- \alpha)\big( 1 \!-\! (\log t)^{1 - 2 \alpha}\big)^2} + O\left(\frac{(\log \log t)^4}{\alpha - \tfrac12}\right)\\
& \leq  \frac{(\log t)^{2 - 2 \alpha}}{\pi} \left( \frac{1}{\big( 1 - (\log t)^{1 - 2 \alpha}\big)} +  \frac{(2\alpha -1)}{\alpha (1- \alpha)\big( 1 - (\log t)^{1 - 2 \alpha}\big)^2}\right)  + O\left(\frac{(\log \log t)^4}{\alpha - \tfrac12}\right)\,,
 \end{split}
\end{align}
where we have just dismissed a nonpositive term in the last inequality. Observe that  
\begin{equation*}
\left|1 -  \frac{1}{\big( 1 - (\log t)^{1 - 2 \alpha}\big)^2}\right| \ll \frac{(\log t)^{1 - 2 \alpha}}{\big( 1 - (\log t)^{1 - 2 \alpha}\big)^2} \ll \frac{1}{(\alpha - \frac12)^2 (\log \log t)^2} \ll \frac{1}{(\alpha - \frac12)^2 (\log \log t)}.
\end{equation*}
Therefore we can rewrite \eqref{Fev02_4:00pm} as 
\begin{align}\label{New_1017}
\begin{split}
S_{-1,\alpha}(t) & \leq \frac{(\log t)^{2 - 2 \alpha}}{\pi} \!\left( \frac{1}{\big( 1 - (\log t)^{1 - 2 \alpha}\big)} +  \frac{(2\alpha -1)}{\alpha (1- \alpha)}\right) \\
&  \ \ \ \ \ \ \ \ \ \ \ \ \ \ \ \ \ + O\left(\!\dfrac{(\log t)^{2 - 2 \alpha}}{(\alpha - \frac12)(1 - \alpha)\log \log t}\!\right) + O\left(\frac{(\log \log t)^4}{\alpha - \tfrac12}\right).
 \end{split}
\end{align}
Again, in the range \eqref{Range_Ess} we may use \eqref{Feb02_4:52pm} to transform the error term of \eqref{New_1017} into the error term on the right-hand side of \eqref{Fev02_3:09pm}. This concludes the proof of the theorem in this case.

\subsection{The case $n \geq1$} Let $n = 2m +1$, with $m\geq 0$. For $\hh \leq \alpha < 1$ and $\Delta \geq 1$, let $g^{\pm}_{\Delta} = g^{\pm}_{2m+1, \alpha,\Delta}$ be  the extremal functions for $f_{2m+1,\alpha}$ obtained in Lemma \ref{lema_extremal_f_2m+1_alpha}.

\smallskip

Let us first consider the case where $m$ is even. In this case, from Lemma \ref{Rep_lem} and Lemma \ref{lema_extremal_f_2m+1_alpha} we have
\begin{align} \label{Est_final_S_n_odd_general}
\begin{split}
&\dfrac{1}{2\pi(2m+2)!}\left(\tfrac{3}{2}-\alpha\right)^{2m+2}\log t-
\dfrac{1}{\pi(2m)!} \sum_{\gamma} g^{+}_{\Delta}(t-\gamma) + O_m(1) \leq S_{2m+1,\alpha}(t) \\
&\ \ \ \ \ \ \ \ \ \ \ \ \ \leq \dfrac{1}{2\pi(2m+2)!}\left(\tfrac{3}{2}-\alpha\right)^{2m+2}\log t-
\dfrac{1}{\pi(2m)!} \sum_{\gamma} g^{-}_{\Delta}(t-\gamma) + O_m(1).
\end{split}
\end{align}
As observed in the case of the majorants for the Poisson kernel, it follows from \eqref{bound_g_2m+1_real}, \eqref{bound_g_2m+1_complex} and the Phragm\'{e}n-Lindel\"{o}f principle that we can then apply the Guinand-Weil explicit formula (Lemma \ref{GW}) to the functions $z\mapsto g^{\pm}_{\Delta}(t-z)$. This yields
\begin{align}\label{GW_applied_to_g}
\begin{split}
\displaystyle\sum_{\gamma}g^{\pm}_{\Delta}(t-\gamma) & = 
\Big\{g^{\pm}_{\Delta}\left(t-\tfrac{1}{2i}\right)+g^{\pm}_{\Delta}\left(t+\tfrac{1}{2i}\right)\Big\}-\dfrac{1}{2\pi}\widehat{g}^{\pm}_{\Delta}(0)\log\pi \\ 
     & \ \ \ \ + \dfrac{1}{2\pi}\int_{-\infty}^{\infty}g^{\pm}_{\Delta}(t-x)\,\re\,\dfrac{\Gamma'}{\Gamma}\bigg(\dfrac{1}{4}+\dfrac{ix}{2}\bigg) \,\dx - \frac{1}{\pi}\sum_{n\geq 2}\dfrac{\Lambda(n)}{\sqrt{n}} \,\widehat{g}^{\pm}_{\Delta}\left(\dfrac{\log n}{2\pi}\right) \cos(t \log n). 
\end{split}
\end{align}
We again proceed with an asymptotic analysis of each of the terms in the last expression.

\subsubsection{First term} The estimate \eqref{bound_g_2m+1_complex} implies that
\begin{align}\label{AsA_eq1_g}
\Big|g^{\pm}_{\Delta}\left(t-\tfrac{1}{2i}\right)+g^{\pm}_{\Delta}\left(t+\tfrac{1}{2i}\right)\Big| \ll_m \frac{\Delta^{2}e^{\pi\Delta}}{1+\Delta t}.
\end{align}
\subsubsection{Second term} From \eqref{Feb23_3:4pm}, it follows that
\begin{align}\label{FT_g_prep_final_comp}
\big|\widehat{g}^{\pm}_{\Delta}(0)\big|\ll_m 1.
\end{align}
\subsubsection{Third term} Using \eqref{bound_g_2m+1_real},  \eqref{Lem8_FTg+_zero}, and Stirling's formula we find that
\begin{align}\label{Conclusion_poisson_log_sum_g}
\begin{split}
& \dfrac{1}{2\pi}\int_{-\infty}^{\infty}g^{\pm}_{\Delta}(t-x)\,\re\,\dfrac{\Gamma'}{\Gamma}\bigg(\dfrac{1}{4}+\dfrac{ix}{2}\bigg) \,\dx  = \dfrac{1}{2\pi}\int_{-\infty}^{\infty}g^{\pm}_{\Delta}(x)\big(\log t + O(\log(2+|x|))\big) \,\dx \\
& = \dfrac{\log t}{2\pi} \left( \dfrac{\pi \left(\tfrac{3}{2}-\alpha\right)^{2m+2}}{(2m+1)(2m+2)}-\dfrac{1}{\Delta}\int_{\alpha}^{3/2}(\sigma-\alpha)^{2m}\,\log\left(\dfrac{1\mp e^{-2\pi(\sigma-\frac12)\Delta}}{1\mp e^{-2\pi\Delta}}\right)\d\sigma\right) + O_m(1).  \\
\end{split}
\end{align}

\subsubsection{Fourth term} This term was treated in Lemma \ref{Lem_SOPP_II}.

\subsubsection{Conclusion \textup{(}lower bound\textup{)}} We combine the leftmost inequality in \eqref{Est_final_S_n_odd_general} with estimates \eqref{GW_applied_to_g}, \eqref{AsA_eq1_g}, \eqref{FT_g_prep_final_comp}, \eqref{Conclusion_poisson_log_sum_g}, and \eqref{Lem8_prime_sum_maj} to get
\begin{align}\label{Conc_caso_geral_n_odd_eq1}
\begin{split}
S_{2m+1,\alpha}(t) & \geq \dfrac{\log t}{(2m)! \, 2\pi^2 \Delta } \int_{\alpha}^{3/2}(\sigma-\alpha)^{2m}\,\log\left(\dfrac{1- e^{-2\pi(\sigma-\frac12)\Delta}}{1- e^{-2\pi\Delta}}\right)\d\sigma - \dfrac{(2\alpha-1)}{\pi\alpha(1-\alpha)}\dfrac{e^{(2-2\alpha)\pi\Delta}}{(2\pi\Delta)^{2m+2}} \\
& \ \ \ \ \ \ \ \ \ \ \ \ \   + O_m(1) + O_m\left(\frac{\Delta^{2}e^{\pi\Delta}}{1+\Delta t}\right) + O_{m,c}\left(\dfrac{e^{(2-2\alpha)\pi\Delta}}{(1-\alpha)^{2}\Delta^{2m+3}}\right)\\
& \geq \dfrac{\log t}{(2m)! \, 2\pi^2 \Delta } \int_{\alpha}^{3/2}(\sigma-\alpha)^{2m}\,\log\big(1- e^{-2\pi(\sigma-\frac12)\Delta}\big)\,\d\sigma - \dfrac{(2\alpha-1)}{\pi\alpha(1-\alpha)}\dfrac{e^{(2-2\alpha)\pi\Delta}}{(2\pi\Delta)^{2m+2}} \\
& \ \ \ \ \ \ \ \ \ \ \ \ \   + O_m(1) + O_m\left(\frac{\Delta^{2}e^{\pi\Delta}}{1+\Delta t}\right) + O_{m,c}\left(\dfrac{e^{(2-2\alpha)\pi\Delta}}{(1-\alpha)^{2}\Delta^{2m+3}}\right).
\end{split}
\end{align}
Observe that 
\begin{align} \label{usar2}
\begin{split}
\left|\int_{3/2}^{\infty}(\sigma-\alpha)^{2m}\,\log\big(1\pm e^{-2\pi(\sigma-\frac12)\Delta}\big)\,\d\sigma\right|  & \ll  \int_{3/2}^{\infty}(\sigma-\hh)^{2m}\,e^{-2\pi(\sigma-\frac12)\Delta}\,\d\sigma  =  \int_{1}^{\infty}\sigma^{2m}e^{-2\sigma\pi\Delta}\,\d\sigma  \\
& \ll_m\dfrac{e^{-\pi\Delta}}{\Delta^{2m+2}} \leq\dfrac{e^{(1-2\alpha)\pi\Delta}}{\Delta^{2m+2}}.
\end{split} 
\end{align}
We now choose $\pi \Delta = \log \log t$. Using \eqref{usar2} and \eqref{Feb02_4:52pm} in \eqref{Conc_caso_geral_n_odd_eq1} leads us to
\begin{align}\label{Jan_26_6:37}
\begin{split}
S_{2m+1,\alpha}(t) &  \geq \dfrac{\log t}{(2m)! \, 2\pi^2 \Delta } \int_{\alpha}^{\infty}(\sigma-\alpha)^{2m}\,\log\big(1- e^{-2\pi(\sigma-\frac12)\Delta}\big)\,\d\sigma - \dfrac{(2\alpha-1)}{\pi\alpha(1-\alpha)}\dfrac{e^{(2-2\alpha)\pi\Delta}}{(2\pi\Delta)^{2m+2}} \\
& \ \ \ \ \ \ \ \ \ \ \ \ \   + O_{m,c}\left(\dfrac{e^{(2-2\alpha)\pi\Delta}}{(1-\alpha)^{2}\Delta^{2m+3}}\right).
\end{split}
\end{align}
From monotone convergence and \eqref{Int_GR_exp_poli} we have
\begin{align}\label{Jan_26_6:38}
\begin{split}
 \int_{\alpha}^{\infty}(\sigma-\alpha)^{2m}&\,\log\big(1- e^{-2\pi(\sigma-\frac12)\Delta}\big)\,\d\sigma  = - \int_{\alpha}^{\infty}(\sigma-\alpha)^{2m}\,\left(\sum_{k=1}^{\infty} \frac{e^{-2k\pi(\sigma-\frac12)\Delta}}{k}\right) \,\d\sigma\\
& = - \sum_{k=1}^{\infty} \frac{1}{k} \int_{\alpha}^{\infty}(\sigma-\alpha)^{2m}\,e^{-2k\pi(\sigma-\frac12)\Delta}\,\d\sigma\\
& = - \frac{(2m)!}{(2\pi\Delta)^{2m+1}}\sum_{k=1}^{\infty} \frac{e^{-2k\pi (\alpha - \frac12)\Delta}}{k^{2m+2}}.
\end{split}
\end{align}
Plugging \eqref{Jan_26_6:38} into \eqref{Jan_26_6:37} leads us to
\begin{align*}
S_{2m+1,\alpha}(t) &\geq  -\left(\frac{1}{2^{2m+2}  \,\pi} \right)\frac{ (\log t)^{2 - 2\alpha}}{(\log \log t)^{2m+2}} \left[ \sum_{k=0}^{\infty} \frac{1}{(k+1)^{2m+2}(\log t)^{(2\alpha -1)k}} + \frac{2\alpha -1}{\alpha(1 - \alpha)} \right] \\
& \ \ \ \ \ \ \ \ \ \ \ \ \   + O_{m,c}\left(\frac{ (\log t)^{2 - 2\alpha}}{(1-\alpha)^2\,(\log \log t)^{2m+3}}\right).
\end{align*}

\subsubsection{Conclusion \textup{(}upper bound\textup{)}} We combine the rightmost inequality in \eqref{Est_final_S_n_odd_general} with estimates \eqref{GW_applied_to_g}, \eqref{AsA_eq1_g}, \eqref{FT_g_prep_final_comp}, \eqref{Conclusion_poisson_log_sum_g}, and \eqref{Lem8_prime_sum_maj} to get
\begin{align}\label{Jan_26_7:03}
S_{2m+1,\alpha}(t) & \leq \dfrac{\log t}{(2m)! \, 2\pi^2 \Delta } \int_{\alpha}^{3/2}(\sigma-\alpha)^{2m}\,\log\left(\dfrac{1+ e^{-2\pi(\sigma-\frac12)\Delta}}{1+ e^{-2\pi\Delta}}\right)\d\sigma + \dfrac{(2\alpha-1)}{\pi\alpha(1-\alpha)}\dfrac{e^{(2-2\alpha)\pi\Delta}}{(2\pi\Delta)^{2m+2}} \nonumber \\
& \ \ \ \ \ \ \ \ \ \ \ \ \   + O_m(1) + O_m\left(\frac{\Delta^{2}e^{\pi\Delta}}{1+\Delta t}\right) + O_{m,c}\left(\dfrac{e^{(2-2\alpha)\pi\Delta}}{(1-\alpha)^{2}\Delta^{2m+3}}\right)\\
& \leq \dfrac{\log t}{(2m)! \, 2\pi^2 \Delta } \int_{\alpha}^{3/2}(\sigma-\alpha)^{2m}\,\log\big(1+ e^{-2\pi(\sigma-\frac12)\Delta}\big)\,\d\sigma + \dfrac{(2\alpha-1)}{\pi\alpha(1-\alpha)}\dfrac{e^{(2-2\alpha)\pi\Delta}}{(2\pi\Delta)^{2m+2}} \nonumber \\
& \ \ \ \ \ \ \ \ \ \ \ \ \   + O_m(1) + O_m\left(\frac{\Delta^{2}e^{\pi\Delta}}{1+\Delta t}\right) + O_{m,c}\left(\dfrac{e^{(2-2\alpha)\pi\Delta}}{(1-\alpha)^{2}\Delta^{2m+3}}\right).\nonumber
\end{align}
We now choose $\pi \Delta = \log \log t$. Using \eqref{usar2} and \eqref{Feb02_4:52pm} in \eqref{Jan_26_7:03} leads us to
\begin{align}\label{Jan_26_7:09}
\begin{split}
S_{2m+1,\alpha}(t) &  \leq \dfrac{\log t}{(2m)! \, 2\pi^2 \Delta } \int_{\alpha}^{\infty}(\sigma-\alpha)^{2m}\,\log\big(1+ e^{-2\pi(\sigma-\frac12)\Delta}\big)\,\d\sigma + \dfrac{(2\alpha-1)}{\pi\alpha(1-\alpha)}\dfrac{e^{(2-2\alpha)\pi\Delta}}{(2\pi\Delta)^{2m+2}} \\
& \ \ \ \ \ \ \ \ \ \ \ \ \   + O_{m,c}\left(\dfrac{e^{(2-2\alpha)\pi\Delta}}{(1-\alpha)^{2}\Delta^{2m+3}}\right).
\end{split}
\end{align}
As in \eqref{Jan_26_6:38}, now using dominated convergence, we have
\begin{align}\label{Jan_26_7:17}
 \int_{\alpha}^{\infty}(\sigma-\alpha)^{2m}\,\log\big(1+ e^{-2\pi(\sigma-\frac12)\Delta}\big)\,\d\sigma  =   \frac{(2m)!}{(2\pi\Delta)^{2m+1}}\sum_{k=1}^{\infty} \frac{(-1)^{k+1}\,e^{-2k\pi (\alpha - \frac12)\Delta}}{k^{2m+2}}.
\end{align}
Finally, plugging \eqref{Jan_26_7:17} into \eqref{Jan_26_7:09} gives us
\begin{align*}
\begin{split}
S_{2m+1,\alpha}(t) &\leq  \left(\frac{1}{2^{2m+2}  \,\pi} \right)\frac{ (\log t)^{2 - 2\alpha}}{(\log \log t)^{2m+2}} \left[ \sum_{k=0}^{\infty} \frac{(-1)^k}{(k+1)^{2m+2}(\log t)^{(2\alpha -1)k}} + \frac{2\alpha -1}{\alpha(1 - \alpha)} \right] \\
& \ \ \ \ \ \ \ \ \ \ \ \ \   + O_{m,c}\left(\frac{ (\log t)^{2 - 2\alpha}}{(1-\alpha)^2\,(\log \log t)^{2m+3}}\right).
\end{split}
\end{align*}

\subsubsection{Case of $m$ odd} \label{Sec5.2.7}
In the case of $m$ odd, the roles of the majorant $g_{\Delta}^+$ and minorant $g_{\Delta}^-$ must be interchanged due to the presence of the factor $(-1)^m$ in the representation lemma \eqref{Lem2_eq_representation_odd}. The remaining computations are exactly the same as in the case of $m$ even.

\smallskip

This concludes the proof of Theorem \ref{Thm2} in the case of odd $n$.

\section{Interpolation} \label{Sec_Int}
In this section we prove Theorem \ref{Thm2} in the case of even $n \geq 0$. Recall that for integer $j\geq 0$ we have defined 
$$H_j(x) =  \sum_{k=0}^{\infty} \frac{x^k}{(k+1)^j}\,,$$
and {\it for odd $n\geq -1$} we have defined 
\begin{equation}\label{Feb07_12:58pm}
C_{n,\alpha}^{\pm}(t) = \frac{1}{2^{n+1}\pi}\left( H_{n+1}\Big(\pm (-1)^{(n+1)/2}\,(\log t)^{1 - 2\alpha}\Big) + \frac{2\alpha -1}{\alpha (1- \alpha)}\right).
\end{equation}
Throughout this section let us write 
\[
\ell_{n,\alpha}(t):=\dfrac{(\log t)^{2-2\alpha}}{(\log\log t)^{n}} \quad \text{and} \quad  r_{n,\alpha}(t):=\dfrac{(\log t)^{2-2\alpha}}{(1-\alpha)^2(\log\log t)^{n}}. 
\]
\subsection{The case $n$ even with $n\geq 2$} \label{Subsec_not_Poisson_Feb01_11:31am}
Let $\frac12 \leq \alpha < 1$. In this subsection we show how to obtain the bounds for $S_{n,\alpha}(t)$ from the corresponding bounds for $S_{n-1,\alpha}(t)$ and $S_{n+1,\alpha}(t)$. This interpolation argument explores the smoothness of these functions via the mean value theorem in an optimal way. This extends the material that previously appeared in \cite[Section 4]{CChi}.

\smallskip

Let us consider here the case of {\it $n/2$ odd}. The case of {\it $n/2$ even} follows the exact same outline, with the roles of $C_{n,\alpha}^{+}(t)$ and $C_{n,\alpha}^{-}(t)$ interchanged. 

\smallskip 

Let $c >0$ be a given real number. In the region $(1-\alpha)^{2} \geq \frac{c/2}{\log\log t}$ we have already established that
\begin{align}  \label{interpol_2}
-C_{n+1,\alpha}^-(t)\,\ell_{n+2,\alpha}(t)+O_{n,c}(r_{n+3,\alpha}(t))\leq S_{n+1,\alpha}(t)\leq C_{n+1,\alpha}^+(t)\,\ell_{n+2,\alpha}(t)+O_{n,c}(r_{n+3,\alpha}(t)),
\end{align}
and 
\begin{align} \label{interpol_1}
-C_{n-1,\alpha}^-(t)\,\ell_{n,\alpha}(t)+O_{n,c}(r_{n+1,\alpha}(t))\leq S_{n-1,\alpha}(t)\leq C_{n-1,\alpha}^+(t)\,\ell_{n,\alpha}(t)+O_{n,c}(r_{n+1,\alpha}(t)).
\end{align}

\subsubsection{Error term estimates} Let $(\alpha,t)$ be such that $(1-\alpha)^{2}\geq\frac{c}{\log\log t}$. Observe that, in the set $\{(\alpha,\mu); \,t-1\leq\mu\leq t+1\}$, estimates \eqref{interpol_2} and \eqref{interpol_1} apply (note the use of $c/2$ instead of $c$ in the domains of these estimates). Then, by the mean value theorem and \eqref{interpol_1} we obtain, for $-1 \leq h \leq 1$, 
\begin{align}
\begin{split} \label{prep_integral_even}
S_{n,\alpha}(t) - &S_{n,\alpha}(t-h)  = h \, S_{n-1,\alpha}(t_h^*) \\
& \leq \big(\chi_{h>0}\,|h|\,C_{n-1,\alpha}^+(t_h^*)\,\ell_{n,\alpha}(t_h^*)+\chi_{h<0}\,|h|\,C_{n-1,\alpha}^-(t_h^*)\,\ell_{n,\alpha}(t_h^*)\big)+|h|\,O_{n,c}(r_{n+1,\alpha}(t_h^*))\\
& = \big(\chi_{h>0}\,|h|\,C_{n-1,\alpha}^+(t_h^*)\,\ell_{n,\alpha}(t_h^*)+\chi_{h<0}\,|h|\,C_{n-1,\alpha}^-(t_h^*)\,\ell_{n,\alpha}(t_h^*)\big)+|h|\,O_{n,c}(r_{n+1,\alpha}(t))\,,
\end{split}
\end{align}
where $t^*_h$ is a suitable point in the segment connecting $t-h$ and $t$, and $\chi_{h>0}$ and $\chi_{h<0}$ are the indicator functions of the sets $\{h \in \R; \,h>0\}$ and $\{h \in \R; \,h<0\}$, respectively. We would like to change $t_h^*$ by $t$ in the last line of \eqref{prep_integral_even}. For all $k\geq 0$ let us define
$$f_{k}(t)=\dfrac{1}{(\log t)^{(2\alpha-1)k}}\dfrac{(\log t)^{2-2\alpha}}{(\log\log t)^n}=\dfrac{(\log t)^{(k+1)(1-2\alpha)+1}}{(\log\log t)^n}.$$
We shall prove that 
\begin{equation}\label{Jan_29_12:03}
\big|C_{n-1,\alpha}^-(t_h^*)\,\ell_{n,\alpha}(t_h^*)-C_{n-1,\alpha}^-(t)\,\ell_{n,\alpha}(t)\big|\ll_n r_{n+1,\alpha}(t).
\end{equation}
Using the mean value theorem, we have that
\begin{align}\label{Jan_29_12:01}
\begin{split} 
\big|C_{n-1,\alpha}^-(t_h^*)\,\ell_{n,\alpha}(t_h^*)-C_{n-1,\alpha}^-(t)&\,\ell_{n,\alpha}(t)\big|\ll_n \frac{1}{(1 - \alpha)} \displaystyle\sum_{k=0}^{\infty}\dfrac{1}{(k+1)^n}\big|f_{k}(t_h^*)-f_{k}(t)\big|\\
&=  \frac{1}{(1 - \alpha)} |t_h^{*}-t|\displaystyle\sum_{k=0}^{\infty}\dfrac{1}{(k+1)^n}\big|f^{'}_{k}(t_{h,k}^{*})\big| \\
& \ll_n \frac{1}{(1 - \alpha)}  \displaystyle\sum_{k=0}^{\infty}\dfrac{((k+1)(2\alpha-1)+1)}{(k+1)^n\,t_{h,k}^{*}\,(\log t_{h,k}^{*}\,)^{(k+1)(2\alpha-1)}(\log\log t_{h,k}^{*})^n},
\end{split}
\end{align}
where, for each $k \geq 0$, $t_{h,k}^{*}$ is a point that belongs to the segment connecting $t_{h}^{*}$ and $t$. Observe now that 
\begin{align}\label{Jan_29_12:02}
\displaystyle\sum_{k=0}^{\infty} & \dfrac{((k+1)(2\alpha-1)+1)}{(k+1)^n\,t_{h,k}^{*}\,(\log t_{h,k}^{*}\,)^{(k+1)(2\alpha-1)}(\log\log t_{h,k}^{*})^n}  \ll_n \displaystyle\sum_{k=0}^{\infty}\dfrac{((k+1)(2\alpha-1)+1)}{(k+1)^n\,t\,(\log t_{h,k}^{*}\,)^{(k+1)(2\alpha-1)}(\log\log t)^n} \nonumber \\
&  \ \ \ \ \ \ \ \ \ \ \ \ \ \ \ \ \ \  \ \ \ \ \  \ll\frac{1}{t} \left[\displaystyle\sum_{k=0}^{\infty}\dfrac{2\alpha-1}{(k+1)^{n-1}(\log(t-1))^{(k+1)(2\alpha-1)}}\right] + \frac{1}{t}\\
&  \ \ \ \ \ \ \ \ \ \ \ \ \ \ \ \ \ \  \ \ \ \ \ \ll\frac{1}{t} \ \ll \  \ell_{n+1,\alpha}(t).  \nonumber 
\end{align}
From \eqref{Jan_29_12:01} and \eqref{Jan_29_12:02}, we arrive at \eqref{Jan_29_12:03}. In a similar way we observe that
\begin{equation}\label{Jan_30_12:07}
\big|C_{n-1,\alpha}^+(t_h^*)\,\ell_{n,\alpha}(t_h^*)-C_{n-1,\alpha}^+(t)\,\ell_{n,\alpha}(t)\big|\ll_n r_{n+1,\alpha}(t).
\end{equation}
From \eqref{prep_integral_even}, \eqref{Jan_29_12:03}, and \eqref{Jan_30_12:07} we obtain
\begin{align}\label{Jan_30_12:21}
\begin{split}
S_{n,\alpha}(t) - &S_{n,\alpha}(t-h)  \leq \big(\chi_{h>0}\,|h|\,C_{n-1,\alpha}^+(t)\,\ell_{n,\alpha}(t)+\chi_{h<0}\,|h|\,C_{n-1,\alpha}^-(t)\,\ell_{n,\alpha}(t)\big)+|h|\,O_{n,c}(r_{n+1,\alpha}(t)).
\end{split}
\end{align}

\subsubsection{Integrating and optimizing} Let $a:=a_{n,\alpha}(t)$ and $b:=b_{n,\alpha}(t)$ be real-valued functions, that shall be properly chosen later, satisfying $0 \leq a,b \leq 1$. In particular, we will be able to choose them in a way that $a + b = 1$ at the end. Let us just assume for now that $a + b \geq 1$ in the following argument. Let $\nu = \nu_{n,\alpha}(t)$ be a real-valued function such that $0 < \nu \leq 1$. For a fixed $t$, we integrate \eqref{Jan_30_12:21} with respect to the variable $h$ and find that
\begin{align*}
S_{n,\alpha}(t) & \leq \frac{1}{(a+b)\nu} \int_{-a\nu}^{b\nu} S_{n,\alpha}(t-h)\,\d h \\
&\ \ \ \ \ \ \ \ \ \ \ \ \  + \frac{1}{(a+b)\nu} \left[ \int_{-a\nu}^{b\nu}\big(\chi_{h>0}\,|h|\,C_{n-1,\alpha}^+(t)+\chi_{h<0}\,|h|\,C_{n-1,\alpha}^-(t)\big)\,\d h\right] \ell_{n,\alpha}(t)\\
&  \ \ \ \ \ \ \ \ \ \ \ \ \ \ \  \ \ \ \ \ \ + \frac{1}{(a+b)\nu} \left[\int_{-a\nu}^{b\nu} |h|\,\d h \right] O_{n,c}(r_{n+1,\alpha}(t))\\
& = \frac{1}{(a+b)\nu} \Big[ S_{n+1,\alpha}(t + a\nu) - S_{n+1,\alpha}(t - b\nu)\Big] \\
& \ \ \ \ \ \ \ \ \ \ \ \ \  + \left[\dfrac{b^2C_{n-1,\alpha}^+(t) + a^{2}C_{n-1,\alpha}^-(t)}{2(a+b)}\right]\nu \,\ell_{n,\alpha}(t) + O_{n,c}(\nu\,r_{n+1,\alpha}(t)).
\end{align*}
Using \eqref{interpol_2} and the same error term estimates as in \eqref{Jan_29_12:03} and \eqref{Jan_30_12:07} we derive that \begin{align}\label{Jan30_1:50}
\begin{split}
\!\!S_{n,\alpha}(t) & \!\leq\! \frac{1}{(a+b)\nu} \Big[C_{n+1,\alpha}^+(t\!+\!a\nu) \, \ell_{n+2,\alpha}(t\!+\!a\nu) \!+\! C_{n+1,\alpha}^-(t\!-\!b\nu)\, \ell_{n+2,\alpha}(t\!-\!b\nu) \!+\! O_{n,c}(r_{n+3,\alpha}(t\!+\!a\nu)) \\
& \ \ \ \ \ \ \ \ \ \ + O_{n,c}(r_{n+3,\alpha}(t-b\nu))\Big] +  \left[\dfrac{b^2C_{n-1,\alpha}^+(t) + a^{2}C_{n-1,\alpha}^-(t)}{2(a+b)}\right]\nu \,\ell_{n,\alpha}(t) + O_{n,c}(\nu\,r_{n+1,\alpha}(t)) \\
& = \left[\frac{C_{n+1,\alpha}^+(t) + C_{n+1,\alpha}^-(t)}{(a+b)}\right] \frac{1}{\nu}\,\,\ell_{n+2,\alpha}(t) + 
\left[\dfrac{b^2C_{n-1,\alpha}^+(t) + a^{2}C_{n-1,\alpha}^-(t)}{2(a+b)}\right]\nu \,\ell_{n,\alpha}(t) \\
&  \ \ \ \ \ \ \ \ \ \  + O_{n,c}\left(\frac{r_{n+3,\alpha}(t)}{\nu}\right) + O_{n,c}(\nu \,r_{n+1,\alpha}(t)).
\end{split}
\end{align}
Choosing $\nu = \frac{\lambda_{n,\alpha}(t)}{\log \log t}$ in \eqref{Jan30_1:50}, where $\lambda_{n,\alpha}(t)>0$ is a function to be determined (recall that we required $0 <\nu\leq 1$), we obtain
\begin{align*}
S_{n,\alpha}(t) & \leq \left\{ \left[\frac{C_{n+1,\alpha}^+(t) + C_{n+1,\alpha}^-(t)}{(a+b)}\right] \frac{1}{\lambda_{n,\alpha}(t)} + \left[\dfrac{b^2C_{n-1,\alpha}^+(t) + a^{2}C_{n-1,\alpha}^-(t)}{2(a+b)}\right]\lambda_{n,\alpha}(t)\right\} \,\ell_{n+1,\alpha}(t) \\
& \ \ \ \ \ + O_{n,c}\left(\dfrac{r_{n+2,\alpha}(t)}{\lambda_{n,\alpha}(t)}\right) + O_{n,c}(\lambda_{n,\alpha}(t)\,r_{n+2}(t)).
\end{align*}
We now choose $\lambda_{n,\alpha}(t)>0$ to minimize the expression in brackets, which corresponds to the choice
\begin{align}\label{Jan30_2:13pm}
\lambda_{n,\alpha}(t) = \left[\frac{C_{n+1,\alpha}^+(t) + C_{n+1,\alpha}^-(t)}{(a+b)}\right]^{1/2} \left[\dfrac{b^2C_{n-1,\alpha}^+(t) + a^{2}C_{n-1,\alpha}^-(t)}{2(a+b)}\right]^{-1/2}.
\end{align}
This leads to the bound
\begin{align}\label{Jan30_2:05pm}
\begin{split}
S_{n,\alpha}(t)& \leq 2  \left[\frac{\big(C_{n+1,\alpha}^+(t) + C_{n+1,\alpha}^-(t)\big)  \big(b^2C_{n-1,\alpha}^+(t) + a^{2}C_{n-1,\alpha}^-(t)\big)}{2(a+b)^2}\right]^{1/2}\! \ell_{n+1,\alpha}(t) \\
& \ \ \ \ \ \ \ \ \ \ \ \ \ + O_{n,c}\left(\dfrac{r_{n+2,\alpha}(t)}{\lambda_{n,\alpha}(t)}\right) + O_{n,c}\left(\lambda_{n,\alpha}(t)\,r_{n+2,\alpha}(t)\right).
\end{split}
\end{align}

We seek to minimize the expression in brackets on the right-hand side of \eqref{Jan30_2:05pm} in the variables $a$ and $b$. It is easy to see that it only depends on the ratio $a/b$. If we set $a = bx$, we must minimize the function
$$W(x) = 2  \left[\frac{\big(C_{n+1,\alpha}^+(t) + C_{n+1,\alpha}^-(t)\big)  \big(C_{n-1,\alpha}^+(t) + x^{2}C_{n-1,\alpha}^-(t)\big)}{2(x+1)^2}\right]^{1/2}.$$
Note that $C_{n-1,\alpha}^{\pm}(t)>0$ and $C_{n+1,\alpha}^{\pm}(t)>0$. Such a minimum is obtained when 
\begin{equation}\label{Jan30_2:14pm}
x = C_{n-1,\alpha}^+(t)/ C_{n-1,\alpha}^-(t),
\end{equation}
leading to the bound
\begin{align}  
\begin{split}\label{final_result_for_bound}
S_{n,\alpha}(t)& \leq \left[\frac{2 \big(C_{n+1,\alpha}^+(t) + C_{n+1,\alpha}^-(t)\big) \ C_{n-1,\alpha}^+(t)\,C_{n-1,\alpha}^-(t)}{C_{n-1,\alpha}^+(t)+C_{n-1,\alpha}^-(t)}\right]^{1/2} \!\ell_{n+1,\alpha}(t) \\
& \ \ \ \ \ \ \ \ \ \   + O_{n,c}\left(\dfrac{r_{n+2,\alpha}(t)}{\lambda_{n,\alpha}(t)}\right) + O_{n,c}\left(\lambda_{n,\alpha}(t)\,r_{n+2,\alpha}(t)\right).
\end{split}
\end{align}
We may now set $a+b=1$. From \eqref{Jan30_2:14pm} we then have the exact values of $a$ and $b$ and expression \eqref{Jan30_2:13pm} yields  
\begin{align*} 
\lambda_{n,\alpha}(t) & = \left[\dfrac{2\big(C_{n+1,\alpha}^+(t) + C_{n+1,\alpha}^-(t)\big)\big(C_{n-1,\alpha}^+(t) + C_{n-1,\alpha}^-(t)\big)}{C_{n-1,\alpha}^+(t)\,C_{n-1,\alpha}^-(t)}\right]^{1/2}.
\end{align*}

In the definition of $C_{n-1,\alpha}^\pm(t)$ and $C_{n+1,\alpha}^\pm(t)$, given by \eqref{Feb07_12:58pm}, we now use the bounds (for $j \geq 2$)
$$1\leq  H_j(x) \leq \zeta(j)$$
for $0 < x < 1$, and 
$$1 - \frac{1}{2^j} \leq H_j(x) \leq 1$$
for $-1< x <0$. Together with the fact that $n\geq 2$, after some computations one arrives at
$$\frac12 \leq \lambda_{n,\alpha}(t) \leq 2.$$
Therefore, if $\log \log t \geq 4$, we have $\nu = \frac{\lambda_{n,\alpha}(t)}{\log \log t} \leq 1$, as we had originally required. Finally, expression \eqref{final_result_for_bound} yields
\begin{align*}
S_{n,\alpha}(t)& \leq \left[\frac{2 \big(C_{n+1,\alpha}^+(t) + C_{n+1,\alpha}^-(t)\big) \ C_{n-1,\alpha}^+(t)\,C_{n-1,\alpha}^-(t)}{C_{n-1,\alpha}^+(t)+C_{n-1,\alpha}^-(t)}\right]^{1/2} \!\ell_{n+1,\alpha}(t) + O_{n,c}\left(r_{n+2,\alpha}(t)\right),
\end{align*}
which concludes the proof in this case. 

\smallskip

The argument for the lower bound of is $S_{n,\alpha}(t)$ is entirely symmetric. This completes the proof of Theorem \ref{Thm2}, when $n \geq 2$ is even. 

\subsection{The case $n = 0$} \label{Subsec_Poisson_Feb01_11:24am}

We now consider $\hh < \alpha < 1$. To treat the case $n=0$ we proceed with a variant of the method presented in \S \ref{Subsec_not_Poisson_Feb01_11:31am}, in which we only use the lower bound for $S_{-1,\alpha}(t)$ since this is stable under the limit $\alpha \to \hh^+$. 

\smallskip

Let $c >0$ be a given real number. In the region  $(1-\alpha)^{2} \geq \frac{c/2}{\log\log t}$ we have already shown that
\begin{align}\label{interpol_2_Feb07}
-C_{1,\alpha}^-(t)\,\ell_{2,\alpha}(t)+O_{c}(r_{3,\alpha}(t))\leq S_{1,\alpha}(t)\leq C_{1,\alpha}^+(t)\,\ell_{2,\alpha}(t)+O_{c}(r_{3,\alpha}(t)),
\end{align} 
and 
\begin{align} \label{interpol_1_Feb07}
-C_{-1,\alpha}^-(t)\,\ell_{0,\alpha}(t)+O_{c}(r_{1,\alpha}(t))\leq S_{-1,\alpha}(t).
\end{align}
Let $(\alpha,t)$ be such that $(1-\alpha)^{2}\geq\frac{c}{\log\log t}$. Observe that, in the set $\{(\alpha,\mu); \,t-1\leq\mu\leq t+1\}$, estimates \eqref{interpol_2_Feb07} and \eqref{interpol_1_Feb07} apply (note again the use of the constant $c/2$ instead of $c$ in the domains of these estimates). Then, by the mean value theorem and \eqref{interpol_1_Feb07} we obtain, for $0 \leq h \leq 1$, 
\begin{align}\label{Feb07_4:02pm}
\begin{split} 
S_{0,\alpha}(t) - S_{0,\alpha}(t-h)  = h\, S_{-1,\alpha}(t_h^*)  & \geq -h\,C_{-1,\alpha}^-(t_h^*)\,\ell_{0,\alpha}(t_h^*)+h\,O_c(r_{1,\alpha}(t_h^*)) \\
& = -h\,C_{-1,\alpha}^-(t_h^*)\,\ell_{0,\alpha}(t_h^*)+h\,O_c(r_{1,\alpha}(t)) ,
\end{split}
\end{align}
where $t^*_h$ is a suitable point in the segment connecting $t-h$ and $t$. From the explicit expression 
$$g(t) := C_{-1,\alpha}^-(t)\,\ell_{0,\alpha}(t) = \frac{1}{\pi}\left( \frac{1}{1 + (\log t)^{1 - 2\alpha}} + \frac{2\alpha -1}{\alpha(1 - \alpha)}\right) (\log t )^{2 - 2\alpha}$$
we observe directly that 
$$|g'(t)| \ll \frac{1}{t}$$
and hence, by the mean value theorem, that
\begin{equation}\label{Feb08_4:05pm}
\big|C_{-1,\alpha}^-(t)\,\ell_{0,\alpha}(t)-C_{-1,\alpha}^-(t_h^*)\,\ell_{0,\alpha}(t_h^*)\big|\ll r_{1,\alpha}(t).
\end{equation}
From \eqref{Feb07_4:02pm} and \eqref{Feb08_4:05pm} it follows that 
\begin{equation}\label{Feb08_4:09pm}
S_{0,\alpha}(t) - S_{0,\alpha}(t-h) \geq -h\,C_{-1,\alpha}^-(t)\,\ell_{0,\alpha}(t)+h\,O_c(r_{1,\alpha}(t)).
\end{equation}
Let $\nu = \nu_{\alpha}(t)$ be a real-valued function such that $0 < \nu \leq 1$. For a fixed $t$, we integrate \eqref{Feb08_4:09pm} with respect to the variable $h$ to get
\begin{align*}
S_{0,\alpha}(t) & \geq \frac{1}{\nu} \int_{0}^{\nu} S_{0,\alpha}(t-h)\,\d h  - \frac{1}{\nu} \left( \int_{0}^{\nu}h\,\d h\right)\,C_{-1,\alpha}^-(t)\,\ell_{0,\alpha}(t) + \frac{1}{\nu} \left(\int_{0}^{\nu} h\,\d h \right) O_c(r_{1,\alpha}(t))\\
& = \frac{1}{\nu} \big( S_{1, \alpha}(t) - S_{1,\alpha}(t - \nu)\big) - \dfrac{\nu}{2}\,C_{-1,\alpha}^-(t)\,\ell_{0,\alpha}(t) + O_c(\nu\,r_{1,\alpha}(t)).
\end{align*}
From \eqref{interpol_2_Feb07} we then get
\begin{align}\label{Feb08_4:36pm}
\begin{split}
S_{0,\alpha}(t) & \geq \frac{1}{\nu} \Big[\!-C_{1,\alpha}^-(t)\,\ell_{2,\alpha}(t)-C_{1,\alpha}^+(t-\nu)\,\ell_{2,\alpha}(t-\nu)+O_{c}(r_{3,\alpha}(t)) + O_{c}(r_{3,\alpha}(t-\nu))\Big]\\
&  \ \ \ \ \ \ \ \ \ \ - \dfrac{\nu}{2}\,C_{-1,\alpha}^-(t)\,\ell_{0,\alpha}(t) + O_c(\nu\,r_{1,\alpha}(t))\\
& = - \Big[C_{1,\alpha}^-(t) + C_{1,\alpha}^+(t)\Big]\,\frac{1}{\nu} \,\ell_{2,\alpha}(t) - \dfrac{\nu}{2}\,C_{-1,\alpha}^-(t)\,\ell_{0,\alpha}(t) + O_{c}\left(\frac{r_{3,\alpha}(t)}{\nu}\right) +  O_c(\nu\,r_{1,\alpha}(t)),
\end{split}
\end{align}
where we have used \eqref{Jan_30_12:07} in the last passage.

\smallskip

We now choose $\nu = \frac{\lambda_{\alpha}(t)}{\log \log t}$ in \eqref{Feb08_4:36pm}, where $\lambda_{\alpha}(t)>0$ is a function to be determined. This yields
\begin{align*}
\begin{split}
S_{0,\alpha}(t) & \geq - \left[\Big(C_{1,\alpha}^-(t) + C_{1,\alpha}^+(t)\Big)\,\frac{1}{\lambda_{\alpha}(t)} +  \dfrac{C_{-1,\alpha}^-(t)}{2}\,\lambda_{\alpha}(t)\right] \ell_{1,\alpha}(t) + O_{c}\left(\frac{r_{2,\alpha}(t)}{\lambda_{\alpha}(t)}\right) +  O_c(\lambda_{\alpha}(t)\,r_{2,\alpha}(t)).
\end{split}
\end{align*}
Choosing $\lambda_{\alpha}(t)$ in order to minimize the expression in brackets, we find that
\begin{equation}\label{Feb08_5:05pm}
\lambda_{\alpha}(t) = \left(\frac{2 \big(C_{1,\alpha}^-(t) + C_{1,\alpha}^+(t)\big)}{C_{-1,\alpha}^-(t)}\right)^{1/2}.
\end{equation}
This leads to the bound
\begin{equation}\label{Feb09_8:48am}
S_{0,\alpha}(t) \geq - \Big[2 \big(C_{1,\alpha}^-(t) + C_{1,\alpha}^+(t)\big)\,C_{-1,\alpha}^-(t)\Big]^{1/2}\, \ell_{1,\alpha}(t) + O_{c}\left(\frac{r_{2,\alpha}(t)}{\lambda_{\alpha}(t)}\right) +  O_c(\lambda_{\alpha}(t)\,r_{2,\alpha}(t)).
\end{equation}
Finally, using the trivial estimates
\begin{equation*}
\frac{1}{\pi} \left(\frac12 + \frac{2\alpha-1}{\alpha(1- \alpha)}\right)\leq C_{-1,\alpha}^-(t) \leq \frac{1}{\pi} \left(1 + \frac{2\alpha-1}{\alpha(1- \alpha)}\right),
\end{equation*}
\begin{equation*}
\frac{1}{4\pi} \left(1 + \frac{2\alpha-1}{\alpha(1- \alpha)}\right)\leq C_{1,\alpha}^-(t) \leq \frac{1}{4\pi} \left(\zeta(2) + \frac{2\alpha-1}{\alpha(1- \alpha)}\right),
\end{equation*}
and
\begin{equation*}
\frac{1}{4\pi} \left(\frac34 + \frac{2\alpha-1}{\alpha(1- \alpha)}\right)\leq C_{1,\alpha}^+(t) \leq \frac{1}{4\pi} \left(1 + \frac{2\alpha-1}{\alpha(1- \alpha)}\right),
\end{equation*}
one can show that $\lambda_{\alpha}(t)$ defined by \eqref{Feb08_5:05pm} verifies the inequalities
\begin{equation*}
\frac{1}{2} \leq \lambda_{\alpha}(t) \leq 2\,,
\end{equation*}
which shows that indeed $0< \nu \leq 1$ and allows us to write  \eqref{Feb09_8:48am} in our originally intended form of
\begin{equation*}
S_{0,\alpha}(t) \geq - \Big[2 \big(C_{1,\alpha}^-(t) + C_{1,\alpha}^+(t)\big)\,C_{-1,\alpha}^-(t)\Big]^{1/2}\, \ell_{1,\alpha}(t) + O_{c}(r_{2,\alpha}(t)).
\end{equation*}

The proof of the upper bound for $S_{0,\alpha}(t)$ follows along the same lines. Instead of \eqref{Feb07_4:02pm}, one would start with the following inequality, valid for $0 \leq h \leq 1$ and $ t_h^* \in [t, t+h]$, 
\begin{align*}
S_{0,\alpha}(t+h) - S_{0,\alpha}(t) = h\, S_{-1,\alpha}(t_h^*)  & \geq -h\,C_{-1,\alpha}^-(t_h^*)\,\ell_{0,\alpha}(t_h^*)+h\,O_c(r_{1,\alpha}(t_h^*)).
\end{align*}
This completes the proof of our Theorem \ref{Thm2}.

\section*{Appendix A: Calculus facts}

\subsection*{Prelude} Throughout these appendices we encounter the following setting in multiple situations:  let $c>0$ be a given real number and $\hh \leq \alpha < 1$ and $x \geq 3$ be such that 
\begin{equation}\label{Feb02_4:39pm}
(1-\alpha)^{2}\log x \geq c.
\end{equation}
Let us note that, if $0 \leq \theta_1 , \theta_2$ are real numbers, it follows from \eqref{Feb02_4:39pm} that 
\begin{equation}\label{Feb02_4:52pm}
(1-\alpha)^{\theta_1} \,(\log x)^{\theta_2} \ll_{c,\theta_1, \theta_2} x^{1-\alpha}.
\end{equation}
In fact, if $\theta _1 > \theta_2$ we simply observe that 
$$(1-\alpha)^{\theta_1} \,(\log x)^{\theta_2} \leq (1-\alpha)^{\theta_2} \,(\log x)^{\theta_2} \ll_{\theta_2} x^{1-\alpha}.$$
On the other hand, if $0 \leq \theta_1 \leq  \theta_2$, we let $\ell = \theta_2 - \theta_1\geq 0$ and $\eta = \theta_1 + 2\ell = \theta_2 + \ell$ to obtain
\begin{align*}
(1-\alpha)^{\theta_1} \,(\log x)^{\theta_2} \ll_{c,\theta_1, \theta_2} (1-\alpha)^{\theta_1} \,(\log x)^{\theta_2} \big((1-\alpha)^{2}\log x\big)^{\ell} = \big((1-\alpha) \log x\big)^{\eta} \ll_{\eta} x^{1-\alpha}.
\end{align*} 

We now proceed with the calculus facts required for our analysis.

\subsection*{A.1} {\it Let $c>0$ be a given real number and $m\geq 0$ be a given integer. For $\hh \leq \alpha < 1$ and $x \geq 3$ such that $(1-\alpha)^{2}\log x \geq c$, we have }
\begin{align*}
\displaystyle\int_{2}^{x}\dfrac{1}{t^{\alpha}(\log t)^{2m+2}}\,\d t =
\dfrac{x^{1-\alpha}}{(1-\alpha)(\log x)^{2m+2}}+ O_{m,c}\bigg(\dfrac{x^{1-\alpha}}{(1-\alpha)^{2}(\log x)^{2m+3}}\bigg).
\end{align*}   
\begin{proof}
Using integration by parts we get
\begin{align}\label{Feb10_1:22pm}
\displaystyle\int_{2}^{x}\dfrac{1}{t^{\alpha}(\log t)^{2m+2}}\,\d t = \dfrac{x^{1-\alpha}}{(1-\alpha)(\log x)^{2m+2}} - \dfrac{2^{1-\alpha}}{(1-\alpha)(\log 2)^{2m+2}} + \frac{(2m+2)}{(1- \alpha)} \int_{2}^{x} \frac{1}{t^{\alpha} \,(\log t)^{2m+3}}\,\dt.
\end{align}
From \eqref{Feb02_4:52pm} we have
\begin{equation}\label{Feb10_1:23pm}
\dfrac{2^{1-\alpha}}{(1-\alpha)(\log 2)^{2m+2}} \ll_m \frac{1}{(1-\alpha)} \ll_{m,c} \dfrac{x^{1-\alpha}}{(1-\alpha)^{2}(\log x)^{2m+3}},
\end{equation}
and
\begin{align}\label{Feb10_1:24pm}
\begin{split}
\int_{2}^{x} \frac{1}{t^{\alpha} \,(\log t)^{2m+3}}\,\dt & = \int_{2}^{x^{2/3}} \frac{1}{t^{\alpha} (\log t)^{2m+3}}\,\dt + \int_{x^{2/3}}^{x} \frac{1}{t^{\alpha} (\log t)^{2m+3}}\,\dt \\
& \leq \frac{1}{(\log 2)^{2m+3}} \int_{2}^{x^{2/3}} \frac{1}{t^{\alpha}}\,\dt + \frac{1}{(\log (x^{2/3}))^{2m+3}} \int_{x^{2/3}}^{x} \frac{1}{t^{\alpha}}\,\dt\\
& \ll_m \frac{x^{\frac{2}{3}(1-\alpha)}}{(1-\alpha)} + \dfrac{x^{1-\alpha}}{(1-\alpha)(\log x)^{2m+3}}\\
& \ll_{m,c} \dfrac{x^{1-\alpha}}{(1-\alpha)(\log x)^{2m+3}}.
\end{split}
\end{align}
The desired inequality follows by combining \eqref{Feb10_1:22pm}, \eqref{Feb10_1:23pm}, and \eqref{Feb10_1:24pm}.
\end{proof}

\subsection*{A.2} {\it Let $c>0$ be a given real number and $m\geq 0$ and $k\geq 1$ be given integers. For $\hh \leq \alpha < 1$ and $x \geq 3$ such that $(1-\alpha)^{2}\log x \geq c$, we have }
		\begin{align*}
		\displaystyle\int_{2}^{x}\dfrac{1}{t^{\alpha}(k\log x +\log t)^{2m+2}}\,\dt & = 
		\dfrac{x^{1-\alpha}}{(1-\alpha)((k+1)\log x)^{2m+2}}-\dfrac{2^{1-\alpha}}{(1-\alpha)(k\log x+\log 2)^{2m+2}} \\
		& \ \ \ \  \   + O_{m,c}\bigg(\dfrac{x^{1-\alpha}}{(1-\alpha)^{2}((k+1)\log x)^{2m+3}}\bigg).
		\end{align*}
\begin{proof}
	Using the change of variables $y=x^kt$ and \textbf{A.1} we obtain 
	\begin{align*} \displaystyle\int_{2}^{x}\dfrac{1}{t^{\alpha}(k\log x +\log t)^{2m+2}}\,\d t & = x^{-k+k\alpha}\displaystyle\int_{2x^k}^{x^{k+1}}\dfrac{1}{y^{\alpha}(\log y)^{2m+2}}\,\dy \\
	& = x^{-k+k\alpha}\Bigg[\displaystyle\int_{2}^{x^{k+1}}\dfrac{1}{y^{\alpha}(\log y)^{2m+2}}\,\dy-\displaystyle\int_{2}^{2x^k}\dfrac{1}{y^{\alpha}(\log y)^{2m+2}}\,\dy\Bigg] \\
	& =  x^{-k+k\alpha}\Bigg[\dfrac{(x^{k+1})^{1-\alpha}}{(1-\alpha)(\log x^{(k+1)})^{2m+2}}+ O_{m,c}\bigg(\dfrac{(x^{k+1})^{1-\alpha}}{(1-\alpha)^{2}(\log x^{(k+1)})^{2m+3}}\bigg) \\
	& \ \ \ \ \ \ \ \ \ \ \ \ \ \ \  - \dfrac{(2x^{k})^{1-\alpha}}{(1-\alpha)(\log(2x^k))^{2m+2}}+ O_{m,c}\bigg(\dfrac{(2x^{k})^{1-\alpha}}{(1-\alpha)^{2}(\log(2x^k))^{2m+3}}\bigg)\Bigg]  \\
	& = \dfrac{x^{1-\alpha}}{(1-\alpha)((k+1)\log x)^{2m+2}}-\dfrac{2^{1-\alpha}}{(1-\alpha)(k\log x+\log 2)^{2m+2}} \\
	& \ \ + O_{m,c}\bigg(\!\dfrac{x^{1-\alpha}}{(1-\alpha)^{2}((k+1)\log x)^{2m+3}}\!\bigg) \!+ O_{m,c}\bigg(\!\dfrac{1}{(1-\alpha)^{2}(k\log x+\log 2)^{2m+3}}\!\bigg). 
	\end{align*}
Since 
\begin{align*}
	\dfrac{1}{(1-\alpha)^{2}(k\log x+\log 2)^{2m+3}}\leq \dfrac{2^{2m+3}}{(1-\alpha)^{2}((k+1)\log x)^{2m+3}} \ll_m \dfrac{x^{1-\alpha}}{(1-\alpha)^{2}((k+1)\log x)^{2m+3}},  
\end{align*}
we obtain the desired result.
\end{proof}

\subsection*{A.3} {\it Let $m\geq 0$ and $k\geq 0$ be given integers. For $\hh \leq \alpha < 1$ and $x \geq 3$ we have }

\begin{align*} \displaystyle\int_{2}^{x}\dfrac{1}{t^{1-\alpha}((k+2)\log x-\log t)^{2m+2}}\,\dt & = \dfrac{x^\alpha}{\alpha((k+1)\log x)^{2m+2}}-\dfrac{2^\alpha}{\alpha((k+2)\log x-\log 2)^{2m+2}} \\
 & \ \ \ \ + O_m\bigg(\dfrac{x^\alpha}{((k+1)\log x)^{2m+3}}\bigg).
\end{align*}

\begin{proof}
	Let $y=\frac{x^{k+2}}{t}$. The integral becomes
	\begin{align*}
	x^{(k+2)\alpha}\displaystyle\int_{x^{k+1}}^{\frac{x^{k+2}}{2}}\dfrac{1}{y^{1+\alpha}(\log y)^{2m+2}}\,\d y  & = \dfrac{x^\alpha}{\alpha((k+1)\log x)^{2m+2}}-\dfrac{2^{\alpha}}{\alpha((k+2)\log x-\log 2)^{2m+2}} \\
	& \ \ \ - \dfrac{(2m+2)x^{(k+2)\alpha}}{\alpha}\displaystyle\int_{x^{k+1}}^{\frac{x^{k+2}}{2}}\dfrac{1}{y^{1+\alpha}(\log y)^{2m+3}}\,\d y\,,
	\end{align*}
where we have used integration by parts. Finally, the result follows from the fact that
	\[
	\displaystyle\int_{x^{k+1}}^{\frac{x^{k+2}}{2}}\dfrac{1}{y^{1+\alpha}(\log y)^{2m+3}}\d y \ll \dfrac{1}{((k+1)\log x)^{2m+3}}\displaystyle\int_{x^{k+1}}^{\frac{x^{k+2}}{2}}\dfrac{1}{y^{1+\alpha}}\,\d y\ll\dfrac{1}{x^{(k+1)\alpha}((k+1)\log x)^{2m+3}}.
	\]
\end{proof}

\subsection*{A.4} {\it For $\hh < \alpha < 1$} and $x \geq 3$ we have 
\begin{equation*}
\displaystyle\sum_{k=1}^{\infty}\dfrac{1}{\big(x^{\alpha-\frac{1}{2}}\big)^k}\leq \dfrac{1}{(\alpha-\frac{1}{2})\log x}.
\end{equation*}

\begin{proof}
Using the mean value theorem we have that
\begin{align*}
\displaystyle\sum_{k=1}^{\infty}\dfrac{1}{\big(x^{\alpha-\frac{1}{2}}\big)^k}=\dfrac{1}{x^{\alpha-\frac{1}{2}}-1}=\dfrac{1}{(\alpha-\frac{1}{2})x^{\xi}\log x}\leq \dfrac{1}{(\alpha-\frac{1}{2})\log x},
\end{align*}
where $\xi$ is a point in the interval $(0,\alpha-\frac{1}{2})$.
\end{proof}

\subsection*{A.5} {\it Let $m\geq 0$ be a given integer. For $\hh \leq \alpha < 1$ and $x \geq 3$ we have }
\begin{align*}
\displaystyle\sum_{k=1}^{\infty}&\dfrac{k+1}{\big(x^{\alpha-\frac{1}{2}}\big)^k} \Bigg|\dfrac{2^\alpha}{x^{2\alpha-1}\,((k+2)\log x-\log 2)^{2m+2}}-\dfrac{2^{1-\alpha}}{(k\log x+\log 2)^{2m+2}}\Bigg| \ll_m \dfrac{x^{1-\alpha}}{(1-\alpha)^2(\log x)^{2m+3}}
\end{align*}
{\it and}
\begin{align*}
\displaystyle\sum_{k=1}^{\infty}&\dfrac{k+1}{\big(x^{\alpha-\frac{1}{2}}\big)^k} \Bigg|\dfrac{2^\alpha}{x^{2\alpha-1}\alpha((k+2)\log x-\log 2)^{2m+2}}-\dfrac{2^{1-\alpha}}{(1-\alpha)(k\log x+\log 2)^{2m+2}}\Bigg| \ll_m \dfrac{x^{1-\alpha}}{(1-\alpha)^2(\log x)^{2m+3}}.
\end{align*}

\begin{proof} Using the mean value theorem for the functions $y\mapsto y^{2m+2}$ and $y\mapsto 2^{\alpha-y}x^y$ we obtain, for $k \geq 1$, that
\begin{align*}
&\left|\dfrac{2^\alpha}{x^{2\alpha-1}((k+2)\log x-\log 2)^{2m+2}}-\dfrac{2^{1-\alpha}}{(k\log x+\log 2)^{2m+2}}\right| \\
& \leq  \left|\dfrac{2^{\alpha}}{x^{2\alpha-1}}\bigg(\dfrac{1}{((k+2)\log x-\log 2)^{2m+2}}-\dfrac{1}{(k\log x+\log 2)^{2m+2}}\bigg)\right| + \bigg|\dfrac{1}{(k\log x+\log 2)^{2m+2}}\bigg(\dfrac{2^{\alpha}}{x^{2\alpha-1}}-2^{1-\alpha}\bigg)\bigg|\\
& =  \dfrac{2^{\alpha}}{x^{2\alpha-1}}\bigg(\dfrac{((k+2)\log x-\log 2)^{2m+2}-(k\log x+\log 2)^{2m+2}}{((k+2)\log x-\log 2)^{2m+2}(k\log x+\log 2)^{2m+2}}\bigg) + \dfrac{1}{(k\log x+\log 2)^{2m+2}}\bigg(\dfrac{2^{1-\alpha}x^{2\alpha-1}-2^\alpha}{x^{2\alpha-1}}\bigg) \\
& \leq \dfrac{2^{\alpha}}{x^{2\alpha-1}}\bigg(\dfrac{2(2m+2)(\log x-\log 2)((k+2)\log x-\log 2)^{2m+1}}{((k+2)\log x-\log 2)^{2m+2}(k\log x+\log 2)^{2m+2}}\bigg) + \dfrac{(2\alpha-1)2^{1-\alpha}(\log x - \log 2)}{(k\log x+\log 2)^{2m+2}} \\
& \ll_m \dfrac{1}{x^{2\alpha-1}(k+1)^{2m+3}(\log x)^{2m+2}} + \dfrac{(2\alpha-1)}{(k+1)^{2m+2}(\log x)^{2m+1}} .
\end{align*}
Therefore, summing over all $k\geq 1$ and using \textbf{A.4}, we arrive at
\begin{align} \label{bestbound}
\begin{split}
\displaystyle\sum_{k=1}^{\infty}&\dfrac{k+1}{\big(x^{\alpha-\frac{1}{2}}\big)^k} \Bigg|\dfrac{2^\alpha}{x^{2\alpha-1}((k+2)\log x-\log 2)^{2m+2}}-\dfrac{2^{1-\alpha}}{(k\log x+\log 2)^{2m+2}}\Bigg| \\
& \ \ \ \ \ \ \ \ \ \ \ll_m  \displaystyle\sum_{k=1}^{\infty}\dfrac{k+1}{\big(x^{\alpha-\frac{1}{2}}\big)^k}\left(\dfrac{1}{x^{2\alpha-1}(k+1)^{2m+3}(\log x)^{2m+2}} + \dfrac{(2\alpha-1)}{(k+1)^{2m+2}(\log x)^{2m+1}}\right)  \\
& \ \ \ \ \ \ \ \ \ \ \leq \dfrac{1}{x^{2\alpha-1}(\log x)^{2m+2}} \displaystyle\sum_{k=1}^{\infty}\dfrac{1}{\big(x^{\alpha-\frac{1}{2}}\big)^k(k+1)^{2m+2}} + \dfrac{2\alpha-1}{(\log x)^{2m+1}}\displaystyle\sum_{k=1}^{\infty}\dfrac{1}{\big(x^{\alpha-\frac{1}{2}}\big)^k} \\ 
& \ \ \ \ \ \ \ \ \ \ \ll \dfrac{1}{(\log x)^{2m+2}} \ll \dfrac{x^{1-\alpha}}{(1-\alpha)(\log x)^{2m+3}} \ll \dfrac{x^{1-\alpha}}{(1-\alpha)^2(\log x)^{2m+3}},
\end{split} 
\end{align}
which establishes our first proposed estimate. To prove the second, we use the first one and \textbf{A.4} as follows
\begin{align*}
\displaystyle\sum_{k=1}^{\infty}&\dfrac{k+1}{\big(x^{\alpha-\frac{1}{2}}\big)^k} \Bigg|\dfrac{2^\alpha}{x^{2\alpha-1}\,\alpha((k+2)\log x-\log 2)^{2m+2}}-\dfrac{2^{1-\alpha}}{(1-\alpha)(k\log x+\log 2)^{2m+2}}\Bigg|  \\
& \leq  \dfrac{1}{\alpha}\displaystyle\sum_{k=1}^{\infty}\dfrac{k+1}{\big(x^{\alpha-\frac{1}{2}}\big)^k} \Bigg|\dfrac{2^\alpha}{x^{2\alpha-1}((k+2)\log x-\log 2)^{2m+2}}-\dfrac{2^{1-\alpha}}{(k\log x+\log 2)^{2m+2}}\Bigg|  \\
& \ \ \ \ + \displaystyle\sum_{k=1}^{\infty}\dfrac{k+1}{\big(x^{\alpha-\frac{1}{2}}\big)^k} \Bigg|\dfrac{2^{1-\alpha}}
{\alpha(k\log x+\log 2)^{2m+2}}-\dfrac{2^{1-\alpha}}
{(1-\alpha)(k\log x+\log 2)^{2m+2}}\Bigg|  \\
& \ll_m \dfrac{x^{1-\alpha}}{(1-\alpha)^2(\log x)^{2m+3}} + \dfrac{2\alpha-1}{\alpha(1-\alpha)} \displaystyle\sum_{k=1}^{\infty}\dfrac{k+1}{\big(x^{\alpha-\frac{1}{2}}\big)^k} \left(\dfrac{2^{1-\alpha}}
{(k\log x+\log 2)^{2m+2}}\right)  \\
& \ll \dfrac{x^{1-\alpha}}{(1-\alpha)^2(\log x)^{2m+3}} + \dfrac{2\alpha-1}{\alpha(1-\alpha)(\log x)^{2m+2}} \displaystyle\sum_{k=1}^{\infty}\dfrac{1}{\big(x^{\alpha-\frac{1}{2}}\big)^k} \\
& \ll \dfrac{x^{1-\alpha}}{(1-\alpha)^2(\log x)^{2m+3}} + \dfrac{1}{(1-\alpha)(\log x)^{2m+3}}\\
&\ll \dfrac{x^{1-\alpha}}{(1-\alpha)^2(\log x)^{2m+3}}.
\end{align*}
\end{proof}


\section*{Appendix B - Number theory facts}
\subsection*{Prelude} Recall that, under the Riemann hypothesis, the prime number theorem takes the form (\cite[Section 13.1]{MV})
\begin{equation}\label{PNT_under_RH}
\displaystyle\sum_{n\leq x}\Lambda(n)=x+O\big(x^{\frac12}(\log x)^2\big).
\end{equation}
In what follows we shall use in integration by parts in multiple occasions. Let $\varepsilon >0$ be a small number and $f:\Omega \to \R$, where $\Omega = \{(x,y) \in \R^2; \ 2 \leq x < \infty \ ; \ 1\leq y \leq x + 2\varepsilon\}$, be a function such that $y \mapsto f(x,y)$ is continuously differentiable in $(1, x + \varepsilon )$, for all $x \in [2,\infty)$. Using \eqref{PNT_under_RH} we obtain
\begin{equation}\label{newpnt}
\displaystyle\sum_{n\leq x}\Lambda(n)f(x,n)=\displaystyle\int_{2}^{x}\,f(x,y)\,\d y + 2 f(x,2) + O\big(x^{\frac12}(\log x)^2 \,|f(x,x)|\big)
+ O\bigg(\!\displaystyle\int_{2}^{x}y^{\frac12}(\log y)^2\, \left|\dfrac{\partial}{\partial y}f(x,y)\right|\,\d y\bigg). 
\end{equation}
We now proceed with the number theory facts required for our analysis. We assume the Riemann hypothesis in all the statements below.

\subsection*{B.1} {\it Let $c>0$ be a given real number and $m\geq 0$ be a given integer. For $\hh \leq \alpha < 1$ and $x \geq 3$ such that $(1-\alpha)^{2}\log x \geq c$, we have }
\begin{equation*}
\displaystyle\sum_{n\leq x}\dfrac{\Lambda(n)}{n^\alpha(\log n)^{2m+2}} = \dfrac{x^{1-\alpha}}{(1-\alpha)(\log x)^{2m+2}}+ O_{m,c}\left(\dfrac{x^{1-\alpha}}{(1-\alpha)^{2}(\log x)^{2m+3}}\right).
\end{equation*}
\begin{proof}
Using \eqref{newpnt}, together with \textbf{A.1} and \eqref{Feb02_4:52pm}, we obtain
\begin{align*}
\begin{split}
\displaystyle\sum_{n\leq x}\dfrac{\Lambda(n)}{n^\alpha(\log n)^{2m+2}}&=\displaystyle\int_{2}^{x}\,\dfrac{1}{y^\alpha(\log y)^{2m+2}}\d y + \dfrac{2^{1-\alpha}}{(\log 2)^{2m+2}} + O\bigg(\dfrac{x^{\frac12}(\log x)^2}{x^\alpha(\log x)^{2m+2}}\bigg) \\
& \ \ \  \ \ \ \ \  + O\left(\displaystyle\int_{2}^{x}y^{\frac12}(\log y)^2\, \left|\dfrac{\partial}{\partial y}\bigg[\dfrac{1}{y^\alpha(\log y)^{2m+2}}\bigg]\right|\d y\right) \\
& = \dfrac{x^{1-\alpha}}{(1-\alpha)(\log x)^{2m+2}}+ O_{m,c}\left(\dfrac{x^{1-\alpha}}{(1-\alpha)^{2}(\log x)^{2m+3}}\right) + O_m\left(\displaystyle\int_{2}^{x}\dfrac{1}{y^{\alpha+\frac12}}\d y\right).
\end{split}
\end{align*}
We now analyze the last term. From \eqref{Feb02_4:52pm} we have
\begin{equation}\label{Feb02_5:27pm}
\displaystyle\int_{2}^{x}\dfrac{1}{y^{\alpha+\frac12}}\dy \leq \displaystyle\int_{2}^{x}\dfrac{1}{y}\,\dy \leq \log x\ll_{m,c} \frac{x^{1-\alpha}}{(1-\alpha)^{2}(\log x)^{2m+3}}\,,
\end{equation}
and this concludes the proof.
\end{proof}

\subsection*{B.2} {\it Let $c>0$ be a given real number and $m\geq 0$ be a given integer. For $\hh \leq \alpha < 1$ and $x \geq 3$ such that $(1-\alpha)^{2}\log x \geq c$, we have }\begin{equation*}
\dfrac{1}{x^{2\alpha-1}}\displaystyle\sum_{n\leq x}\dfrac{\Lambda(n)}{n^{1-\alpha}(2\log x-\log n)^{2m+2}} = \dfrac{x^{1-\alpha}}{\alpha(\log x)^{2m+2}}+ O_{m,c}\left(\dfrac{x^{1-\alpha}}{(1-\alpha)^{2}(\log x)^{2m+3}}\right).
\end{equation*}
\begin{proof}
	Using \eqref{newpnt} together with \textbf{A.3}, we have 	
	\begin{align} \label{sumprin2}
	& \dfrac{1}{x^{2\alpha-1}}\displaystyle\sum_{n\leq x}\dfrac{\Lambda(n)}{n^{1-\alpha}(2\log x-\log n)^{2m+2}} \nonumber \\
	&  = \dfrac{1}{x^{2\alpha-1}}\displaystyle\int_{2}^{x}\,\dfrac{1}{y^{1-\alpha}(2\log x-\log y)^{2m+2}}\,\d y + \dfrac{2^{\alpha}}{x^{2\alpha-1}\,(2\log x-\log 2)^{2m+2}}  + O\left(\dfrac{1}{x^{\alpha-\frac12} (\log x)^{2m}}\right) \nonumber \\
	& \ \ \ \ \ \ \ \ + O\left(\dfrac{1}{x^{2\alpha-1}}\displaystyle\int_{2}^{x}y^{\frac12}(\log y)^2\, \left|\dfrac{\partial}{\partial y}\bigg[\dfrac{1}{y^{1-\alpha}(2\log x-\log y)^{2m+2}}\bigg]\right|\d y\right) \\
	& = \dfrac{x^{1-\alpha}}{\alpha(\log x)^{2m+2}}+ O_m\left(\dfrac{x^{1-\alpha}}{(\log x)^{2m+3}}\right) + O(1) \nonumber \\
	& \ \ \ \ \ \ \ \ + O\left(\dfrac{1}{x^{2\alpha-1}}\displaystyle\int_{2}^{x}y^{\frac12}(\log y)^2\, \left|\dfrac{\partial}{\partial y}\bigg[\dfrac{1}{y^{1-\alpha}(2\log x-\log y)^{2m+2}}\bigg]\right|\d y\right).\nonumber 
	\end{align}		
	We further analyze the last term
	\begin{align*}
	\begin{split}
 \dfrac{1}{x^{2\alpha-1}}\displaystyle\int_{2}^{x}y^{\frac12}(\log y)^2\, \left|\dfrac{\partial}{\partial y}\bigg[\dfrac{1}{y^{1-\alpha}(2\log x-\log y)^{2m+2}}\bigg]\right|\d y& \ll_m \displaystyle\int_{2}^{x}\frac{(\log y)^2}{x^{2\alpha-1}\,y^{\frac32-\alpha}(2\log x-\log y)^{2m+2}}\,\d y \\
	& \leq \displaystyle\int_{2}^{x}\frac{(\log y)^2}{y^{2\alpha-1}\,y^{\frac32-\alpha}(2\log x-\log y)^{2m+2}}\,\d y  \\
	& \leq \displaystyle\int_{2}^{x}\dfrac{1}{y^{\alpha+\frac12}}\,\d y.
	\end{split}
	\end{align*}
	Therefore, using \eqref{Feb02_4:52pm} and \eqref{Feb02_5:27pm} in \eqref{sumprin2} we obtain the result.
\end{proof}

\subsection*{B.3} {\it Let $c>0$ be a given real number and $m\geq 0$ be a given integer. For $\hh \leq \alpha < 1$ and $x \geq 3$ such that $(1-\alpha)^{2}\log x \geq c$, we have }
\begin{align*}
\displaystyle\sum_{k=1}^{\infty}\dfrac{k+1}{\big(x^{\alpha-\frac{1}{2}}\big)^k}& \Bigg|\displaystyle\sum_{n\leq x}\Lambda(n) \Bigg(\dfrac{1}{n^{\alpha}(k\log x+\log n)^{2m+2}}-\dfrac{1}{x^{2\alpha-1}\,n^{1-\alpha}((k+2)\log x-\log n)^{2m+2}}\Bigg)\Bigg| \\
& \ \ \ \ll_{m,c} \dfrac{x^{1-\alpha}}{(1-\alpha)^{2}(\log x)^{2m+3}}. 
\end{align*}

\begin{proof} Using \eqref{newpnt}, \textbf{A.2} and \textbf{A.3} we have, for any $k\geq 1$, 
\begin{align}\label{Feb13_1:10pm}
&\displaystyle\sum_{n\leq x}\Lambda(n) \Bigg(\dfrac{1}{n^{\alpha}(k\log x+\log n)^{2m+2}}-\dfrac{1}{x^{2\alpha-1}\,n^{1-\alpha}((k+2)\log x-\log n)^{2m+2}}\Bigg) \nonumber \\
&=\int_{2}^{x}\Bigg(\dfrac{1}{y^{\alpha}(k\log x+\log y)^{2m+2}}-\dfrac{1}{x^{2\alpha-1}\,y^{1-\alpha}((k+2)\log x-\log y)^{2m+2}}\Bigg)\d y \nonumber  \\
& \ \ \ + 2\Bigg(\dfrac{1}{2^{\alpha}(k\log x+\log 2)^{2m+2}}-\dfrac{1}{x^{2\alpha-1}\,2^{1-\alpha}((k+2)\log x-\log 2)^{2m+2}}\Bigg) \\
& \ \ \ + O\Bigg(\int_{2}^{x}y^{\frac12}(\log y)^2\left|\dfrac{\partial}{\partial y}\Bigg[\dfrac{1}{y^{\alpha}(k\log x+\log y)^{2m+2}}-\dfrac{1}{x^{2\alpha-1}\,y^{1-\alpha}((k+2)\log x-\log y)^{2m+2}}\Bigg]\right|\dy\Bigg) \nonumber \\
&= \dfrac{2\alpha-1}{\alpha(1-\alpha)}\dfrac{x^{1-\alpha}}{((k+1)\log x)^{2m+2}} + \left(\dfrac{2^\alpha}{x^{2\alpha-1}\,\alpha((k+2)\log x-\log 2)^{2m+2}}-\dfrac{2^{1-\alpha}}{(1-\alpha)(k\log x+\log 2)^{2m+2}}\right) \nonumber \\
& \ \ \ + \left(\dfrac{2^{1-\alpha}}{(k\log x+\log 2)^{2m+2}}-\dfrac{2^{\alpha}}{x^{2\alpha-1}\,((k+2)\log x-\log 2)^{2m+2}}\right)  + O_{m,c}\bigg(\dfrac{x^{1-\alpha}}{(1-\alpha)^{2}((k+1)\log x)^{2m+3}}\bigg) \nonumber \\ 
& \ \ \ + O\Bigg(\int_{2}^{x}y^{\frac12}(\log y)^2\left|\dfrac{\partial}{\partial y}\Bigg[\dfrac{1}{y^{\alpha}(k\log x+\log y)^{2m+2}}-\dfrac{1}{x^{2\alpha-1}\,y^{1-\alpha}((k+2)\log x-\log y)^{2m+2}}\Bigg]\right|\dy\Bigg).\nonumber 
\end{align} 
We now sum over $k\geq 1$ and analyze each term that appears in \eqref{Feb13_1:10pm}. \\

\noindent 1. {\it First term}: Using \textbf{A.4} we obtain 
	\begin{align*} 
	\begin{split}
	\displaystyle\sum_{k=1}^{\infty}\dfrac{k+1}{\big(x^{\alpha-\frac{1}{2}}\big)^k}\left(\dfrac{2\alpha-1}{\alpha(1-\alpha)}\dfrac{x^{1-\alpha}}{((k+1)\log x)^{2m+2}}\right) &\leq \dfrac{2\alpha-1}{\alpha(1-\alpha)}\dfrac{x^{1-\alpha}}{(\log x)^{2m+2}}\displaystyle\sum_{k=1}^{\infty}\dfrac{1}{\big(x^{\alpha-\frac{1}{2}}\big)^k} \\
	&  \ll  \dfrac{x^{1-\alpha}}{(1-\alpha)^{2}(\log x)^{2m+3}}.
	\end{split}
	\end{align*}

\noindent 2. {\it Second and third terms}: Using \textbf{A.5} we obtain  
\begin{align*}
	& \displaystyle\sum_{k=1}^{\infty}\dfrac{k+1}{\big(x^{\alpha-\frac{1}{2}}\big)^k} \Bigg|\dfrac{2^\alpha}{x^{2\alpha-1}\,\alpha((k+2)\log x-\log 2)^{2m+2}}-\dfrac{2^{1-\alpha}}{(1-\alpha)(k\log x+\log 2)^{2m+2}}\Bigg| \\
	& + 	\displaystyle\sum_{k=1}^{\infty}\dfrac{k+1}{\big(x^{\alpha-\frac{1}{2}}\big)^k} \Bigg|\dfrac{2^{1-\alpha}}{(k\log x+\log 2)^{2m+2}}-\dfrac{2^{\alpha}}{x^{2\alpha-1}((k+2)\log x-\log 2)^{2m+2}}\Bigg|\ll_m  \dfrac{x^{1-\alpha}}{(1-\alpha)^{2}(\log x)^{2m+3}}. 
	\end{align*}

\noindent 3. {\it Fourth term}:
\begin{align*}
	\displaystyle\sum_{k=1}^{\infty}\dfrac{k+1}{\big(x^{\alpha-\frac{1}{2}}\big)^k}\dfrac{x^{1-\alpha}}{(1-\alpha)^{2}((k+1)\log x)^{2m+3}}\ll \dfrac{x^{1-\alpha}}{(1-\alpha)^2(\log x)^{2m+3}}.
	\end{align*}
	
\noindent 4. {\it Fifth term}: Using \textbf{A.4} again we have
\begin{align} 
 \label{troy}
\displaystyle\sum_{k=1}^{\infty}&\dfrac{k+1}{\big(x^{\alpha-\frac{1}{2}}\big)^k}\int_{2}^{x}y^{\frac{1}{2}}(\log y)^2\left|\dfrac{\partial}{\partial y}\Bigg[\dfrac{1}{y^{\alpha}(k\log x+\log y)^{2m+2}}-\dfrac{1}{x^{2\alpha-1}\,y^{1-\alpha}((k+2)\log x-\log y)^{2m+2}}\Bigg]\right|\dy \nonumber \\
& =  \displaystyle\sum_{k=1}^{\infty}\dfrac{k+1}{\big(x^{\alpha-\frac{1}{2}}\big)^k} \int_{2}^{x}y^{\frac{1}{2}}(\log y)^2 \,\bigg| -\dfrac{2m+2}{y^{1+\alpha}(k\log x+\log y)^{2m+3}}-\dfrac{\alpha}{y^{1+\alpha}(k\log x+\log y)^{2m+2}} \nonumber  \\
	& \ \ \ -\dfrac{1}{x^{2\alpha-1}}
	\bigg(\dfrac{2m+2}{y^{2-\alpha}((k+2)\log x-\log y)^{2m+3}}-\dfrac{1-\alpha}{y^{2-\alpha}((k+2)\log x-\log y)^{2m+2}}\bigg)\bigg|\,\dy  \nonumber \\
	& \leq  \displaystyle\sum_{k=1}^{\infty}\dfrac{k+1}{\big(x^{\alpha-\frac{1}{2}}\big)^k}\int_{2}^{x}y^{\frac{1}{2}}(\log y)^2\bigg(\dfrac{2m+2}{y^{1+\alpha}(k\log x+\log y)^{2m+3}}+\dfrac{2m+2}{x^{2\alpha-1}\,y^{2-\alpha}((k+2)\log x-\log y)^{2m+3}}\bigg)\d y \nonumber \\
	& \ \ \ + \displaystyle\sum_{k=1}^{\infty}\dfrac{k+1}{\big(x^{\alpha-\frac{1}{2}}\big)^k}\int_{2}^{x}y^{\frac{1}{2}}(\log y)^2\bigg(\dfrac{\alpha}{y^{1+\alpha}(k\log x+\log y)^{2m+2}}-\dfrac{1-\alpha}{x^{2\alpha-1}\,y^{2-\alpha}((k+2)\log x-\log y)^{2m+2}}\bigg)\d y \nonumber \\
	& \leq 
	\displaystyle\sum_{k=1}^{\infty}\dfrac{k+1}{\big(x^{\alpha-\frac{1}{2}}\big)^k}\int_{2}^{x}y^{\frac{1}{2}}(\log y)^2\bigg(\dfrac{4m+4}{y^{1+\alpha}(k\log x+\log y)^{2m+3}}\bigg)\d y  \nonumber \\
    & \ \ \ + \displaystyle\sum_{k=1}^{\infty}\dfrac{k+1}{\big(x^{\alpha-\frac{1}{2}}\big)^k}\int_{2}^{x}y^{\frac{1}{2}}(\log y)^2\bigg(\dfrac{2\alpha-1}{y^{1+\alpha}(k\log x+\log y)^{2m+2}}\bigg)\d y \\
	& \ \ \ + \displaystyle\sum_{k=1}^{\infty}\dfrac{k+1}{\big(x^{\alpha-\frac{1}{2}}\big)^k}\int_{2}^{x}y^{\frac{1}{2}}(\log y)^2\bigg(\dfrac{1-\alpha}{y^{1+\alpha}(k\log x+\log y)^{2m+2}}-\dfrac{1-\alpha}{x^{2\alpha-1}\,y^{2-\alpha}((k+2)\log x-\log y)^{2m+2}}\bigg)\d y  \nonumber \\
	& \leq    \displaystyle\sum_{k=1}^{\infty}\dfrac{k+1}{\big(x^{\alpha-\frac{1}{2}}\big)^k}\int_{2}^{x}y^{\frac{1}{2}}(\log y)^2\bigg(\dfrac{4m+4}{y^{1+\alpha}((k+1)\log y)^{2m+3}}\bigg)\d y \nonumber \\
	& \ \ \ + \displaystyle\sum_{k=1}^{\infty}\dfrac{k+1}{\big(x^{\alpha-\frac{1}{2}}\big)^k}\int_{2}^{x}y^{\frac{1}{2}}(\log y)^2\bigg(\dfrac{2\alpha-1}{y^{1+\alpha}((k+1)\log y)^{2m+2}}\bigg)\d y \nonumber \\
	& \ \ \ + (1-\alpha)\displaystyle\sum_{k=1}^{\infty}\dfrac{k+1}{\big(x^{\alpha-\frac{1}{2}}\big)^k}\int_{2}^{x}\dfrac{(\log y)^2}{y^{\frac{1}{2}}}\bigg(\dfrac{1}{y^{\alpha}(k\log x+\log y)^{2m+2}}-\dfrac{1}{x^{2\alpha-1}\,y^{1-\alpha}((k+2)\log x-\log y)^{2m+2}}\bigg)\d y  \nonumber \\
	& \ll_m \int_{2}^{x}\dfrac{1}{y^{\alpha+\frac12}}\,\d y + (2\alpha-1)\Bigg( \int_{2}^{x}\dfrac{1}{y^{\alpha+\frac{1}{2}}}\,\d y\Bigg)\displaystyle\sum_{k=1}^{\infty}\dfrac{1}{\big(x^{\alpha-\frac{1}{2}}\big)^k} \nonumber \\
	& \ \ \ + \displaystyle\sum_{k=1}^{\infty}\dfrac{k+1}{\big(x^{\alpha-\frac{1}{2}}\big)^k}\int_{2}^{x}\bigg(\dfrac{1}{y^{\alpha}(k\log x+\log y)^{2m+2}}-\dfrac{1}{x^{2\alpha-1}\,y^{1-\alpha}((k+2)\log x-\log y)^{2m+2}}\bigg)\d y \nonumber \\	
	& \ll \int_{2}^{x}\dfrac{1}{y^{\alpha+\frac12}}\d y +  \displaystyle\sum_{k=1}^{\infty}\dfrac{k+1}{\big(x^{\alpha-\frac{1}{2}}\big)^k}\int_{2}^{x}\bigg(\dfrac{1}{y^{\alpha}(k\log x+\log y)^{2m+2}}-\dfrac{1}{x^{2\alpha-1}\,y^{1-\alpha}((k+2)\log x-\log y)^{2m+2}}\bigg)\d y\nonumber 
	\end{align}
	We can see that the last sum already appeared in our analysis,  in the first, second and fourth terms treated above. Therefore, an application of \eqref{Feb02_5:27pm} in \eqref{troy} concludes the proof. 
\end{proof}
   
\subsection*{B.4} {\it For $0 \leq \beta < \hh$ and $x \geq 3$, we have }
	\begin{equation*}
	\displaystyle\sum_{n\leq x}\dfrac{\Lambda(n)}{n^{1/2}}\left(\dfrac{x^{\beta}}{n^{\beta}}-\dfrac{n^\beta}{x^{\beta}}\right)= \dfrac{2\beta x^{1/2} - 2^{\frac{1}{2}-\beta}x^\beta \big(\frac12 + \beta\big)^2 + 2^{\frac{1}{2}+\beta}x^{-\beta} \big(\frac12 - \beta\big)^2}{\frac{1}{4}-\beta^2}+ O\left(\beta\, x^{\beta}\,(\log x)^4\right)
	\end{equation*}	
\begin{proof}Using \eqref{newpnt} we have that 
	\begin{align} \label{enfermo}
	\displaystyle\sum_{n\leq x}\dfrac{\Lambda(n)}{n^{1/2}}\bigg(\dfrac{x^{\beta}}{n^{\beta}}-\dfrac{n^\beta}{x^{\beta}}\bigg)& = \int_{2}^{x}\left(\dfrac{x^\beta}{y^{\beta+\frac{1}{2}}}-\dfrac{x^{-\beta}}{y^{\frac{1}{2}-\beta}}\right)\,\dy +2^{\frac{1}{2}-\beta}x^\beta - 2^{\frac{1}{2}+\beta}x^{-\beta} \nonumber \\
	& \ \ \ \ \ \ \ \ + O\left(\int_{2}^{x}\left|\dfrac{-(\frac{1}{2}+\beta)x^\beta}{y^{\beta+\frac{3}{2}}}-\dfrac{(\beta-\frac{1}{2})x^{-\beta}}{y^{\frac{3}{2}-\beta}}\right|y^{1/2\,}(\log y)^2\,\d y\right) \\
	& = \dfrac{x^{\frac{1}{2}}}{\frac{1}{2}-\beta}-\dfrac{2^{\frac{1}{2}-\beta}x^{\beta}}{\frac{1}{2}-\beta}-\dfrac{x^{\frac{1}{2}}}{\frac{1}{2}+\beta}+\dfrac{2^{\beta+\frac{1}{2}}x^{-\beta}}{\frac{1}{2}+\beta} +2^{\frac{1}{2}-\beta}x^\beta - 2^{\frac{1}{2}+\beta}x^{-\beta} \nonumber \\
	& \ \ \ \ \ \ \ \ + O\left(\int_{2}^{x}\left(\dfrac{(\frac{1}{2}+\beta)x^\beta}{y^{1+\beta}}-\dfrac{(\frac{1}{2}-\beta)x^{-\beta}}{y^{1-\beta}}\right)(\log y)^2\,\d y\right). \nonumber 
	\end{align}
Using the mean value theorem for the function $t \mapsto (\hh+t)x^t$ we find
		\begin{align}\label{Fev01_2:58pm}
		\begin{split}
		\int_{2}^{x}\left(\dfrac{(\frac{1}{2}+\beta)x^\beta}{y^{1+\beta}}-\dfrac{(\frac{1}{2}-\beta)x^{-\beta}}{y^{1-\beta}}\right)(\log y)^2\,\d y& \leq \int_{2}^{x}\left(\dfrac{(\frac{1}{2}+\beta)x^\beta}{y}-\dfrac{(\frac{1}{2}-\beta)x^{-\beta}}{y}\right)(\log y)^2\,\d y  \\
		&	\ll \Big[\big(\hh+\beta\big)x^{\beta}-\big(\hh-\beta\big)x^{-\beta}\Big](\log x)^3  \\
		&  \ll  \beta x^{\beta}(\log x)^4.
		\end{split}
		\end{align}
The desired estimate follows from \eqref{enfermo} and \eqref{Fev01_2:58pm}.
\end{proof}

\section*{Acknowledgements}
E.C. acknowledges support from CNPq-Brazil grant $305612/2014-0$, FAPERJ grant $E-26/103.010/2012$ and the Fulbright Junior Faculty Award. A.C. acknowledges support from CNPq-Brazil. M.B.M. acknowledges support from the NSA Young Investigator Grants H98230-15-1-0231 and H98230-16-1-0311. We are thankful to Vorrapan Chandee for helpful discussions related to the material of this paper. We are also thankful to the anonymous referee for the valuable comments and suggestions.

\end{document}